\documentclass[a4paper]{amsart}
\usepackage{graphicx}
\usepackage{amsmath}
\usepackage{amssymb}
\usepackage{amsfonts}
\usepackage{amsthm}
\usepackage[all, cmtip]{xy}
\usepackage{eucal}
\usepackage{tikz}
\usepackage{hyperref}
\usepackage{mathtools}
\usepackage{accents}

\usetikzlibrary{trees, positioning, calc}

\newcommand{\Sets}{\mathbf{Sets}}
\newcommand{\sSets}{\mathbf{sSets}}
\newcommand{\dSets}{\mathbf{dSets}}
\newcommand{\odSets}{\mathbf{odSets}}
\newcommand{\cdSets}{\mathbf{cdSets}}
\newcommand{\dSpaces}{\mathbf{dSpaces}}
\newcommand{\odSpaces}{\mathbf{odSpaces}}
\newcommand{\cdSpaces}{\mathbf{cdSpaces}}
\newcommand{\Cat}{\mathbf{Categories}}
\newcommand{\Oper}{\mathbf{Operads}}
\newcommand{\oOper}{\mathbf{oOperads}}
\newcommand{\cOper}{\mathbf{cOperads}}

\DeclareMathOperator{\Int}{Int}
\DeclareMathOperator{\Aut}{Aut}
\DeclareMathOperator{\cSpine}{cSpine}
\DeclareMathOperator{\hhom}{\underline{hom}}
\DeclareMathOperator{\Hom}{Hom}
\DeclareMathOperator{\HHom}{\underline{Hom}}

\theoremstyle{plain}
\newtheorem{theorem}{Theorem}[section]

\newtheorem{lemma}[theorem]{Lemma}
\newtheorem{proposition}[theorem]{Proposition}

\newtheorem{corollary}[theorem]{Corollary}

\newtheorem*{proposition*}{Proposition}
\newtheorem*{corollary*}{Corollary}

\theoremstyle{definition}
\newtheorem{definition}[theorem]{Definition}
\newtheorem{example}[theorem]{Example}

\newtheorem{remark}[theorem]{Remark}

\makeatletter
\@addtoreset{theorem}{section}
\@addtoreset{subsection}{section}
\makeatother

\newdir{ >}{{}*!/-5pt/\dir{>}} 

\makeatletter
\newcommand{\colim@}[2]{%
	\vtop{\m@th\ialign{##\cr
			\hfil$#1\operator@font colim$\hfil\cr
			\noalign{\nointerlineskip\kern1.5\ex@}#2\cr
			\noalign{\nointerlineskip\kern-\ex@}\cr}}%
}
\newcommand{\colim}{%
	\mathop{\mathpalette\colim@{\rightarrowfill@\scriptscriptstyle}}\nmlimits@
}
\renewcommand{\varprojlim}{%
	\mathop{\mathpalette\varlim@{\leftarrowfill@\scriptscriptstyle}}\nmlimits@
}
\renewcommand{\varinjlim}{%
	\mathop{\mathpalette\varlim@{\rightarrowfill@\scriptscriptstyle}}\nmlimits@
}
\makeatother

\DeclareMathAlphabet{\mathpzc}{OT1}{pzc}{m}{it}

\begin{document}
	
	\title{Closed dendroidal sets and unital operads}
		
	\author[Ieke Moerdijk]{Ieke Moerdijk}
	\address{Utrecht Universiteit, The Netherlands}
	\email{i.moerdijk@uu.nl}

	\date{}
	
	\begin{abstract}
		We discuss a variant of the category of dendroidal sets, the so-called closed dendroidal sets which are indexed by trees without leaves. This category carries a Quillen model structure which behaves better than the one on general dendroidal sets, mainly because it satisfies the pushout-product property, hence induces a symmetric monoidal structure on its homotopy category. We also study complete Segal style model structures on closed dendroidal spaces, and various Quillen adjunctions to model categories on (all) dendroidal sets or spaces. As a consequence, we deduce a Quillen equivalence from closed dendroidal sets to the category of unital simplicial operads, as well as to that of simplicial operads which are unital up-to-homotopy. The proofs exhibit several new combinatorial features of categories of closed trees.
	\end{abstract}
	
	\maketitle

\section*{Introduction}

A simplicial or topological operad $P$ is called \emph{unital} \cite{may} if its space $P(0)$  of nullary operations or ``constants''  consists of a single point. More generally, one says that a coloured operad  is unital if this condition is satisfied for each colour. Many operads arising in algebraic topology, such as the various models for $\textbf{E}_n$ operads \cite{may, boardmanvogt} and their variations and extensions \cite{salv2, salv1} have this property. The goal of this paper is to describe a combinatorial model for these unital coloured operads by means of dendroidal sets.

Recall that the category $\dSets$ of dendroidal sets is the category of presheaves of sets on a category $\Omega$ of trees. This category is an extension of that of simplicial sets, which are presheaves of sets on the subcategory  $\Delta$ of $\Omega$ consisting of linear trees. Among the main results of \cite{cisinskimoerdijk1, cisinskimoerdijk2, cisinskimoerdijk3}, we recall that there is a so-called ``operadic'' Quillen model structure on this category $\dSets$, which is Quillen equivalent to a Dwyer-Kan style model structure on simplicial coloured operads in which the weak equivalences are the  morphisms of operads which are essentially surjective and fully faithful in the appropriate homotopical sense. The fibrant objects in this model structure are the operads all of whose spaces of operations are fibrant (i.e., Kan complexes), and for this reason we also refer to this model structure as the projective one. The  Quillen equivalence is given by an explicit Quillen pair, obtained by Kan extension of the functor from $\Omega$ to simplicial operads which views a tree as an operad and then takes its Boardman-Vogt resolution;  see loc.  cit. for details.

When a tree is viewed as a coloured operad, this operad is unital iff the tree has no leaves, i.e.  is closed. Therefore, it is natural to try and model unital operads by  ``closed" dendroidal sets, presheaves of the category of closed trees. This is indeed possible, and we will establish a theorem of the following form (see Theorems \ref{teo11-4}  and \ref{teo13-1} for a precise formulation).\\

\noindent \textbf{Theorem.} \emph{There are Quillen model structures on the categories of closed dendroidal sets and of unital simplicial coloured operads, respectively,  between which a suitable adaptation of the Boardman-Vogt resolution induces a Quillen equivalence.}

Motivated by this theorem, and because unitality of an operad has little to do with the units of the operad, we will also refer to unital operads as ``closed operads". These are to be contrasted with ``open operads" for which the space(s) of constants, one for each colour, are all empty. It is an easy consequence of the results of \cite{cisinskimoerdijk1, cisinskimoerdijk2, cisinskimoerdijk3} that the Quillen equivalence between all dendroidal sets and all simplicial operads restricts to one between open dendroidal sets (presheaves on open trees, trees without stumps) and open simplicial operads.

The theorem stated above  has been expected to hold since the start of the development of the theory of dendroidal sets (at least by the author of this paper). However,  its proof is more  involved than one would perhaps expect. Indeed, unlike the case of open dendroidal sets, the model structure on dendroidal sets does not restrict in an obvious way to closed dendroidal sets. Instead, it seems one has to construct a new model structure from scratch. The fibrant objects in this model category deserve the name ``unital infinity operads" .This model structure on closed dendroidal sets does have one important feature that it shares with open dendroidal sets (cf. \cite{HHM}) but which fails for general dendroidal sets; namely, it satisfies the pushout-product property for cofibrations and trivial cofibrations, hence induces a well-defined symmetric monoidal structure on its homotopy category (Corollary \ref{cor9-4} below).  For this model category of closed dendroidal sets, we will also describe a  Quillen equivalent  dendroidal complete Segal  version, a localization of either the projective or the Reedy model structure on closed dendroidal spaces, i.e. simplicial presheaves on the category of closed trees. These closed dendroidal spaces play a role, for example, in the study by Boavida and Weiss of mapping spaces between operads, cf. \cite{boavidaweiss}.

For this model structure on closed dendroidal sets, the Boardman-Vogt resolution is not left Quillen, nor does  it produce unital or closed operads. (It yields operads which are weakly unital, in the sense of having weakly contractible spaces of constants.) For this reason, we both have to modify the Boardman-Vogt resolution, and have to change the ``projective" model category structure on closed operads to a Reedy type model structure with more cofibrations. Once these model structures and the Quillen pair induced by a modified Boardman-Vogt resolution have been put in place, it is possible to deduce the fact that this Quillen pair is in fact a Quillen equivalence from the corresponding statement for open operads and open dendroidal sets, by passing through the various complete Segal space versions of these model categories.

Using these complete Segal model categories, it is also possible to show that the category of closed dendroidal sets is equivalent to a left Bousfield localization of the category of all dendroidal sets, see Corollary \ref{cor14-3} below for a precise formulation. This has as a consequence that the homotopy category of unital operads is equivalent to that of weakly unital operads, a fact that has been suspected to be true but for which I am unaware of a proof in the literature.\\

The plan of this paper, then, is as follows. In the first section, we review the definitions of the categories of coloured simplicial operads and their closed and open variants, the model structures these carry,  and the adjoint Quillen functors that exist between these. In Section \ref{sect3},  we introduce the model category structures on open and  closed operads with a fixed set of colours. For closed operads, these are the projective and Reedy style model structures, which can be viewed as induced by transfer from the standard model category structure on collections (i.e., simplicial presheaves on the category of finite sets and bijections),  and the generalized Reedy model structure \cite{bergermoerdijkReedy} on the category of simplicial presheaves on the category of finite sets and injections, respectively. The first model structure was already considered by \cite{caviglia} in the more general context of operads in a monoidal model category, and the second was already considered by \cite{fresse} in the more restrictive monochromatic case (operads with a single colour). In a third section, we extend these model structures to model structures on operads with arbitrary sets of colours and morphisms that can change colours, in which the Dwyer-Kan style weak equivalences only need to be surjective on colours up to equivalence

In Sections \ref{sect5}-\ref{sect9}, we discuss the category of closed dendroidal sets and its model structure. Compared to \cite{cisinskimoerdijk1, cisinskimoerdijk2, cisinskimoerdijk3}, the set of inner horns and its saturation of inner anodyne morphisms will have to be replaced by the set of ``very inner horns'' the saturation of which we will refer to as very inner anodyne (``via") morphisms. This model category has all the properties of a \emph{monoidal} model category, except for the fact that the monoidal structure is only associative up to via morphisms, see Theorem \ref{teo9-2} for a precise formulation.

In Section \ref{sect10}, we will compare this model structure on closed dendroidal sets to model structures on open dendroidal sets and arbitrary dendroidal sets through various Quillen pairs. These Quillen pairs will all play a role in the proof of the main theorem. The proofs in Sections \ref{sect5}-\ref{sect10} involve several rather delicate combinatorial arguments for closed trees which haven't occurred in the literature before; cf. Lemmas \ref{lem7-2} and \ref{lemma10-5}, for an example.

Towards the proof of the main theorem, we will introduce a modified Boardman-Vogt resolution in Section \ref{sect11} and prove that it yields a Quillen adjunction between closed dendroidal sets and unital or closed operads. Section \ref{sect12} then describes dendroidal complete Segal space versions of the model strcutures on (closed, open and arbitrary) dendroidal sets, and some special properties of the Quillen pairs of Section \ref{sect10}. With all these preparations out of the way, we can then prove our main theorem in Section \ref{sect13}.

In a final Section \ref{sect14}, we prove that the model category of closed dendroidal sets can be viewed as a localization of that of all dendroidal sets, and prove the equivalence to weakly unital operads.

For most of this paper, we expect the reader to be familiar with the theory of Quillen model structures and their left Bousfield localizations \cite{quillenHA, hirschhorn, hovey}, as well as with the basic theory of dendroidal sets. The most accessible source for the latter is probably the first part of the book \cite{HMbook}.

\textbf{Acknowledgments.} This paper obviously builds on earlier and ongoing 
\cite{HMbook} work and helpful discussions with Denis-Charles Cisinski, Gijs 
Heuts and Vladimir Hinich. In addition, I am grateful to Matija Ba\v si\' c for his careful reading of the manuscript.

\tableofcontents

\section{Unital simplicial operads}  \label{section2}

In this paper, the term \emph{operad} will always mean \emph{symmetric coloured simplicial operad}. For such an operad $P$, we write $P=(P,C)$, where $C$ is the set of ``colours'' or ``objects'' of the operad. For a sequence $c_1,\ldots, c_n, c$ of colours, the corresponding simplicial set of operations is denoted $P(c_1, \ldots, c_n; c)$. We refer to \cite{bergermoerdijk2} for a detailed definition and examples. An operad P is called \emph{open} if $P(c_1, \ldots, c_n; c)=\emptyset$ whenever $n=0$, and \emph{closed} or \emph{unital} if $P(c_1, \ldots, c_n; c)=\Delta[0]=\mathrm{pt}$ whenever $n=0$. We will refer to the unique element of $P(-; c)$ as \emph{the constant} of colour $c$. In the latter case, we will denote the category of coloured operads and its full subcategories of open and closed ones by 
\begin{equation} \label{eq1}
\Oper, \quad \oOper, \quad \cOper,
\end{equation}
respectively. For each of them, and for a fixed set $C$ of colours, the subcategories of operads with colours $C$ are denoted
\begin{equation} \label{eq2}
\Oper_C, \quad \oOper_C, \quad \cOper_C,
\end{equation}
respectively. The morphism in these subcategories are the identities on the colours, so these subcategories are not full. 
The categories in (\ref{eq1}) are related by several adjoint functors. The ones most relevant for us are the functors which we denote by 
\begin{equation} \label{eq3}
\xymatrix@C=40pt{
	\oOper \ar@/^1pc/[r]^{g_!} \ar@/_1pc/[r]_{g_*} & \cOper  \ar[l]_{g^*} 
}
\end{equation}
For a closed operad $Q$, the operad $g^*(Q)$ is defined by simply forgetting the nullary operations, i.e. the constants. Its left adjoint $g_!$ maps an open operad $P$ to its `closure' $g_!(P)$ which can most easily be defined in terms of the Boardman-Vogt tensor product (\cite{boardmanvogt}), as 
\begin{equation} \label{eq4}
g_!(P)=P\otimes \overline{\eta}.
\end{equation}
Here $\overline{\eta}$ is the closed operad with one colour, and the identity as only non-nullary operation. Another way to describe this closure is as the operad with the same set $C$ of colours as $P$, and operations defined by the formula
\begin{equation} \label{eq5}
g_!(P)(c_1,\ldots, c_n;c) = \varinjlim_A P(c_1, \ldots, c_n, A; c)  
\end{equation}
where $A$ ranges over finite sequences of colours, while the morphisms in the colimit diagram are the morphisms in the prop $P^\otimes$ associated to $P$. So, if $q=(q_b)_{b\in B} \colon A \to~B$ is such a morphism in $P^\otimes$ and $p\in P(c_1, \ldots, c_n, B; c)$ then in $g_!(P)(c_1, \ldots, c_n; c)$ the element $p$ is identified with $p\circ_B q \in P(c_1, \ldots, c_n, A; c)$. Note that for a fixed colour $c$ and $n=0$, the colimit diagram has a terminal object given by $A=\{c\}$ and the identity $1_c\in P(c;c) = P(A;c)$. So $g_!(P)$ as described by (\ref{eq5}) is indeed closed. If is straightforward to give an explicit description of the composition operations of the operad $g_!(P)$ in terms of (\ref{eq5}), which we will leave to the reader. 

The restriction functor $g^*$ from closed to open operads also has a right adjoint $g_*$. For an open operad $P$, this operad $g_*(P)$ has the same set $C$ of colours of $P$, while the operations are defined by 
\begin{equation} \label{eq6}
g_*(P)(c_1,\ldots, c_n; c) = \prod_I P(c_I; c). 
\end{equation}
Here the product ranges over all \emph{non-empty} subsets $I\subseteq \{1,\ldots, n\}$ and $c_I$ denotes the corresponding subsequence $(c_i\colon i\in I)$ of $(c_1, \ldots, c_n)$. This is a product over the empty set if $n=0$, so $g_*(P)$ is indeed closed. The operadic composition in $g_*(P)$ is defined as follows. Suppose $p=(p_I)\in g_*(P)(c_1, \ldots, c_n; c)$, and $q^{(i)} \in g_*(P)(d_1^{(i)},\ldots, d_{k_i}^{(i)}; c_i )$ for $i=1, \ldots, n$. Then the composition $p(q^{(1)}, \ldots, q^{(n)}) \in g_*(P) (d_1^{(1)}, \ldots, d_{k_n}^{(n)}; c)$ has coordinates for non-empty subsets $J=J_1+\ldots+J_n$ where $J_i\subseteq \{1,\ldots, k_i\}$, defined in terms of the composition operation of $P$ by 
\begin{equation} \label{eq7}
p(q^{(1)}, \ldots, q^{(n)})_J=p_I \circ (q_{J_1}^{(1)}, \ldots, q_{J_n}^{(n)})
\end{equation}
where $I=\{i | J_i \neq \emptyset\}$ and $(q_{J_1}^{(1)}, \ldots, q_{J_n}^{(n)})$ in (\ref{eq7}) only contains those operations $q_{J_i}^{(i)}$ for which $J_i\neq \emptyset$. 

We will need the explicit description of the units and counits of these adjunctions. The unit $\alpha\colon Q \to g^* g_! (Q)$ is given by the evident maps 
\[
Q(c_1, \ldots, c_n; c) \to \varinjlim_A Q(c_1, \ldots, c_n, A; c) 
\]
given by the vertex for $A=\emptyset$ in the colimit diagram.  The counit $ \beta\colon  g_!g^* (Q) \to Q$ if this adjunction maps an operation in $Q(c_1,\ldots, c_n, A; c)$ to the one in $Q(c_1, \ldots, c_n;c)$ given by substitution of the constants of colours in $A$. For the other adjunction, the unit  $\alpha'\colon Q \to g_* g^* (Q)$ has components $Q(c_1, \ldots, c_n; c) \to Q(c_I; c)$ for $\emptyset \neq I \subseteq \{1,\ldots, n\}$ given by substitution of the constants of colours $c_j$ with $j\not\in I$. The counit of this adjunction 
\[ \beta'\colon g^*g_* (Q) \to Q \]
is simply given by the projection maps 
\[ 
  g^*g_* (Q) (c_1, \ldots, c_n; c) = \prod_I Q(c_I; c)  \to Q(c_1,\ldots, c_n; c), 
\]
onto the factor $I=\{1,\ldots, n\}$ for $n>0$. (For $n=0$, it is the unique map from a point to a point.)
Observe that the same projection map also defines a retraction $\rho$, 
\[
\xymatrix@C=40pt{
	Q(c_1, \ldots, c_n; c) \ar@<.5ex>[r]_{\alpha'} & g_*g^* Q(c_1, \ldots, c_n; c) \ar@<.5ex>[l]_\rho} 
\]
which one can easily check to be a map of operads. We summarize the discussion as follows. 

\begin{theorem}
	The restriction functor $g^*\colon \cOper \to \oOper$, from closed simplicial operads to open ones, admits both a left adjoin $g_!$ and a right adjoint $g_*$ (described explicitly in (\ref{eq5}) and (\ref{eq6}) above). Moreover, for each closed operad $Q$ the unit of the adjunction $\alpha'\colon Q \to g_*g^*  (Q)$ admits a retraction $\rho$. 
\end{theorem}

\begin{remark}
	We observed that the functors in the theorem do not change the colours, hence restrict to adjoint functors 
	\begin{equation*} 
	\xymatrix@C=40pt{
		\oOper_C \ar@/^1pc/[r]^{g_!} \ar@/_1pc/[r]_{g_*} & \cOper_C  \ar[l]_{g^*} 
	}
	\end{equation*}
for some fixed set of colours $C$. The functor $g^*$ is moreover natural in $C$, in the sense that for a map of colours $f\colon D\to C$, the diagram of right adjoints (all obvious restriction functors) 
\[
\xymatrix{
\oOper_C \ar[d]_{f^*} & \cOper_C \ar[d]^{f^*} \ar[l]_{g^*}\\
\oOper_D & \cOper_D \ar[l]_{g^*}
}
\]
commutes, and hence so does the diagram of all the left adjoints $f_!$ and $g_!$, respectively. The same is true with $g^*$ replaced by $g_*$. 
\end{remark}

\section{Quillen model structures for closed operads with fixed colours} 	\label{sect3}

In this section and the next we discuss two types of model structures on closed operads, one of ``transfer type'' and of ``Reedy type''. We begin with the case of a fixed set of colours, and leave that of varying sets of colours to the next section. 

\begin{proposition}
	The category $\cOper_C$ of coloured operads with $C$ as set of colours carries a cofibrantly generated Quillen model structure for which a map $Q\to P$ is a fibration, respectively a weak equivalence, if and only if for each sequence of colours $c_1, \ldots, c_n, c$ the map $Q(c_1,\ldots, c_n; c)\to P(c_1,\ldots, c_n; c)$ is a Kan fibration, respectively a weak equivalence, of simplicial sets. 
\end{proposition}

In other words, this model structure is obtained by transfer form the one on simplicial sets, and we refer to it as the \emph{transferred} or the \emph{projective} model structure. 

\begin{proof}
	As the category $\cOper_C$ is itself the category of algebras in simplicial sets for another coloured operad the proposition is a special case of \cite{bergermoerdijk2}, Theorem 2.1. 
\end{proof}

The same category $\cOper_C$ also carries another model structure with the same weak equivalences, more relevant for what follows in this paper. We refer to this model structure as the one \emph{of Reedy type}. Its formulation requires a bit of notation. For a sequence of colours $c_1, \ldots, c_n, c$ in an operad $P$ write 
\[
P^-(c_1, \ldots, c_n; c) = \varprojlim_I P(c_I; c)
\] 
where the limit is taken over proper subsets $I\subseteq \{1, \ldots, n\}$. For two such sets $I\subseteq J$, the maps $P(c_J; c)\to P(c_I; c)$ in the system are induced by the substitution of the constants of the colours $c_j$ for $j\in J - I$. The same substitution induces a map 
\begin{equation} \label{eq3.1}
P(c_1, \ldots, c_n; c) \to P^- (c_1, \ldots, c_n; c). 
\end{equation}
We call a closed operad $P$ with $C$ as set of colours \emph{Reedy fibrant} if each of the maps is a Kan fibration. More generally, we call a map $Q\to P$ between such operads a \emph{Reedy fibration} if for each sequence $c_1, \ldots, c_n; c$ of colours, the map 
\begin{equation} \label{eq3.2}
Q(c_1, \ldots, c_n; c)\to Q^-(c_1, \ldots, c_n; c) \times_{P^-(c_1, \ldots, c_n; c)} P(c_1, \ldots, c_n; c)
\end{equation}
is a Kan fibration. 

\begin{proposition} \label{prop3.2} 
	The category $\cOper_C$ carries a cofibrantly generated model structure, in which the fibrations are the Reedy fibrations and in which a map $Q\to P$ is a weak equivalence if each $Q(c_1, \ldots, c_n; c) \to P(c_1, \ldots, c_n; c)$ is a weak equivalence of simplicial sets. 
\end{proposition}

\begin{remark}
	Suppose $Q\to P$ is a Reedy fibration. Then for each sequence $c_1, \ldots, c_n, c$ of colours, an easy induction on $I\subseteq \{1,\ldots, n\}$ shows that each of the maps $Q(c_I; c )\to P(c_I; c)$ is a fibration. In particular, each Reedy fibration is a fibration in the transferred model structure. Thus, the identity functor on $\cOper_C$ is a right Quillen equivalence from the Reedy model structure to the transferred one. 
\end{remark}

The proof of Proposition \ref{prop3.2} is quite standard, but requires an explicit description of pushouts of generating trivial cofibrations, and we need some notation for this. First of all, for a simplicial set $X$ let us write 
\[
\overline{C}_n[X]
\]
for the closed operad with $\{0,1,\ldots, n\}$ as set of colours, and characterized by the universal property that for each closed operad $P$ with set $C$ of colours, and for each sequence $c_i$, $i=0,\ldots, n$ of colours, maps of coloured operads $\overline{C}_n[X]\to P$ over the evident map $\varphi \colon \{0,\ldots, n\} \to C$ of colours mapping $i$ to $c_i$ correspond to maps of simplicial sets $X\to P(c_1, \ldots, c_n; c)$. Concretely, $\overline{C}_n[X]$ is the closed operad with spaces of operations $\overline{C}_n[X](i_1, \ldots, i_k; 0)=X$ for any non-empty subset $\{i_1, \ldots, i_k\} \subseteq\{1, \ldots, n\}$, and besides this only identity operations in $\overline{C}_n[X](i,i)$ and unique constants in $\overline{C}_n[X](-, i)$ for $i=0,\ldots, n$. If $I\subseteq \{1,\ldots, n\}$ we also write $\overline{C}_I[X]$ for the similar operad with colours $0$ and $i\in I$ only. (Thus $\overline{C}_I[X]\cong \overline{C}_k[X]$ for $k=|I|$.) If $I\subseteq J \subseteq \{1,\ldots, n\}$ there is an obvious map of coloured operads 
\[
\overline{C}_I[X]\to \overline{C}_J[X]
\]
and we write 
\[
\overline{C}_n^\partial[X]= \varinjlim_I \overline{C}_I[X]
\]
where the colimit ranges over \emph{proper} subsets $I\subseteq \{1,\ldots, n\}$. 

For an arbitrary set $C$ of colours and a sequence $c_1, \ldots, c_n, c$ of elements of $C$, there is a corresponding map $\varphi \colon \{0,\ldots, n\} \to C$ as above, and similar restrictions of $\varphi$ also denoted $\varphi \colon \{0\} \cup I \to C$. We write
\begin{eqnarray*}
\overline{C}_{c_1, \ldots, c_n, c}[X] = \varphi_! \overline{C}_n[X] \\
\overline{C}_{c_I,c}[X] = \varphi_! \overline{C}_I[X] \\
\overline{C}_{c_1, \ldots, c_n, c}^\partial[X] = \varphi_! \overline{C}_n^\partial[X] 
\end{eqnarray*}
for the induced operads with $C$ as set of colours. (These operads are considerably larger when the colour $c$ occurs among the $c_i$ as well since there are generated composition of operations.) Then for a map $Q\to P$ in $\cOper_C$, the map (\ref{eq3.2}) has the RLP with respect to an inclusion $X\rightarrowtail Y$ of simplicial sets if and only if $Q\to P$ has the RLP with respect to 
\begin{equation}\label{eq3.3}
\overline{C}_{c_1, \ldots, c_n, c}[X] \cup_{\overline{C}_{c_1, \ldots, c_n, c}^\partial[X]} \overline{C}_{c_1, \ldots, c_n, c}^\partial[Y] \to \overline{C}_{c_1, \ldots, c_n, c}[Y].
\end{equation}
Thus, these maps of the form (\ref{eq3.3}), where $X\rightarrowtail Y$ is a generating (trivial) cofibrations in $\sSets$, act as generating (trivial) cofibrations for the Reedy model structure on $\cOper_C$. 

\begin{lemma} \label{lemma2.4}
	For a trivial cofibration $X\to Y$ of simplicial sets, the pushout of (\ref{eq3.3}) along a map $\overline{C}_{c_1, \ldots, c_n, c}[X] \cup_{\overline{C}_{c_1, \ldots, c_n, c}^\partial[X]} \overline{C}_{c_1, \ldots, c_n, c}^\partial[Y] \to P$ in $\cOper_C$ is a weak equivalence.
\end{lemma}

\begin{proof}[Proof (of lemma)] 
	Let us consider the category of closed operads with $C$ a set of colours in $\Sets$ rather than $\sSets$ for the moment. For an inclusion of \emph{sets} $X\subseteq Y$, consider the pushout 
	\begin{equation} \label{eq3.4}
	\xymatrix{
	\overline{C}_{\overrightarrow{c}}[X] \cup_{\overline{C}_{\overrightarrow{c}}^\partial[X]} \overline{C}_{\overrightarrow{c}}^\partial[Y] \ar[rr] \ar[d]^f && \overline{C}_{\overrightarrow{c}}[Y] \ar[d] \\ P  \ar[rr] && Q }
	\end{equation}
	where $\overrightarrow{c}$ stands for $c_1, \ldots, c_n, c$. So $Q$ is obtained from $P$ by adjoining for each element $y\in Y$ an operation $y\colon c_1, \ldots, c_n \to c$ to $P$, while identifying this operation with the one already existing in $P$ and given by the map $f$ if $y\in X$, and similarly identifying all the operations $y_I\colon c_I \to c$ given by substitution of constants of colours $c_j$ for $j\in \{1,\ldots, n\} - I $ for proper subsets $I$ of $\{1,\ldots, n\}$. For colours $d_1, \ldots, d_k, d$ in $C $, the operations in $Q(d_1, \ldots, d_k; d)$ can be described explicitly by equivalence classes of planar labelled trees of the following kind (similar to \cite{bergermoerdijk1}, Appendix):
	\begin{enumerate}
		\item[(a)] the edges of the tree are labelled by colours in $C$,
		\item[(b)] the leaves of the tree are numbered $1,\ldots, k$ and the leaf with number $i$ has colour $d_i$, 
		\item[(c)] the root is labelled $d$, 
		\item[(d)] the vertices of the tree are of two possible colours, black and white,
		\item[(e)] each internal edge of the tree is connected to at least one white vertex,
		\item[(f)] the white vertices have input leaves labelled $c_1, \ldots, c_n$ in the given planar order, and output edge $c$, 
		\item[(g)] above each of the $n$ input edges of each white vertex there is at least one leaf, 
		\item[(h)] the white vertices are labelled by elements of $Y$, 
		\item[(i)] the black vertices are labelled by non-identity operations in $P$, compatible with the colours labelling the edges. 
	\end{enumerate}
(See the example given after the proof.)

For such a planar tree $T$, let $\mathrm{Aut}_0(T)$ be the subgroup of (non-planar) automorphisms of $T$ which respect the labelling, the colours of the vertices and the numbering of the leaves (as given by (a), (b), (c), (d)) and which do \emph{not} change the planar order of the input edges of the white vertices. In other words, as a \emph{set} $\mathrm{Aut}_0(T)$ is a subset of the product over black vertices $v$ of symmetric groups $\Sigma_{\mathrm{In}(v)}$ on the set $\mathrm{In}(v)$ of input edges of $v$. (As a group, it is rather an iterated semidirect product of these, cf. \cite{bergermoerdijk1}.) Then $\mathrm{Aut}_0(T)$ acts on the labelling of the black vertices via the symmetries of the operad $P$, and we write $P_T[Y]$ for the quotient. This quotient is really a quotient of the form 
\begin{equation} \label{eq3.5}
P_T[Y] = \prod_w Y \times \left( \prod_b P(\mathrm{In}(b), \mathrm{Out}(b)) \times L \right) \big/ \mathrm{Aut}_0(T)
\end{equation}
where $L$ stands for the labelling of edges and numbering of leaves, $w$ runs through the white vertices and $b$ through the black ones. 
The map $P(d_1, \ldots, d_k; d) \to Q(d_1, \ldots, d_k; d)$ in the pushout (\ref{eq3.4}) is a pushout of sets of the form 
\begin{equation} \label{eq3.6}
\xymatrix{
\coprod_T P_T[X] \ar[r]^{\tilde{f}} \ar[d] &  P(d_1, \ldots, d_k; d) \ar[d]  \\ \coprod_T P_T[Y] \ar[r]  & Q(d_1, \ldots, d_k; d) }
\end{equation}
where the upper map $\tilde{f}$ is given by $f\colon \overline{C}_{\overrightarrow{c}}[X] \to P$ and composition of the operad $P$ ``along'' the trees $T$. (The sum is over (non-planar) isomorphism classes of trees.) Exactly the same description applies degreewise if $X$ and $Y$ are simplicial sets. But then clearly, if $X\rightarrowtail Y$ is a trivial cofibration then by the form (\ref{eq3.5}) so is the map on the left of (\ref{eq3.6}), and hence so is the map on the right. This shows that $P\to Q$ is a weak equivalence, and the lemma is proved.   
\end{proof}

\begin{example} For $n=2$, $y\in Y$ and $p\in P(a,b; c_2)$ the following two trees are related by an automorphism in $\mathrm{Aut}_0(T)$, and represent the same operation in $Q(b,a,c_1;c)$

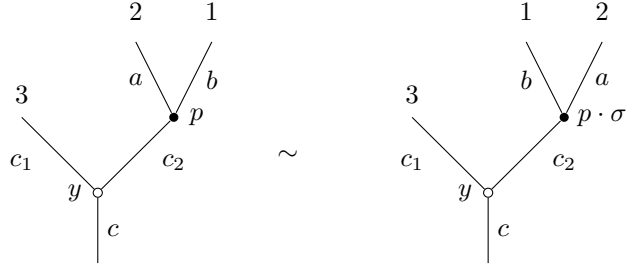
\[
\begin{tikzpicture} 
[level distance=10mm, 
every node/.style={fill, circle, minimum size=.12cm, inner sep=0pt}, 
level 1/.style={sibling distance=20mm}, 
level 2/.style={sibling distance=20mm}, 
level 3/.style={sibling distance=10mm}]

\node(anchorR1)[style={color=white}] {} [grow'=up] 
child {node(vertexR1)[draw,fill=none] {} 
	child
	child{node {}
		child
		child
	}
};

\node[style={color=white}, right=5cm of anchorR1] {} [grow'=up] 
child {node(vertexR2)[draw,fill=none] {} 
	child
	child{node {}
		child
		child
	}
};

\tikzstyle{every node}=[]

\node at ($(vertexR1) + (.2,-.5cm)$) {$c$};
\node at ($(vertexR1) + (-.3,0cm)$) {$y$};
\node at ($(vertexR1) + (-1cm,.4cm)$) {$c_1$};
\node at ($(vertexR1) + (-1cm,1.3cm)$) {$3$};
\node at ($(vertexR1) + (1cm,.4cm)$) {$c_2$};
\node at ($(vertexR1) + (1.3cm,1cm)$) {$p$};
\node at ($(vertexR1) + (.5cm,1.5cm)$) {$a$};
\node at ($(vertexR1) + (.5cm,2.4cm)$) {$2$};
\node at ($(vertexR1) + (1.5cm,1.5cm)$) {$b$};
\node at ($(vertexR1) + (1.5cm,2.4cm)$) {$1$};
\node at ($(vertexR1) + (2.5cm,.5cm)$) {$\sim$};

\node at ($(vertexR2) + (.2,-.5cm)$) {$c$};
\node at ($(vertexR2) + (-.3,0cm)$) {$y$};
\node at ($(vertexR2) + (-1cm,.4cm)$) {$c_1$};
\node at ($(vertexR2) + (-1cm,1.3cm)$) {$3$};
\node at ($(vertexR2) + (1cm,.4cm)$) {$c_2$};
\node at ($(vertexR2) + (1.5cm,1cm)$) {$p\cdot \sigma$};
\node at ($(vertexR2) + (.5cm,1.5cm)$) {$b$};
\node at ($(vertexR2) + (.5cm,2.4cm)$) {$1$};
\node at ($(vertexR2) + (1.5cm,1.5cm)$) {$a$};
\node at ($(vertexR2) + (1.5cm,2.4cm)$) {$2$};
\end{tikzpicture} 
\]
This operation could be denoted $(y \circ_2 p) \cdot \tau$ for the appropriate permutation $\tau$ (as pictured in the left hand tree); or by using dummy variables $v_1$, $v_2$, $v_3$ as place holders, $y(v_3, p(v_2, v_1))$. The tree

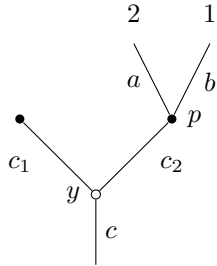
\[
\begin{tikzpicture} 
[level distance=10mm, 
every node/.style={fill, circle, minimum size=.12cm, inner sep=0pt}, 
level 1/.style={sibling distance=20mm}, 
level 2/.style={sibling distance=20mm}, 
level 3/.style={sibling distance=10mm}]

\node(anchorR1)[style={color=white}] {} [grow'=up] 
child {node(vertexR1)[draw,fill=none] {} 
	child{node {}}
	child{node {}
		child
		child
	}
};

\tikzstyle{every node}=[]

\node at ($(vertexR1) + (.2,-.5cm)$) {$c$};
\node at ($(vertexR1) + (-.3,0cm)$) {$y$};
\node at ($(vertexR1) + (-1cm,.4cm)$) {$c_1$};
\node at ($(vertexR1) + (1cm,.4cm)$) {$c_2$};
\node at ($(vertexR1) + (1.3cm,1cm)$) {$p$};
\node at ($(vertexR1) + (.5cm,1.5cm)$) {$a$};
\node at ($(vertexR1) + (.5cm,2.4cm)$) {$2$};
\node at ($(vertexR1) + (1.5cm,1.5cm)$) {$b$};
\node at ($(vertexR1) + (1.5cm,2.4cm)$) {$1$};

\end{tikzpicture} 
\]
violates condition (f). Indeed, it represents an operation involving the substitution of the constant of colour $c_1$ into $y$, i.e. an operation in $\overline{C}_{\overrightarrow{c}}^\partial[Y]$, so the map $f$ in (\ref{eq3.4}) allows us to identify this operation with one already in $P$. 
\end{example}

\begin{proof}[Proof of Proposition \ref{prop3.2}]
	The proof of this proposition now follows one of several standard patterns. The cofibrations are by definition the maps having the LLP with respect to trivial fibrations, and one observes that the ``generating (trivial) cofibrations'' as described in (\ref{eq3.3}) are indeed (trivial) cofibrations. Referring to the formulation of the axioms as CM1-5 in \cite{quillenRHT}, axioms CM1-3 are now obvious. The factorization of a map into a cofibration followed by a trivial fibration (more precisely, by a map having the RLP with respect to all cofibrations), and the factorization into a trivial cofibration followed by a fibration, is achieved by the small object argument. (It is here that one uses the preceding Lemma.) Finally, one half of the lifting axiom CM4 holds by definition of the cofibrations, while the other half is proved by the standard retract argument. 
\end{proof}

\begin{remark}
	For a map of sets $f\colon D\to C$, the pullback functor 
	\[
	f^* \colon \cOper_C \to \cOper_D
	\]
	clearly preserves weak equivalences, and one easily checks that is also preserves Reedy fibrations. Thus, this functor together with its left adjoint forms a Quillen pair. 
\end{remark}

\section{Quillen model structure on closed operads} \label{sect4}

In this section, we will prove the existence of two Quillen model structures on the category $\cOper$, of all closed simplicial operads without fixing the colours. One of the structures is of projective or transfer type and one is of Reedy type, as for the case of fixed colours discussed above. 

To begin with, recall the Quillen model structure on the category $\Oper$ of all simplicial coloured operads introduce in \cite{cisinskimoerdijk3}. For two such operads $Q=(Q, D)$ and $P=(P,C)$, a morphism $\varphi \colon Q \to P$ is a \emph{fibration} in this model structure iff (i) for each sequence $d_1, \ldots, d_n, d$ of colours of $Q$ the map $Q(d_1, \ldots, d_n; d) \to P(\varphi d_1, \ldots, \varphi d_n; \varphi d)$ is a fibration of simplicial sets (i.e., $\varphi$ is a ``local fibration''); and (ii) the map $\pi_0 j^* Q \to \pi_0 j^* P$ is a fibration in the naive model structure on discrete categories. Here 
\[
j^*\colon \Oper \to \Cat 
\]
is the functor taking a simplicial operad to the simplicial category given by restricting the operad to its unary operations only, and $\pi_0$ takes a simplicial category to a discrete category by taking the set of connected components of each simplicial set of morphisms. A morphism $\varphi \colon Q \to P$ as above is \emph{weak equivalence} in this model structure iff each $Q(d_1, \ldots, d_n; d) \to P(\varphi d_1, \ldots, \varphi d_n; \varphi d)$ is a weak equivalence of simplicial sets (i.e., $\varphi$ is a ``local weak equivalence'', or $\varphi$ is ``fully faithful'') and $\pi_0 j^* Q \to \pi_0 j^* P$ is an essentially surjective functor between discrete categories. (One says that $\varphi\colon Q \to P$ is essentially surjective in this case.)

This model structure is cofibrantly generated: The generating cofibrations are the maps of the following two types 
\begin{itemize}
	\item[(C1)] $\emptyset \to \eta$
	\item[(C2)] $C_n[X] \to C_n[Y]$, for $X\rightarrowtail Y$ a generating cofibration of simplicial sets, and $n\geqslant 0$.
\end{itemize}
Here $\eta$ is the operad with one colour and an identity operation only. For a simplicial set $X$, the operad $C_n[X]$ has colours $0,1,\ldots, n$, a space $X$ of operations $1, \ldots, n \to~0$, and no other operations besides identities. The generating trivial cofibrations are the maps of the following two types: 
\begin{itemize}
	\item[(A1)] $\eta \to H$
	\item[(A2)] $C_n[X]\to C_n[Y]$, for $X\rightarrowtail Y$ a generating trivial cofibration of simplicial sets, and $n\geqslant 0$.
\end{itemize}
Here $H$ is a \emph{countable} simplicial category with two objects $0$ and $1$ which is cofibrant in the Bergner model structure on simplicial categories \cite{bergner}, and in which all the simplicial hom-sets are weakly contractible. (Of course, in (C2) and (A2) we can restrict ourselves to any small set of monomorphisms $X\rightarrowtail Y$ which generate the (trivial) cofibrations; for example, we can ask $X$ and $Y$ to be finite.)

This model structure on the category $\Oper$ restricts to one on the category $\oOper$ of open simplicial operads because the latter category is a slice of the former. Since categories are open operads, the sets of generating (trivial) cofibrations are still described by (C1,2) and (A1,2) except that one asks $n>0$ in (C2) and (A2). The description of fibrations and weak equivalences also applies verbatim. 

\begin{proposition} \label{prop4-1}
	The category $\cOper$ of closed simplicial operads carries a Quillen model structure, uniquely determined by the requirement that the restriction functor $g^* \colon \cOper\to \oOper$ preserves and detects fibrations and weak equivalences. (We refer to this model structure as the \emph{transferred} of \emph{projective} one.) 
\end{proposition} 

\begin{proof}
	As the formulation of the proposition already makes clear, this model structure is obtained by transfer along the adjunction
	\[
	\xymatrix{
	g_!: \oOper \ar@<.5ex>[r] & \cOper :g^*, \ar@<.5ex>[l]} 
	\]	
	 and we need to check that the conditions for transfer are met. Since $g^*$ also has a right adjoint (Section \ref{section2}, (\ref{eq3})), it suffices to prove that for any generating trivial cofibration $U\rightarrowtail V$ in $\oOper$ (of type (A1) or (A2)), the map 
	\[
	g^* g_! U\to g^*  g_! V
	\]
	is again a trivial cofibration in $\oOper$. The composite functor $g^*g_!$ is the identity on simplicial categories (i.e., simplicial operads with unary operations only), so this is clear for maps of type (A1). Let us consider a ``corolla'' $C_n[X]$ with $n>0$ as in the maps of type (A2). Then $g_!C_n[X]$ is the ``closed'' corolla $\overline{C}_n[X]$ generated by $\overline{C}_n[X](1,\ldots, n; 0) = X$, so has 
	\[
	\overline{C}_n[X](I; 0)=X
	\]
	for any non-empty sequence $I=(i_1, \ldots, i_k)$ of \emph{distinct} numbers $0\leq i_1, \ldots, i_k \leq n$. The only possible compositions in this operad are the substitutions of constants, which give the identity map $\overline{C}_n[X](I; 0) \to \overline{C}_n[X](J; 0)$ for $\emptyset\neq J \subseteq I$. Restricting back to the open part we find that 
	\[ \label{4.1}
	g^*g_! C_n[X] = \bigvee_I C_I[X].
	\]
	Here $C_I[X]\cong C_k[X]$ is the $k$-corolla with root edge $0$ and $I=\{i_1 < \ldots < i_k\}$ as set of leaves, and $\bigvee$ is the coproduct in $\oOper_{\{0,\ldots,n\}}$ ranging over all non-empty subsets $I\subseteq \{1,\ldots,n\}$. (So, strictly speaking, we view $C_I[X]$ in this coproduct as an operad with colours $\{0,\ldots,n\}$ by adding identites for all $j\in \{1,\ldots,n\}-I$.) In particular, the functor $g^*g_!$ maps a generating trivial cofibration $C_n[X]\to C_n[Y]$ to the coproduct 
	\begin{equation} \label{eq4-1}
	\bigvee_I C_I[X] \to \bigvee_I C_I[Y],
	\end{equation}
	which is again a trivial cofibration. (Notice that, for any fixed set $C$ of colours, the inclusion $\Oper_C \hookrightarrow \Oper$ obviously preserves trivial cofibrations.) This shows that the conditions for transfer are fulfilled, and proves Proposition \ref{prop4-1}.
\end{proof}

\begin{remark}
	As for any transferred model structure, the model structure of Proposition \ref{prop4-1} exhibits the adjoint pair $g_!, g^*$ as a Quillen pair; i.e. $g^*$ is a right Quillen functor. The proof above makes it clear that in fact, $g^*$ is also a left Quillen functor. 
\end{remark}

For the Reedy model structure, the situation is slightly more subtle. Let us call a map $\varphi\colon (Q,D) \to (P,C)$ in $\cOper$ a \emph{Reedy fibration} if it is a fibration in the projective model structure of Proposition \ref{prop4-1}, and if moreover $Q\to \varphi^* P$ is a Reedy fibration in $\cOper_D$. In other words, $\varphi\colon Q\to P$ is a Reedy fibration iff it has the RLP with respect to 
\begin{itemize}
	\item[(R1)] $\eta \to H$ (as in (A1))
	\item[(R2)] $\overline{C}_n[X] \cup_{\overline{C}^\partial_n[X]} \overline{C}^\partial_n[Y] \to \overline{C}_n[Y]$ for each $n \geq 0$ and each generating trivial cofibration $X\rightarrowtail Y$ in $\sSets$ (cf. Section \ref{sect3}, (\ref{eq3.3})).
\end{itemize}
Recall that the generating trivial cofibrations in the Reedy model structure on $\cOper_C$ are the maps of the form 
\[
f_!U \rightarrowtail f_!V
\]
for $U\rightarrowtail V$ a map of type (R2) above, and for any function $f\colon \{0,\ldots, n\}\to C$ of colours. As a preparation for the proof of Proposition \ref{prop4-2}, we also observe the following two properties of pushouts in $\cOper$, both easily verified. 

\begin{lemma}
	\begin{enumerate}
		\item Let $U\to V$ be a map in $\cOper_C$ and let $f\colon C\to D$ be a function of sets. Then the diagram 
		\[
		\xymatrix{ U \ar[r] \ar[d] & f_!(U) \ar[d]\\ 
					V \ar[r] & f_!(V)}
		\]
		is a pushout in $\cOper$. 
		\item The inclusion $\cOper_C \to \Oper$ preserves pushouts. 
	\end{enumerate}
\end{lemma}

\begin{proposition} \label{prop4-2}
	The category $\cOper$ of closed simplicial operads carries a model structure with the same weak equivalences as the one of Proposition \ref{prop4-1}, and with the Reedy fibrations just defined as fibrations. (We refer to this model structure as the Reedy model structure.)
\end{proposition}

\begin{proof}
	Define the ``Reedy'' cofibrations to be the maps having the LLP with respect to all the trivial Reedy fibrations; i.e. Reedy fibrations which are weak equivalences in the model structure of Proposition \ref{prop4-1}. Since these are in particular projective trivial fibrations, any cofibration in the projective model structure is a Reedy cofibration. So any map having the RLP with respect to the Reedy cofibrations is a weak equivalence. The existence of the model structure now follows by standard arguments, once we show that for a generating trivial cofibration $U\to V$ of type (R1) and (R2) and for any map $U\to P$ of closed operads, the pushout map $P\to V\cup_U P$ is a weak equivalence. For maps of type (R1) this was already verified in the proof of Proposition \ref{prop4-1}. For a map $U\to V$ of type (R2), let $f \colon V\to P$ and also write $f\colon \{0,\ldots, n\} \to C$ for the map on colours. Then we can decompose the pushout as 
	\begin{equation} \label{eq4-2}
	\xymatrix{ U \ar[r] \ar[d] & f_!(U) \ar[d] \ar[r] & P \ar[d] \\ 
		V \ar[r] & f_!(V) \ar[r] & V\cup_U P. }
	\end{equation}
	But $f_! U\to f_! V$ is a generating trivial cofibration in $\cOper_C$ (by definition), and these were verified to be weak equivalences in Section \ref{sect3}, Lemma \ref{lemma2.4}. The right hand square in (\ref{eq4-2}) is a pushout since the rectangle and the left hand square are. So it follows by the model structure on $\cOper_C$ that $P\to V\cup_U P$ is a weak equivalence here, hence a fortiori in $\cOper$. This completes the proof.
\end{proof}

\begin{remark}
	One easily checks that the map of type (C1), together with the maps like those in (R2), but where $X\rightarrowtail Y$ is just a generating cofibration, form a set of generating cofibrations for the Reedy model structure of Proposition \ref{prop4-2}. 
\end{remark}

\begin{remark} As any Reedy fibration is a projective fibration, the identity functor is a right Quillen equivalence from the Reedy model structure on $\cOper$ to the projective one. For the same reason, the functor $g^* \colon \cOper\to \oOper$ is still right Quillen for the Reedy model structure on $\cOper$. As for the projective model structure, it is in fact also left Quillen. Indeed, it maps a generating cofibration of the form 
	\[
	\overline{C}_n[X] \cup_{\overline{C}^\partial_n[X]} \overline{C}^\partial_n[Y] \to \overline{C}_n[Y]
	\]
	to the map 
	\[
	\bigvee_I C_I[Y] \cup C_n[X] \to \bigvee_I C_I[Y] \cup C_n[Y]
	\]
	where $I$ ranges over non-empty proper subsets of $\{1,\ldots,n\}$. This map is a cofibration in $\oOper$, in fact a pushout of $C_n[X]\rightarrowtail C_n[Y]$. 
\end{remark}	
	
\section{Open and closed dendroidal sets} \label{sect5}

In this section we review the definition of the categories of open and closed dendroidal sets, and introduce several adjoint pairs of functors between these categories. In the next sections, we will introduce a model structure on closed dendroidal sets and investigate to what extent these adjoint pairs are Quillen pairs. 

We refer to \cite{moerdijkweiss, HMbook} for a definition of the category $\Omega$ of trees indexing the category $\dSets$ of dendroidal sets, i.e. of presheaves on $\Omega$. A tree $T$ in $\Omega$ is called \emph{open} if it has no stumps (vertices without input edges, nullary vertices), and closed if it has no leaves (i.e. no external edges other than the root edge). 

\[
\begin{tikzpicture} 
[level distance=8mm, 
every node/.style={fill, circle, minimum size=.12cm, inner sep=0pt}, 
level 1/.style={sibling distance=10mm}, 
level 2/.style={sibling distance=10mm}, 
level 3/.style={sibling distance=10mm}]

\node(anchorR1)[style={color=white}] {} [grow'=up] 
child {node(vertexR1)[draw] {} 
	child{node {}
		child{node {}}
		child{node {}}
	}
	child{node {}}
	child{node {}}
};

\node[style={color=white}, right=5cm of anchorR1] {} [grow'=up] 
child {node(vertexR2)[draw] {} 
	child{node {}
	child
	child{node {}}
}
child{node {}}
child
};

\node[style={color=white}, right=10cm of anchorR1] {} [grow'=up] 
child {node(vertexR3)[draw] {} 
	child{node {}
		child
		child
	}
	child
	child
};

\tikzstyle{every node}=[]

\node at ($(vertexR1) + (0,-1.5cm)$) {closed tree};

\node at ($(vertexR2) + (0,-1.5cm)$) {arbitrary tree};

\node at ($(vertexR3) + (0,-1.5cm)$) {open tree};

\end{tikzpicture} 
\]

Notice that in the terminology of \cite{HMbook}, every elementary face $S\to T$ of a closed tree $T$ is an \emph{inner face}, with the possible exception of the root face. Moreover, a closed tree has a root face only if its root vertex is unary. We shall write $\Omega_{cl}$ for the full subcategory of $\Omega$ consisting of closed trees and $\Omega_o$ for the one consisting of open trees. Moreover, we shall denote the inclusion functors by $u$ and $o$ respectively: 
\[
\xymatrix{
\Omega_{cl} \ar[r]^{u} & \Omega & \Omega_o \ar[l]_o.
}
\]
The inclusion functor $u$ has a left adjoint $cl$, taking a tree $T$ to its \emph{closure}, denoted $cl(T)$ or $\overline{T}$: this is the tree obtained from $T$ by simply putting a stump on top of each leaf. Moreover, we will write $h=cl \circ o$. These functors all fit into a diagram 
\begin{equation} \label{eq5-1}
\xymatrix{
\Omega_o \ar[r]^o \ar[rd]_h & \Omega \ar@<0.5ex>[d]^{cl} \\ 
							& \Omega_{cl} \ar@<0.5ex>[u]^u
} \qquad h=cl\circ o
\end{equation}
These functors all induce adjoint pairs between the presheaf categories. In general, if $f\colon \mathbb{A}\to \mathbb{B}$ is a functor between small categories, restriction along $f$ defines a functor $f^*$ which has both a left adjoint $f_!$ and a right adjoint $f_*$, 
\begin{equation} \label{eq5-2} 
\xymatrix@C=40pt{
	\Sets^{\mathbb{A}^{op}} \ar@/^1pc/[r]^{f_!} \ar@/_1pc/[r]_{f_*} & \Sets^{\mathbb{B}^{op}} \ar[l]_{f^*} 
}
\end{equation}
The presheaf categories associated to the small categories in (\ref{eq5-1}) are the categories $\dSets$ of \emph{dendroidal sets}, $\odSets$ of \emph{open dendroidal sets} (presheaves on $\Omega_o$), and $\cdSets$ of \emph{closed dendroidal sets} (presheaves on $\Omega_{cl}$). Thus we obtain functors  
\begin{equation} \label{eq5-3} 
\xymatrix@C=40pt{
	\odSets \ar@/^1pc/[r]^{h_!} \ar@/_1pc/[r]_{h_*} & \cdSets \ar[l]_{h^*} 
}
\end{equation}
\begin{equation} \label{eq5-4} 
\xymatrix@C=40pt{
	\odSets \ar@/^1pc/[r]^{o_!} \ar@/_1pc/[r]_{o_*} & \dSets \ar[l]_{o^*} 
}
\end{equation}
\begin{equation} \label{eq5-5} 
\xymatrix@C=40pt{
	\cdSets \ar@/^1pc/[r]^{u_!} \ar@/_1pc/[r]_{u_*} & \dSets \ar[l]_{u^*} 
}
\end{equation}
\begin{equation} \label{eq5-6} 
\xymatrix@C=40pt{
	\dSets \ar@/^1pc/[r]^{{cl}_!} \ar@/_1pc/[r]_{{cl}_*} & \cdSets \ar[l]_{{cl}^*} 
}
\end{equation}
respectively. Moreover, since $cl$ is left adjoint to $u$, also 
\begin{equation} \label{5-7}
 cl^*=u_!, \quad cl*=u^*
\end{equation}
while by definition of $h$ as $cl \circ o$, also 
\begin{equation} \label{5-8}
h_!=cl_!o_!, \quad h^*=cl^*o^*, \quad h_*=cl_*o_*.
\end{equation}
For a closed tree $T$, viewed as a representable presheaf, we will also write 
\[
h^*(T) = \Int(T) = \varinjlim_{S\to T,\, S \text{ open}}S
\]
for the ``interior'' of $T$. (Warning: this interior is not just obtained simply by removing the stumps from $T$. For examples, for the closed $2$-corolla pictured on the left, 
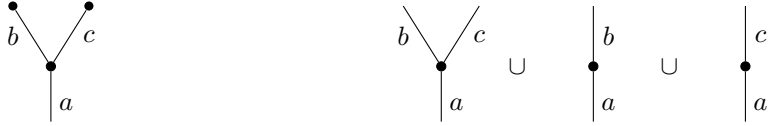
\[
\begin{tikzpicture} 
[level distance=8mm, 
every node/.style={fill, circle, minimum size=.12cm, inner sep=0pt}, 
level 1/.style={sibling distance=10mm}, 
level 2/.style={sibling distance=10mm}, 
level 3/.style={sibling distance=10mm}]

\node(anchorR1)[style={color=white}] {} [grow'=up] 
child {node(vertexR1)[draw] {} 
	child{node {}}
	child{node {}}
};

\node[style={color=white}, right=5cm of anchorR1] {} [grow'=up] 
child {node(vertexR2)[draw] {} 
	child
	child
};

\node[style={color=white}, right=7cm of anchorR1] {} [grow'=up] 
child {node(vertexR3)[draw] {} 
	child
};

\node[style={color=white}, right=9cm of anchorR1] {} [grow'=up] 
child {node(vertexR4)[draw] {} 
	child
};

\tikzstyle{every node}=[]

\node at ($(vertexR1) + (-.5cm,.4cm)$) {$b$};
\node at ($(vertexR1) + (.5cm,.4cm)$) {$c$};
\node at ($(vertexR1) + (.2cm,-.5cm)$) {$a$};

\node at ($(vertexR2) + (1cm,0cm)$) {$\cup$};
\node at ($(vertexR2) + (-.5cm,.4cm)$) {$b$};
\node at ($(vertexR2) + (.5cm,.4cm)$) {$c$};
\node at ($(vertexR2) + (.2cm,-.5cm)$) {$a$};

\node at ($(vertexR3) + (1cm,0cm)$) {$\cup$};
\node at ($(vertexR3) + (.2cm,.4cm)$) {$b$};
\node at ($(vertexR3) + (.2cm,-.5cm)$) {$a$};

\node at ($(vertexR4) + (.2cm,.4cm)$) {$c$};
\node at ($(vertexR4) + (.2cm,-.5cm)$) {$a$};

\end{tikzpicture} 
\]
its interior is the union of the three representable open dendroidal sets on the right (where ``union'' means pushout identifying the edges)).

\begin{remark} \label{remark4.1}
	(1) Let $O$ be the dendroidal set defined by 
		\[
		O_T = \left\{ 
		\begin{array}{ll}\{*\}, & \text{if } T \text{ is open}, \\ \emptyset, & \text{otherwise}.
		\end{array}
		\right. 
		\]
		(This dendroidal set is well-defined because if $S\to T$ is a morphism in $\Omega$ and $T$ is open, then $S$ must be open as well.) Then $\odSets= \dSets / O$, and the functors in (\ref{eq5-4}) correspond to the usual adjoint functors associated with a slice category ($o_*$ corresponds to ``product with $O$'', etc.).
		
	(2) The functors $cl_!$ and $h_!$ behave rather poorly on (normal) monomorphisms, cf. \cite{HMbook}. For example, consider the open tree $S$, 
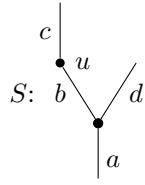
		\[
		\begin{tikzpicture} 
		[level distance=8mm, 
		every node/.style={fill, circle, minimum size=.12cm, inner sep=0pt}, 
		level 1/.style={sibling distance=10mm}, 
		level 2/.style={sibling distance=10mm}, 
		level 3/.style={sibling distance=10mm}]
		
		\node(anchorR1)[style={color=white}] {} [grow'=up] 
		child {node(vertexR1)[draw] {} 
			child{node {}
				child
			}
			child
		};

		\tikzstyle{every node}=[]
		
		\node at ($(vertexR1) + (-.5cm,.4cm)$) {$b$};
		\node at ($(vertexR1) + (-.7cm,1.2cm)$) {$c$};
		\node at ($(vertexR1) + (.5cm,.4cm)$) {$d$};
		\node at ($(vertexR1) + (.2cm,-.5cm)$) {$a$};
		\node at ($(vertexR1) + (-.2cm,.8cm)$) {$u$};
		\node at ($(vertexR1) + (-1cm,.4cm)$) {$S$:};
		
		\end{tikzpicture} 
		\]
		and $U=\partial_u(S) \cup \partial_b(S)$. Then $U\rightarrowtail S$ is a normal monomorphism. The intersection $V=\partial_u(S) \cap \partial_b(S)$ is a disjoint sum of copies of the unit tree $\eta$, one for $a$ and one for $d$. However, $\overline{\partial_u(S)} \cap \overline{\partial_b(S)}$ is the larger representable
		
			\[
		\begin{tikzpicture} 
		[level distance=8mm, 
		every node/.style={fill, circle, minimum size=.12cm, inner sep=0pt}, 
		level 1/.style={sibling distance=10mm}, 
		level 2/.style={sibling distance=10mm}, 
		level 3/.style={sibling distance=10mm}]
		
		\node(anchorR1)[style={color=white}] {} [grow'=up] 
		child {node(vertexR1)[draw] {} 
			child{node {}}
		};

		\tikzstyle{every node}=[]
		
		\node at ($(vertexR1) + (.2cm,.4cm)$) {$d$};
		\node at ($(vertexR1) + (.2cm,-.5cm)$) {$a$};

		\end{tikzpicture} 
		\]
		strictly larger than $\overline{V}$. Thus, $cl_!(U) \to cl_!(S) = \overline{S}$ is not mono. 
		
		(3) Since the inclusions $o\colon \Omega_o \to \Omega$ and $u\colon \Omega_{cl} \to \Omega$ are fully faithful, so are their left Kan extensions $o_! \colon \odSets\hookrightarrow \dSets$ and $u_! \colon\cdSets\hookrightarrow \dSets$. Therefore, we often identify $\odSets$ and $\cdSets$ with the corresponding full subcategories of $\dSets$, and delete $o_!$ and $u_!$ from the notation whenever this does not lead to confusion. 
\end{remark}

\section{Tensor products and normal monomorphisms} \label{sect6}

Recall that the category $\dSets$ of dendroidal sets carries a tensor product. This tensor product, denoted $X\otimes Y$, preserves colimits in each variable separately, so is determined by what it does on representables. If $X=S$ and $Y=T$ are representables given by trees $S$ and $T$, then 
\[
S \otimes T = \bigcup_{\text{shuffles}} A 
\]
is a union of representables, namely of all shuffles $A\subseteq S\otimes T$. The same is true for an $n$-fold tensor product $S_1\otimes \ldots \otimes S_n$, and there are associativity morphisms 
\[
S_1\ldots \otimes (S_i\otimes \ldots \otimes S_j) \otimes \ldots \otimes S_n \to S_1\otimes \ldots \otimes S_n. 
\]
We refer to \cite{HHM, HMbook} for a detailed discussion, including conditions under which these associativity morphisms are isomorphisms or weak equivalences. 

The closure $cl_!(X)$ of a dendroidal set $X$ can be defined in terms of this tensor product as 
\begin{equation} \label{eq6-1}
cl_! (X)=X\otimes \overline{\eta} 
\end{equation}
where $\overline{\eta}$ is the closed unit tree. 
	
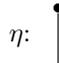
\[
\begin{tikzpicture} 
[level distance=8mm, 
every node/.style={fill, circle, minimum size=.12cm, inner sep=0pt}, 
level 1/.style={sibling distance=10mm}, 
level 2/.style={sibling distance=10mm}, 
level 3/.style={sibling distance=10mm}]

\node(anchorR1)[style={color=white}] {} [grow'=up] 
child {node(vertexR1) {} 
};

\tikzstyle{every node}=[]

\node at ($(vertexR1) + (-.5cm,-.4cm)$) {$\eta$:};

\end{tikzpicture} 
\]
(A more precise way of writing (\ref{eq6-1}) is as $u_!cl_! X= X\otimes \overline{\eta}$, cf. Remark \ref{remark4.1}, (3) above.) The full subcategories of open and of closed dendroidal sets are both closed under tensor products, and we will use the same notation $X\otimes Y$ for the tensor products on $\cdSets$ and $\odSets$. (So, more formally, $u_! X \otimes u_! Y = u_! cl_! (u_!X\otimes u_!Y)$, and we will write $cl_!(u_!X \otimes u_!Y)$ simply as $X\otimes Y$. Then $u_!(X\otimes Y)=u_!X \otimes u_!Y$.) Indeed, these properties are easily checked for representables and extend to colimits, so hold generally. 

It follows that the simplicial enrichment of dendroidal sets (see \cite{HHM}) extends to open and to closed dendroidal sets. Writing 
\[
i_! \colon \sSets \hookrightarrow \dSets
\]
for the inclusion as in \cite{moerdijkweiss}, one notices that $i$ factors through $\odSets\hookrightarrow \dSets$, and one can define the simplicial tensor on $\cdSets$ more formally as 
\[
M\otimes X = cl_! i _! (M) \otimes X 
\]
for $M$ a simplicial set and $X$ an object of $\cdSets$. Then the adjoint functors $u_! \colon \cdSets\hookrightarrow \dSets$ and $cl_! \colon \dSets\to \cdSets$ preserve the tensor up to canonical isomorphism: 
\[
u_! (M\otimes X) = M \otimes u_!X \quad \text{and} \quad cl_!(M\otimes Y) = M \otimes cl_!(Y), 
\]
for $M$ a simplicial set, $X$ in $\cdSets$ and $Y$ in $\dSets$. 

Next, we turn to the notion of a \emph{normal monomorphism}. Recall that a monomorphism $X\rightarrowtail Y$ of dendroidal sets is called \emph{normal} if for each tree $T$, the group $\Aut(T)$ acts freely on the complement of $X_T$ inside $Y_T$. The class of normal monomorphisms is saturated, and generated by boundary inclusions $\partial T \rightarrowtail T$ of representables. As a consequence, every $X$ has a ``normalization'' $\tilde{X}$, i.e. a factorization of $\emptyset \to X$ into a normal mono $\emptyset \to \tilde{X}$ (i.e. $\tilde{X}$ is a normal object) and a map $\tilde{X} \to X$ having the RLP with respect to all normal monomorphisms. 

All this extends to open dendroidal sets in the obvious way, because if $X$ is open, then for any map $Y\to X$, the object $Y$ must be open as well. So $\tilde{X}$ is open if $X$ is, and $\partial T$ is open if $T$ is. 

All this also extends to closed dendroidal sets, but some care is needed to distinguish the ``closed boundary'' of a closed tree from its boundary $\partial T$, so we will introduce a separate notation. 

\begin{definition}
	For a closed tree $T$, its \emph{closed boundary} is the largest proper subobject of $T$ in the category $\cdSets$. It is the union of all the inner faces of $T$, together with the root face if the root face exists. It is denoted by 
	\[
	\partial_{cl}(T) \subseteq T. 
	\]
\end{definition}

\begin{example}
	For the tree $T$ pictured below 
	
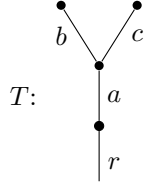
\[
	\begin{tikzpicture} 
	[level distance=8mm, 
	every node/.style={fill, circle, minimum size=.12cm, inner sep=0pt}, 
	level 1/.style={sibling distance=10mm}, 
	level 2/.style={sibling distance=10mm}, 
	level 3/.style={sibling distance=10mm}]
	
	\node(anchorR1)[style={color=white}] {} [grow'=up] 
	child {node(vertexR1)[draw] {} 
		child{node {}
			child{node {}}
			child{node {}}
		}
	};

	\tikzstyle{every node}=[]
	
	\node at ($(vertexR1) + (.2cm,.4cm)$) {$a$};
	\node at ($(vertexR1) + (-.5cm,1.2cm)$) {$b$};
	\node at ($(vertexR1) + (.5cm,1.2cm)$) {$c$};
	\node at ($(vertexR1) + (.2cm,-.5cm)$) {$r$};
	\node at ($(vertexR1) + (-1cm,.4cm)$) {$T$:};
	
	\end{tikzpicture} 
	\] 
	the closed boundary $\partial_{cl} T$ is the union of the root face on the left and the three inner faces on the right: 
	
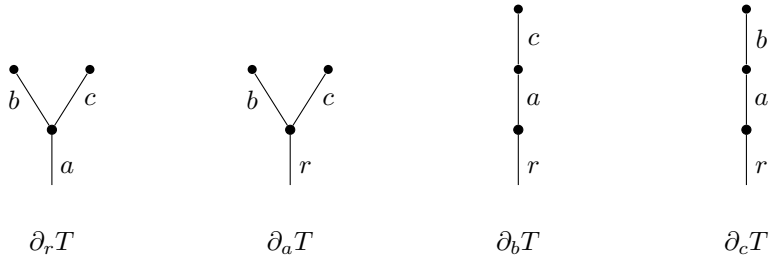
\[
	\begin{tikzpicture} 
	[level distance=8mm, 
	every node/.style={fill, circle, minimum size=.12cm, inner sep=0pt}, 
	level 1/.style={sibling distance=10mm}, 
	level 2/.style={sibling distance=10mm}, 
	level 3/.style={sibling distance=10mm}]
	
	\node(anchorR1)[style={color=white}] {} [grow'=up] 
	child {node(vertexR1)[draw] {} 
		child{node {}}
		child{node {}}
	};
	
	\node[style={color=white}, right=3cm of anchorR1] {} [grow'=up] 
	child {node(vertexR2)[draw] {} 
		child{node {}}
		child{node {}}
	};
	
	\node[style={color=white}, right=6cm of anchorR1] {} [grow'=up] 
	child {node(vertexR3)[draw] {} 
		child{node {}
			child{node {}}
		}
	};
	
	\node[style={color=white}, right=9cm of anchorR1] {} [grow'=up] 
	child {node(vertexR4)[draw] {} 
		child{node {}
			child{node {}}
		}
	};
	
	\tikzstyle{every node}=[]
	
	\node at ($(vertexR1) + (-.5cm,.4cm)$) {$b$};
	\node at ($(vertexR1) + (.5cm,.4cm)$) {$c$};
	\node at ($(vertexR1) + (.2cm,-.5cm)$) {$a$};
	\node at ($(vertexR1) + (0,-1.5cm)$) {$\partial_r T$};
	
	\node at ($(vertexR2) + (-.5cm,.4cm)$) {$b$};
	\node at ($(vertexR2) + (.5cm,.4cm)$) {$c$};
	\node at ($(vertexR2) + (.2cm,-.5cm)$) {$r$};
	\node at ($(vertexR2) + (0,-1.5cm)$) {$\partial_a T$};
	
	\node at ($(vertexR3) + (.2cm,1.2cm)$) {$c$};
	\node at ($(vertexR3) + (.2cm,.4cm)$) {$a$};
	\node at ($(vertexR3) + (.2cm,-.5cm)$) {$r$};
	\node at ($(vertexR3) + (0,-1.5cm)$) {$\partial_b T$};
	
	\node at ($(vertexR4) + (.2cm,1.2cm)$) {$b$};
	\node at ($(vertexR4) + (.2cm,.4cm)$) {$a$};
	\node at ($(vertexR4) + (.2cm,-.5cm)$) {$r$};
	\node at ($(vertexR4) + (0,-1.5cm)$) {$\partial_c T$};
	
	\end{tikzpicture} 
	\]
	The boundary $\partial T$ of $T$ as a dendroidal set is much larger, and also contains
	
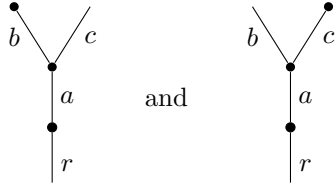
\[
	\begin{tikzpicture} 
	[level distance=8mm, 
	every node/.style={fill, circle, minimum size=.12cm, inner sep=0pt}, 
	level 1/.style={sibling distance=10mm}, 
	level 2/.style={sibling distance=10mm}, 
	level 3/.style={sibling distance=10mm}]
	
	\node(anchorR1)[style={color=white}] {} [grow'=up] 
	child {node(vertexR1)[draw] {} 
		child{node {}
			child{node {}}
			child
		}
	};
	
	\node[style={color=white}, right=3cm of anchorR1] {} [grow'=up] 
	child {node(vertexR2)[draw] {}
			child{node {}
			child
			child{node {}}
		}
	};
	
	\tikzstyle{every node}=[]
	
	\node at ($(vertexR1) + (.2cm,.4cm)$) {$a$};
	\node at ($(vertexR1) + (-.5cm,1.2cm)$) {$b$};
	\node at ($(vertexR1) + (.5cm,1.2cm)$) {$c$};
	\node at ($(vertexR1) + (.2cm,-.5cm)$) {$r$};
	\node at ($(vertexR1) + (1.5cm,.4cm)$) {and};
	
	\node at ($(vertexR2) + (.2cm,.4cm)$) {$a$};
	\node at ($(vertexR2) + (-.5cm,1.2cm)$) {$b$};
	\node at ($(vertexR2) + (.5cm,1.2cm)$) {$c$};
	\node at ($(vertexR2) + (.2cm,-.5cm)$) {$r$};
	
	\end{tikzpicture} 
	\] 
	In other words, for a closed tree $T$, the inclusion $\partial_{cl} T \rightarrowtail \partial T$ is a proper monomorphism, which we could also write as 
	\[
	u_!(\partial_{cl}T) \rightarrowtail \partial(u_!T).
	\]
	Note also that the closure of the boundary is larger than the closed boundary: 
	\[\partial_{cl}(T) \rightarrowtail cl_!(\partial u_! T ) \xrightarrow{\sim}  T.
	\]
\end{example}

With these closed boundaries in place, one can define a monomorphism $X\to Y$ between closed dendroidal sets to be \emph{normal} in one of the following equivalent ways: 
\begin{enumerate}
	\item $u_!X \to u_!Y$ is a normal monomorphism in $\dSets$;
	\item For every closed tree $T$, the group $\Aut(T)$ acts freely on the complement of $X_T \hookrightarrow Y_T$;
	\item $X\rightarrowtail Y$ lies in the saturation of the set of closed boundary inclusions $\partial_{cl}(T)\rightarrowtail T$, for all objects $T$ in $\Omega_{cl}$.
\end{enumerate}

The following nice property is somewhat in contrast with the case of $\dSets$, cf. \cite{HHM}. 

\begin{lemma} \label{lemma5.3}
	Let $U\rightarrowtail X$ and $V \rightarrowtail Y$ be normal monomorphisms of closed dendroidal sets. Then the pushout-product map
	\[
	U \otimes Y \cup X \otimes V \rightarrowtail X\otimes Y 
	\]
	is again a normal monomorphism. (The union sign $\cup$ denotes the pushout under $U\otimes V$.)
\end{lemma}

\begin{proof}
	See \cite{HMbook}, 4.22(ii).
\end{proof}
\section{Via morphisms} \label{sect7}

With an eye towards the next section, we recall the ``operadic'' model structure \cite{cisinskimoerdijk1} on $\dSets$, for which the cofibrations are the normal monomorphisms and the fibrant objects are the dendroidal inner Kan complexes, also referred to as (dendroidal) $\infty$-operads. The fibrations between fibrant object in this model structure are the maps $Y\to X$ having the RLP with respect to two kinds of morphisms: the inclusions 
\[
\Lambda^e T \rightarrowtail T 
\]
of inner horns, for any tree $T$ and any inner edge $e$ in $T$; and the two inclusions 
\[
\eta = i_! (\Delta[0]) \to i_!(J)
\]
where $J$ is the nerve of the groupoid $0 \leftrightarrow 1$. 

\begin{remark} If $P$ is any operad in $\Sets$, its nerve $N(P)$ is a fibrant object in this model structure. However, even if $P$ is unital, $P$ need not have the RLP with respect to ``closed inner horns'', i.e. inclusions of the form $\Lambda_{cl}^e (T)\rightarrowtail T$ where $T$ is a closed tree and $\Lambda_{cl}^e(T) \subseteq \partial_{cl} (T)$ is the union of all the closed faces except the one contracting $e$. For example, for $T$ a closed $2$-corolla as pictured below, $\Lambda_{cl}^e (T)$ just encodes a single unary operation: 
	
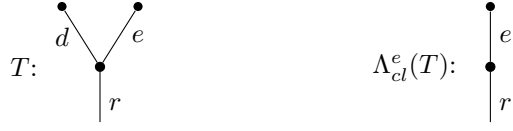
\[
\begin{tikzpicture} 
[level distance=8mm, 
every node/.style={fill, circle, minimum size=.12cm, inner sep=0pt}, 
level 1/.style={sibling distance=10mm}, 
level 2/.style={sibling distance=10mm}, 
level 3/.style={sibling distance=10mm}]

\node(anchorR1)[style={color=white}] {} [grow'=up] 
child {node(vertexR1)[draw] {} 
		child{node {}}
		child{node {}}
};

\node[style={color=white}, right=5cm of anchorR1] {} [grow'=up] 
child {node(vertexR2)[draw] {}
		child{node {}
	}
};

\tikzstyle{every node}=[]

\node at ($(vertexR1) + (-.5cm,.4cm)$) {$d$};
\node at ($(vertexR1) + (.5cm,.4cm)$) {$e$};
\node at ($(vertexR1) + (.2cm,-.5cm)$) {$r$};
\node at ($(vertexR1) + (-1cm,0cm)$) {$T$:};

\node at ($(vertexR2) + (.2cm,.4cm)$) {$e$};
\node at ($(vertexR2) + (.2cm,-.5cm)$) {$r$};
\node at ($(vertexR2) + (-1cm,0cm)$) {$\Lambda_{cl}^e(T)$:};

\end{tikzpicture} 
\] 
\end{remark}

Thus, we need to adapt the definition of ``inner horn'' in the context of closed dendroidal sets: 

\begin{definition}
	\begin{enumerate}
		\item[(a)] An edge $e$ in a closed tree $T$ is called \emph{very inner edge} if it is an inner edge which is not connected to a stump. 
		\item[(b)] The \emph{very inner horn} $\Lambda_{cl}^e T \subseteq T$ associated to a very inner edge $e$ in a closed tree $T$ is the union of all the closed faces (so $\Lambda^e T \subseteq \partial_{cl} T$) except the one contracting $e$. 
		\item[(c)] The saturation of the set of very inner horn inclusions, for all closed trees $T$, is called the class of \emph{very inner anodyne} (``via'') morphisms. 
	\end{enumerate}
\end{definition}	

Note that any via morphism is in particular a normal monomorphism. The following proposition shows that these via morphisms behave well with respect to tensor products of closed dendroidal sets, a fact which simplifies many things in comparison with the case of general dendroidal sets. 

\begin{proposition} \label{prop7-3}
	Let $U\rightarrowtail X$ and $V\rightarrowtail Y$ be normal monomorphisms between closed dendroidal sets. If at least one of these is a via morphism, then so is the pushout-product map 
	\[
	U \otimes Y \cup X \otimes V \rightarrowtail X\otimes Y. 
	\]
\end{proposition}

Before embarking on the proof of the proposition, we observe the following two lemmas. 

\begin{lemma} \label{lem7-1}
Let $E$ be a non-empty set of very inner edges in a closed tree $T$, and let $\Lambda_{cl}^E(T) \rightarrowtail \partial_{cl}(T)$ be the union of all the closed faces except the ones contracting an edge in $E$. Then $\Lambda_{cl}^E(T) \rightarrowtail T$ is a via morphism. 
\end{lemma}

\begin{proof}
The case where $E$ has only one element holds by definition. And a larger such set can be written as $E\cup \{d\}$ where $E$ is non-empty. Then the diagram 
\[
\xymatrix{
\Lambda_{cl}^{E\cup \{d\}}(T) \ar@{>->}[r]  & \Lambda_{cl}^E(T) \ar[r] \ & T \\
\Lambda_{cl}^E(\partial_d (T)) \ar[u] \ar@{>->}[r] & \partial_{d}(T) \ar[u] & 
}
\]
shows that $\Lambda_{cl}^{E\cup\{d\}}(T) \to T$ is again a via morphism, as the square is a pullback and a pushout, while the lower morphism is via by induction. 
\end{proof}

\begin{lemma} \label{lem7-2}
Let $E\subset T$ be a non-empty set of very inner edges in a closed tree $T$, and let $D$ be a set of inner edges in $T$, disjoint from $E$ and such that no edge $d\in D$ is immediately above any $e\in E$ (or more generally, such that each $e\in E$ is still very inner in $\partial_{d_1} \ldots \partial_{d_n} T$ for any $d_1, \ldots, d_n$ in $D$). Then $\Lambda_{cl}^{E\cup D}(T) \rightarrowtail T$ is a via morphism. 
\end{lemma}

\begin{proof}
We proceed by induction on the size of $D$. For $D=\emptyset$ this is the previous lemma. If $D=D'\cup \{d\}$ and  the lemma has been proved for $D'$ then the diagram, 
\[
\xymatrix{
	\Lambda_{cl}^{E\cup D}(T) \ar@{>->}[r]  & \Lambda_{cl}^{E\cup D'}(T) \ar@{>->}[r] \ & T \\
	\Lambda_{cl}^{E\cup D'}(\partial_d (T)) \ar@{>->}[u] \ar@{>->}[r] & \partial_{d}(T) \ar@{>->}[u] & 
}
\]in which the square is a pullback as well as a pushout, shows that $\Lambda_{cl}^{E\cup D}(T)\rightarrowtail T$ is via.
\end{proof}

Let us now turn to the proof of the proposition. 

\begin{proof}[Proof (of Proposition \ref{prop7-3})]
	The proof is a relatively straightforward modification of the one of \cite{HMbook}, 6.2.4. Observe first that the pushout-product map is a normal monomorphism by Lemma \ref{lemma5.3}. Next, a standard induction along saturated classes reduces the problem to the case of two closed boundary inclusions of representables. So let us suppose that $U\rightarrowtail X$ and $V\rightarrowtail Y$ are of the form 
	\[
	\Lambda_{cl}^e(S) \rightarrowtail S \quad \text{and} \quad \partial_{cl}(T) \rightarrowtail T
	\]
	respectively for closed trees $S$ and $T$ and a very inner edge $e$ in $S$. Now write $S\otimes T$ as a union of shuffles $R_i\subseteq S\otimes T$, and order these as 
	\[
	R_1, \ldots, R_N
	\]
	by a linear order which extends the natural partial order (in which ``copies of $S$ on top of $T$'' is the smallest shuffle and ``copies of $T$ on top of $S$'' is the largest). This defines a filtration of $S\otimes T$ as 
	\[
	A_0 \subseteq A_1 \subseteq A_2 \subseteq \ldots \subseteq A_N = S\otimes T, 
	\]
	where 
	\[A_0=\Lambda_{cl}^{e}(S)\otimes T \cup S\otimes \partial_{cl}(T), \]
	and \[A_i=A_{i-1}\cup R_i.\]
	One then shows that each $A_{i-1}\rightarrowtail A_i$ is a via map. To this end, consider the shuffle $R=R_i$. It has edges $(e,t)$ for the given edge $e$ in $S$ and various edges $t$ in $T$. Call the highest occurrences of these edges \emph{special}. These are edges with an $S$-vertex immediately above it in the shuffle $R$: 
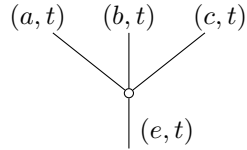
	\[
	\begin{tikzpicture} 
	[level distance=8mm, 
	every node/.style={fill, circle, minimum size=.12cm, inner sep=0pt}, 
	level 1/.style={sibling distance=10mm}, 
	level 2/.style={sibling distance=10mm}, 
	level 3/.style={sibling distance=10mm}]
	
	\node(anchorR1)[style={color=white}] {} [grow'=up] 
	child {node(vertexR1)[draw, fill=none] {} 
		child
		child
		child
	};

	\tikzstyle{every node}=[]
	
	\node at ($(vertexR1) + (-1.2cm,1cm)$) {$(a,t)$};
	\node at ($(vertexR1) + (0cm,1cm)$) {$(b,t)$};
	\node at ($(vertexR1) + (1.2cm,1cm)$) {$(c,t)$};
	\node at ($(vertexR1) + (.5cm,-.5cm)$) {$(e,t)$};
	
	\end{tikzpicture} 
	\]   
	Let $\Sigma=\Sigma_R$ be the set of these special edges in $R$. Notice that these are all very inner in $R$ as $e$ is very inner in $S$. For a subset $H$ of inner edges disjoint from $\Sigma$, let $R^{[H]}$ be the face obtained by contracting all the inner edges in $R$ \emph{except} the ones in $H\cup \Sigma$. We then adjoin these $R^{[H]}$ to $A_{i-1}$ in some order extending the inclusion order of the $H$'s. So consider a specific $R^{[H]}$, and suppose all $R^{[H']}$ for strictly smaller $H'\subseteq H$ have already been adjoined, yielding an intermediate closed dendroidal set $A_{i-1} \subseteq B \subseteq A_i$. We wish to adjoin $R^{[H]}$ to $B$. If $R^{[H]}\subseteq B$ already, then there is nothing to do. This holds in particular if none of the special edges $(e, t)$ in $R^{[H]}$ is very inner in $R^{[H]}$ (because then the $S$-colours above $e$ have disappeared in $R^{[H]}$, so $R^{[H]}\subseteq A_0$).
	
	Now let us consider closed faces of $R^{[H]}$. The root face of $R^{[H]}$ (if it exists) must miss the root edge of $S$ or that of $T$, hence must be contained in $A_0$. (Note that $R^{[H]}$ may have a root face while $S$ and $T$ do not.) Any non-special inner face of $R^{[H]}$ is contained in $R^{[H']}$ for a smaller $H'$, hence in $B$. If $(e, t)$ is a special edge in $R^{[H]}$, on the other hand, then the face $\partial_{e\otimes t} R^{[H]}$ cannot belong to an earlier $R^{[H']}$. And it can not belong to an earlier shuffle $R_j$ ($j<i$) unless $R^{[H]}$ itself already does (in which case we are back in the earlier case where $R^{[H]}\subset B$ already). Similarly, if $\partial_{(e, t)}(R^{[H]})$ would be contained in $\Lambda_{cl}^e(S) \otimes T \subseteq A_0$, then so would $R^{[H]}$. The only remaining case in which $\partial_{(e, t)}(R^{[H]})$ is contained in $B$ is where the edge $t$ has disappeared entirely, so that $\partial_{(e, t)} (R^{[H]}) \subseteq S\otimes \partial_{cl} (T)\subseteq A_0$. This cannot happen if $(e, t)$ is \emph{very inner} in $R^{[H]}$ however, because then $t$ still occurs immediately above (in $(e, t)$). 
	
	The conclusion is that $B\cap R^{[H]}$  contains the root face of $R^{[H]}$ and all inner faces except the one contracting a very inner edge $(e, t)$, as well as some of those contracting just an inner edge $(e, t')$. The same applies to intersection of these (as they involve different $t$'s). It follows by Lemma \ref{lem7-2} that $B\cap R^{[H]} \rightarrowtail R^{[H]}$ is a via morphism, and hence so is its pushout $B\rightarrowtail B\cup R^{[H]}$. 
	
	This completes the induction step, and proves the proposition.
\end{proof}

It will be necessary to formulate the precise relation between what are called ``Segal cores'' in \cite{cisinskimoerdijk2, HHM} and ``spines'' in \cite{HMbook}, and via morphisms. The treatment is to a large extent analogous to the one in these two references, but does not seem to be a formal consequence of either (and, in fact, seems a bit easier for closed trees). 

If $T$ is a closed tree, each vertex $v$ in $T$ defines a map 
\[
c_v \colon \overline{C}_v \rightarrowtail T
\] 
from a closed corolla $\overline{C}_v$ to $T$, which maps the edges of $\overline{C}_v$ to the edges in $T$ attached to $v$ in a bijective fashion. (This map $\overline{C}_v$ is a composition of inner faces contracting all the edges not attached to $v$ and not on the path from $v$ down to the root, followed by a (possibly empty) composition of root faces.) The intersection of two such corollas $\overline{C}_v\rightarrowtail T$ and $\overline{C}_w \rightarrowtail T$ in $T$ is either empty, or if there is an edge connecting $v$ and $w$, is the map $\overline{\eta}\to T$ corresponding to this edge $e$. The \emph{closed spine} 
\[
\cSpine(T) \rightarrowtail T
\]
is by definition the union of all these corollas $\overline{C}_v\to T$ ranging over all vertices $v$ in $T$. (If $T\neq \overline{\eta}$, it is enough to consider inner vertices in $T$ only, since for a stump $v$ in $T$ the corolla $\overline{C}_v\cong \overline{\eta}$ is contained in the corolla $\overline{C}_w$ for the vertex $w$ immediately below $v$.) Thus, if $T$ is itself a corolla, then $\cSpine(T)\to T$ is an isomorphism. The following theorem is analogous to \cite{cisinskimoerdijk2}, Propositions 2.4 and 2.5.

\begin{theorem} \label{theorem7-6}
	\begin{enumerate}
		\item[(a)] Let $T$ be a closed tree. Then $\cSpine(T)\to T$ is a via morphism. 
		\item[(b)] The class of via morphisms is the smallest saturated class containing all the closed spine inclusions, and which moreover has the property that if a composition $A\rightarrowtail B\rightarrowtail C$ of normal monomorphisms as well as $A\rightarrowtail B$ belong to the class then so does $B\rightarrowtail C$. 
	\end{enumerate}
\end{theorem}

\begin{proof}
	Let $T$ be a tree with at least one very inner edge, and let $E$ be the set of \emph{all} very inner edges in $T$. We will show that \begin{equation} \label{eq7-1}
	\cSpine(T) \to \Lambda^E (T)
	\end{equation}
	is a via morphism, \emph{and} that it belongs to the saturation of the class of closed spine inclusions. Since $\Lambda^E T \to T$ is a via morphism (cf. Lemma \ref{lem7-1}), this proves that the composition $\cSpine(T)\to T$ is a via morphism as well, thus proving Part (a).
	
	And by applying the cancellation condition in Part (b) of the theorem to the composition $\cSpine(T)\to \Lambda^E(T) \to T$, it proves that $\Lambda^E T \to T$ belongs to the class. This holds for all closed trees, in particular for all faces of $T$. Now write $E=\{e_1, \ldots, e_n\}$ and $E_i=\{e_1, \ldots, e_i\}$ for $i=1,\ldots, n$. Then the pushout in Lemma \ref{lem7-1} shows that each of the morphisms $\Lambda^{E_{i+1}} (T) \to \Lambda^{E_i}(T)$ belongs to the class (by induction on the size of the set of very inner edges in face of $T$). By the cancellation property, $\Lambda^{E_1}(T)\to T$ belongs to the class, proving Part (b).
	
	Consider a set $A$ of edges in $T$, disjoint from $E$. Write $\partial_a T$ for the face corresponding to an edge $a\in A$. (If $a$ is the root edge, we use this notation only if the root vertex is unary.) Also write 
	\[
	\partial_A T = \bigcup_{a\in A} \partial_a T.
	\]
	It now suffices to show for each such set $A$ that 
	\begin{equation} \label{eq7-2}
	\cSpine(T) \rightarrowtail \cSpine(T)\cup \partial_A(T)
	\end{equation}
	belongs to the saturation of the closed spine inclusions \emph{and} is a via morphism. Indeed, if $A$ is maximal then $\cSpine(T) \subseteq \partial_A(T)=\Lambda^E (T)$, so we conclude that $\cSpine(T) \rightarrowtail \Lambda^E(T)$ belongs to this saturation \emph{and} is a via morphism, as was to be shown. 
	
	We argue by induction on $T$ and $A$. The minimal case is where $T=\overline{[2]}$, a tree with just one very inner edge: 
	
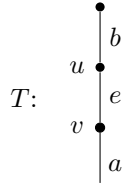
\[
	\begin{tikzpicture} 
	[level distance=8mm, 
	every node/.style={fill, circle, minimum size=.12cm, inner sep=0pt}, 
	level 1/.style={sibling distance=10mm}, 
	level 2/.style={sibling distance=10mm}, 
	level 3/.style={sibling distance=10mm}]
	
	\node(anchorR1)[style={color=white}] {} [grow'=up] 
	child {node(vertexR1)[draw] {} 
		child{node {}
			child{node {}}
		}
	};

	\tikzstyle{every node}=[]
	
	\node at ($(vertexR1) + (.2cm,1.2cm)$) {$b$};
	\node at ($(vertexR1) + (.2cm,.4cm)$) {$e$};
	\node at ($(vertexR1) + (.2cm,-.5cm)$) {$a$};
	\node at ($(vertexR1) + (-.3cm,0cm)$) {$v$};
	\node at ($(vertexR1) + (-.3cm,.8cm)$) {$u$};
	\node at ($(vertexR1) + (-1cm,.4cm)$) {$T$:};
	
	\end{tikzpicture} 
	\] 
	Then $\cSpine(T)\to \Lambda^E(T)$ is an isomorphism, as $\overline{C}_u=\partial_a T$ and $\overline{C}_v=\partial_b T$. For a larger tree $T$, suppose we have shown that (\ref{eq7-2}) is a via morphism for all smaller trees and smaller sets $A$, and write $A=A'\cup \{a\}$ for $a\not\in A'$. Then there is a pushout 
	\[
	\xymatrix{
	\cSpine(T) \cup \partial_{A'} (T) \ar@{>->}[r] & \cSpine(T) \cup \partial_A(T) \\
	(\cSpine(T) \cup \partial_{A'} (T)) \cap \partial_a(T) \ar[u] \ar@{>->}[r] & \partial_a(T) \ar[u] }
	\]
	and one easily checks that 
	\[
	\partial_{A'} (T) \cap \partial_a T = \partial_{A'}\partial_a T 
	\]
	and also 
	\[
	\cSpine(T) \cap \partial_a (T) = \cSpine(\partial_a T)
	\]
	(remember that $a\not\in E$, so $a$ does not connect two inner vertices). It follows by induction that the lower map in the diagram belongs to the class, and is via. But then the same is true for the upper one, and hence for the composition
	\[
	\cSpine(T) \rightarrowtail \cSpine(T) \cup \partial_{A'} (T) \rightarrowtail \cSpine(T) \cup \partial_A (T),
	\]
	which is the map (\ref{eq7-2}). This completes the induction, and the proof of the theorem. 
\end{proof}

\section{Unital $\infty$-operads} \label{sect8}

Recall from \cite{moerdijkweiss, HMbook} the inclusion $\Omega\to \Oper$, with an induced adjunction 
\[
\xymatrix@C=40pt{
	\tau : \dSets \ar@<.5ex>[r] & \Oper_{\Sets} : N \ar@<.5ex>[l] } 
\]
between dendroidal sets and operads in $\Sets$, i.e. discrete operads. For such an operad $P$, the dendroidal set $N(P)$ is called its \emph{dendroidal nerve}. This adjunction restricts to an adjunction, again denoted 
\[
\xymatrix@C=40pt{
	\tau : \cdSets \ar@<.5ex>[r] & \cOper_{\Sets} : N \ar@<.5ex>[l] } 
\]
between closed dendroidal sets and discrete closed (or unital) operads. For a general operad $P$, its nerve $N(P)$ is an $\infty$-operad in the sense of having the RLP with respect to all inner horn inclusions between arbitrary trees. For a unital operad $P$, the closed dendroidal set $N(P)$ similarly has the RLP with respect to all very horn inclusions of closed trees. (It clearly need not have the RLP with respect to all inner horn inclusions of closed trees. For example, the inner horn $\Lambda^1(\overline{C}_2)$ of the closed $2$-corolla 
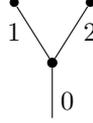
\[
\begin{tikzpicture} 
[level distance=8mm, 
every node/.style={fill, circle, minimum size=.12cm, inner sep=0pt}, 
level 1/.style={sibling distance=10mm}, 
level 2/.style={sibling distance=10mm}, 
level 3/.style={sibling distance=10mm}]

\node(anchorR1)[style={color=white}] {} [grow'=up] 
child {node(vertexR1)[draw] {} 
	child{node {}}
	child{node {}}
};

\tikzstyle{every node}=[]

\node at ($(vertexR1) + (-.5cm,.4cm)$) {$1$};
\node at ($(vertexR1) + (.5cm,.4cm)$) {$2$};
\node at ($(vertexR1) + (.2cm,-.5cm)$) {$0$};
\end{tikzpicture} 
\]
is a copy of the closed $1$-simplex
\[
\begin{tikzpicture} 
[level distance=8mm, 
every node/.style={fill, circle, minimum size=.12cm, inner sep=0pt}, 
level 1/.style={sibling distance=10mm}, 
level 2/.style={sibling distance=10mm}, 
level 3/.style={sibling distance=10mm}]

\node(anchorR1)[style={color=white}] {} [grow'=up] 
child {node(vertexR1)[draw] {} 
	child{node {}}
};

\tikzstyle{every node}=[]

\node at ($(vertexR1) + (.2cm,.4cm)$) {$1$};
\node at ($(vertexR1) + (.2cm,-.5cm)$) {$0$};
\end{tikzpicture} 
\]
and the RLP of $N(P)$ against $\Lambda^1(\overline{C}_2)\rightarrowtail \overline{C}_2$ would require any unary operation to extend to a binary one in $P$.)

\begin{definition}
	A closed dendroidal set $E$ is called a \emph{unital $\infty$-operad} if it has the RLP with respect to very inner horn inclusions $\Lambda^e T \rightarrowtail T$, for every very inner edge $e$ in a closed tree $T$; or equivalently, if it has the RLP with respect to every via morphism. 
\end{definition}

Thus, for every discrete unital operad $P$ its ``closed'' nerve $N(P)$ is such a unital $\infty$-operad. Later, we will see that a suitable \emph{homotopy coherent} nerve of a unital simplicial operad is again a unital $\infty$-operad. 

We will prove in the next section that these unital $\infty$-operads are the fibrant objects in a model category structure on closed dendroidal sets, cf. Theorem \ref{teo9-1} The main technical result that we need for this is the following theorem, Theorem \ref{theo8-1} below. To state this theorem, let us denote for two closed dendroidal sets $X$ and $Y$ the simplicial hom-set by $\hhom(X,Y)$. Thus, for any $n\geq 0$, 
\[
\hhom(X,Y)_n= \Hom_\cdSets(\Delta[n] \otimes X, Y).
\]
(Recall that in terms of the tensor product on the category $\cdSets$, $\Delta[n]\otimes X$ is defined as $cl_! i_!(\Delta[n]) \otimes X$, cf. Section \ref{sect6}) With this notation, Proposition \ref{prop7-3} has as a consequence for closed dendroidal sets $A$ and $E$ that if $A$ is normal and $E$ is a unital $\infty$-operad, then $\hhom(A,E)$ is an $\infty$-category. More generally, we have the following result.

\begin{theorem} \label{theo8-1}
	Let $E$ be a unital $\infty$-operad. Then for every normal monomorphism $A\rightarrowtail B$ in $\cdSets$, the restriction map 
	\[
	\hhom(B, E) \to \hhom(A, E) 
	\]
	is a fibration between $\infty$-categories. 
\end{theorem}

\begin{proof}
	For a simplex $\Delta[n]$, any inner face $\partial_i\colon \Delta[n-1]\rightarrowtail \Delta[n]$ (i.e., $0 < i < n$) becomes a very inner face upon closure. So for any inner horn $\Lambda^i[n] \rightarrowtail \Delta[n]$, the map $cl_!(\Lambda^i[n]) \rightarrowtail cl_!(\Delta[n])$ is a via morphism. It then follows from Proposition \ref{prop7-3} that $\hhom(B, E)$ and $\hhom(A, E)$ are $\infty$-categories, and that the map $\hhom(B,E)\to \hhom(A,E)$ is an inner fibration. It thus remains to be shown that this map has the RLP with respect to the inclusion $\{1\}\rightarrowtail J$ of one of the endpoints into the ``interval'' $J$, the nerve of the groupoid $0\leftrightarrow 1$ viewed as a dendroidal set. In other words, using an induction on normal monomorphisms, we have to show that for any closed tree $S$, every unital $\infty$-operad $E$ has the RLP with respect to the map 
	\[
	\{1\}\otimes S \cup J \otimes \partial_{cl}(S) \rightarrowtail J\otimes S.
	\]
	By basic properties of maps into $\infty$-categories (\cite{joyal}, Corollary 1.6 and \cite{lurie}, Section 1.2.5) it in fact suffices to prove that every unital $\infty$-operad $E$ has the extension property with respect to 
	\begin{equation} \label{eq8-1}
	\{1\}\otimes S \cup \Delta[1]\otimes \partial_{cl}(S) \rightarrowtail \Delta[1]\otimes S.
	\end{equation}
	for maps which send each copy $\Delta[1]\otimes s$ (for $s$ an edge in $S$) of $\Delta[1]$ to an equivalence in $E$. Notice that for $S=\overline{\eta}$ this map (\ref{eq8-1}) is a retract, so there is nothing to prove. For larger $S$, this extension property follows from the following two lemmas, which then complete the proof of the theorem. 
\end{proof}

The first of these two lemmas is an analogue of \cite{cisinskimoerdijk1, HMbook} although the proof for closed trees is a bit easier: 

\begin{lemma}
	Let $S$ be a closed tree with at least two vertices. Then the inclusion (\ref{eq8-1}),
	\[
	\{1\}\otimes S \cup \Delta[1]\otimes \partial_{cl}(S) \rightarrowtail \Delta[1]\otimes S.
	\]
	is a composition of a finite number of via morphisms followed by the pushout of a unary root horn into a closed tree with at least three vertices (with unary root vertex corresponding to the vertex of $\Delta[1]\otimes r$ for the root edge $r$ in $S$). 
\end{lemma}

	(The tree with at least three vertices occurring in the lemma is the tree obtained by grafting $S$ on top of $\Delta[1]$.)

\begin{proof}
	Let $A=\{1\}\otimes S \cup \Delta[1]\otimes \partial_{cl}(S)$, the domain of the map in the lemma. Let us also write $\Delta[1]\otimes S$ as a union of shuffles $R_1, \ldots, R_n$, where the first shuffle is ``copies of $\Delta[1]$ on top of $S$'' and the last one is ``a copy of $S$ on top of $\Delta[1]$''. For example, if $S$ is the closed $2$-corolla as pictured on the left, there are only two shuffles as pictured on the right: 
	
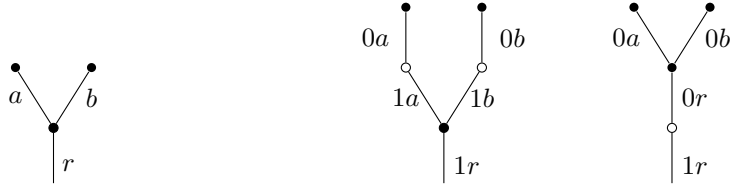
\[
	\begin{tikzpicture} 
	[level distance=8mm, 
	every node/.style={fill, circle, minimum size=.12cm, inner sep=0pt}, 
	level 1/.style={sibling distance=10mm}, 
	level 2/.style={sibling distance=10mm}, 
	level 3/.style={sibling distance=10mm}]
	
	\node(anchorR1)[style={color=white}] {} [grow'=up] 
	child {node(vertexR1)[draw] {} 
		child{node {}}
		child{node {}}
	};
	
	\node[style={color=white}, right=5cm of anchorR1] {} [grow'=up] 
	child {node(vertexR2)[draw] {} 
		child{node[draw, fill=none] {}
			child{node {}}}
		child{node[draw, fill=none] {}
			child{node {}}}
	};
	
	\node[style={color=white}, right=8cm of anchorR1] {} [grow'=up] 
	child {node(vertexR3)[draw, fill=none] {} 
		child{node {}
			child{node {}}
			child{node {}}
		}
	};

	\tikzstyle{every node}=[]
	
	\node at ($(vertexR1) + (-.5cm,.4cm)$) {$a$};
	\node at ($(vertexR1) + (.5cm,.4cm)$) {$b$};
	\node at ($(vertexR1) + (.2cm,-.5cm)$) {$r$};

	\node at ($(vertexR2) + (-.5cm,.4cm)$) {$1a$};
	\node at ($(vertexR2) + (.5cm,.4cm)$) {$1b$};
	\node at ($(vertexR2) + (-.9cm,1.2cm)$) {$0a$};
	\node at ($(vertexR2) + (.9cm,1.2cm)$) {$0b$};
	\node at ($(vertexR2) + (.3cm,-.5cm)$) {$1r$};

	\node at ($(vertexR3) + (-.6cm,1.2cm)$) {$0a$};
	\node at ($(vertexR3) + (.6cm,1.2cm)$) {$0b$};
	\node at ($(vertexR3) + (.3cm,.4cm)$) {$0r$};
	\node at ($(vertexR3) + (.3cm,-.5cm)$) {$1r$};
	\end{tikzpicture} 
	\] 
	Consider the filtration of $\Delta[1]\otimes S$, 
	\[
	A=B_0 \subseteq B_1 \subseteq \ldots \subseteq B_n = \Delta[1]\otimes S, 
	\]
	defined by $B_i = A\cup R_1 \cup \ldots \cup R_i$. Then one easily sees that the last inclusion $B_{n-1}\hookrightarrow B_n$ is a pushout of a unary root horn. Indeed, for the last shuffle $R_n$, the face contracting the edge $(0,r)$ just above the unary root belongs to $R_{n-1}$, and the higher faces belong to $A$, while the face chopping of the root vertex and root edge $(1,r)$ is the only one that is missing. 
	
	We claim that for each $i<n$, the inclusion $B_{i-1}\hookrightarrow B_i$ is a via morphism. To see this, consider the shuffle $R=R_i$. Now consider ``special'' faces $F\subseteq R$ obtained by successively contracting edges with colour $(0,a)$, for an edge in $S$. These special faces are ordered by inclusion, and we can adjoin them successively to $B_{i-1}$. Suppose for a given such special face $F$, all smaller special faces have already been adjoined, while $F$ is not contained in $A$. Write $C$ for the union of $B_{i-1}$ and these smaller special faces. Then each face of $F$ contracting an edge $(0,a)$ belongs to $C$. There must be at least one such face, because otherwise $F$ would already be contained in $A$. This implies in particular that at least one of the highest occurrences of an edge coloured $(1,a)$ for some $a$ is very inner. Contracting any such highest very inner edge $(1,a)$ results in a face of $F$ which cannot belong to $A$, nor to an earlier shuffle. On the other hand, contracting an edge $(1,a)$ without an edge $(0,a)$ above it in $F$ yields a face belonging to $A$. This shows that $F\cap C\rightarrowtail F$ is a via morphism (cf. Lemma \ref{lem7-2}), and hence so is its pushout $C\to C\cup F$. This completes the induction step, and the proof of the lemma.
\end{proof}

\begin{lemma}
	Let $T$ be a close tree with at least three vertices, and a unary root vertex. Then the inclusion $\Lambda_{cl}^r(T)\rightarrowtail T$ of the root horn has the extension property for maps into a unital $\infty$-operad $E$ which sends the unary root vertex to an equivalence. 
\end{lemma}

\begin{proof}
	The proof is completely analogous to the one in \cite{cisinskimoerdijk1} or \cite{HMbook}, and we only give a sketch. (The main differences lie in distinguishing the closed boundary from the full  dendroidal boundary, and the observation to be made below that joins of trees give very anodyne maps.)
	
	For a tree $T$ in the lemma, one can write $T$ as the join $T=F\star\Delta[1]$, where $F$ is the forest obtained by chopping off the lower two vertices in $T$: 
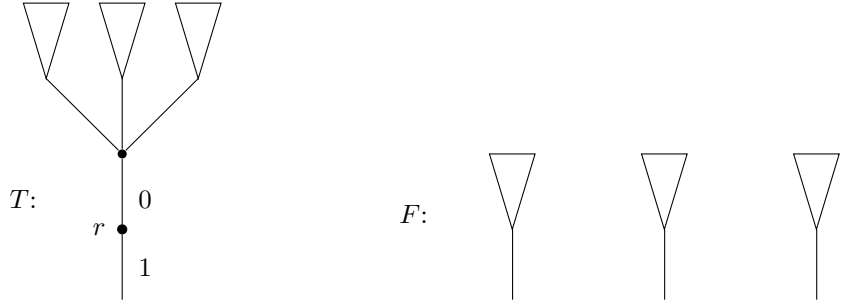
	\[
	\begin{tikzpicture} 
	[level distance=10mm, 
	every node/.style={fill, circle, minimum size=.12cm, inner sep=0pt}, 
	level 1/.style={sibling distance=10mm}, 
	level 2/.style={sibling distance=10mm}, 
	level 3/.style={sibling distance=10mm}]
	
	\node(anchorR1)[style={color=white}] {} [grow'=up] 
	child {node(vertexR1)[draw] {} 
		child{node {}
		child
		child
		child
	}
	};

	\draw (-1.3,4) -- (-1,3) -- (-.7,4) -- cycle;
	\draw (-.3,4) -- (0,3) -- (.3,4) -- cycle;
	\draw (.7,4) -- (1,3) -- (1.3,4) -- cycle;
	
	\node[style={color=white}, right=5cm of anchorR1] {} [grow'=up] 
	child {
	};

	\draw (4.83,2) -- (5.13,1) -- (5.43,2) -- cycle;
	
	\node[style={color=white}, right=7cm of anchorR1] {} [grow'=up] 
	child {
	};
	
	\draw (6.83,2) -- (7.13,1) -- (7.43,2) -- cycle;
	
	\node[style={color=white}, right=9cm of anchorR1] {} [grow'=up] 
	child {
	};
	
	\draw (8.83,2) -- (9.13,1) -- (9.43,2) -- cycle;

	\tikzstyle{every node}=[]
	
	\node at ($(vertexR1) + (-1.3cm,.4cm)$) {$T$:};
	\node at ($(vertexR1) + (.3cm,.4cm)$) {$0$};
	\node at ($(vertexR1) + (.3cm,-.5cm)$) {$1$};	
	\node at ($(vertexR1) + (-.3cm,0cm)$) {$r$};	
	
	\node at ($(vertexR2) + (-1.3cm,.4cm)$) {$F$:};
	\end{tikzpicture} 
	\]

	One next translates an extension problem as in the diagram on the left into one as on the right: 
	\[
	\xymatrix{\Lambda^r T \ar[d] \ar[r] & E \\ T \ar@{.>}[ru]}
	\qquad \qquad \qquad
	\xymatrix{\{1\} \ar[d] \ar[r] & F/E \ar[d] \\ \Delta[1] \ar[r] \ar@{.>}[ru] & \partial_{cl}(F)/E}
	\]
	where the slice denotes the adjoint to the join, as in \cite{cisinskimoerdijk1}, and $\partial_{cl}(F)$ is the closed forest boundary of $F$. One next observes that $F/E \to \partial_{cl}(F)/E$ is a left fibration. Indeed, for the inclusion $\partial_{cl}(F)\rightarrowtail F$ into a non-empty forest, the map 
	\[
	F\star \Lambda^i[n] \cup \partial_{cl}(F) \star \Delta[n] \rightarrowtail F\star \Delta[n]
	\]
	is a via morphism for $0< i < n$ (cf. \cite{HMbook}, Lemma 6.4.3). One then finishes the proof exactly as in loc. cit., observing that $\Delta[1]\to \partial_{cl}(F)/E$ is an equivalence in $\partial_{cl}(F)/E$ and using that left fibrations have the RLP with respect to such equivalences. 
\end{proof}

\section{The unital operadic model structure for closed dendroidal sets} \label{sect9}

In this section we present a model structure on the category of $\cdSets$ of closed dendroidal sets, suitable for a comparison with unital operads in Section \ref{sect13} below. The following two theorems summarize the main aspects of this model structure. 

\begin{theorem} \label{teo9-1}
	The category $\cdSets$ of closed dendroidal sets carries a Quillen model structure in which the cofibrations are the normal monomorphism while the fibrant objects are precisely the unital $\infty$-operads. 
\end{theorem}

We shall refer to this model structure as the \emph{unital operadic model structure} on $\cdSets$. 

\begin{proof}
	The proof of the model structure follows the same pattern as the one in \cite{HMbook}. The cofibrations are generated by the closed boundary inclusions $\partial_{cl}(T)\rightarrowtail T$, for all closed trees $T$. The small object argument then gives a factorization of any map $X\to Y$ into a cofibration $X\to Z$ followed by a map $Z\to Y$ having the RLP with respect to all the cofibrations. In particular, for $X=\emptyset$, this yields a \emph{normalization} $\tilde{Y} \twoheadrightarrow Y$ of any object $Y$. One then defines a map $f\colon A\to B$ to be a weak equivalence iff for some map $\tilde{f}$ between normalizations which covers $f$, as in 
	\[
	\xymatrix{\tilde{A} \ar@{->>}[d] \ar[r]^{\tilde{f}}  & \tilde{B} \ar@{->>}[d]	 \\
			A \ar[r]^{f} & B }
	\] 
	the induced morphism $\hhom(\tilde{B}, E)\to \hhom(\tilde{A}, E)$ is a weak equivalence of $\infty$-categories, for \emph{any} unital operad $E$. This definition is independent of the choice of the normalization $\tilde{f}\colon \tilde{A} \to \tilde{B}$ of $f\colon A \to B$. 
	
	Next, one shows that a map $Y\to X$ having the RLP with respect to all the normal monomorphisms (i.e., all the cofibrations) is a weak equivalence. Indeed, for a normalization $\tilde{X}\to X$ the pullback $\tilde{Y} = \tilde{X} \times_{X} Y \to Y$ is a normalization of $Y$ and $\tilde{Y} \to \tilde{X}$ still has the RLP with respect to all normal monomorphisms. It follows that this map has a section, making $\tilde{X}$ a $J$-deformation retract of $\tilde{f}$. This makes $\hhom(\tilde{X}, E) \to \hhom(\tilde{Y}, E)$ a deformation retraction for any unital $\infty$-operad $E$, proving that $Y\to X$ is a weak equivalence. 
	
	The next thing to prove is that any map can be factored as a trivial cofibration followed by a map having the RLP with respect to all the trivial cofibrations. For this, one again uses the small object argument and shows first that the trivial cofibrations are generated by trivial cofibration between countable (and normal) objects, exactly as in \cite{HHM} and \cite{moerdijklectures}, and that they are stable under pushout. The latter follows readily from the fact that a pushout square can be covered by another pushout square between normalizations, hence maps to a pullback suqare of $\infty$-categories after applying $\hhom(-, E)$ for a unital $\infty$-operad $E$, exactly as in \cite{HHM}, Lemma 3.7.13. 
	
	Finally, one \emph{defines} the fibrations as the maps having the RLP with respect to all trivial cofibrations. The proof can then be completed in the standard way, using the retract argument for the verification of one of the lifting axioms. 
	\end{proof}

It will be useful to explicitly state some properties of this model structure: 

\begin{theorem} \label{teo9-2}
	The unital operadic model structure on $\cdSets$ has the following properties: 
	\begin{enumerate}
		\item[(a)] A map $A\to B$ between normal objects is a weak equivalence iff for every unital $\infty$-operad $E$, the map $\hhom(B, E)\to \hhom(A, E)$ is a weak equivalence of $\infty$-categories. 
		\item[(b)] The model structure is cofibrantly generated. 
		\item[(c)] The model structure is left proper.
		\item[(d)]\label{the9-2d} A morphism $Y\to X$ between unital $\infty$-operads is a fibration iff it has the RLP with respect to all via morphisms as well as with respect to the two inclusions $\overline{\eta}\to cl_!(J)$ (where $J$ is the nerve of the gruopoid $0\leftrightarrow 1$ viewed as a dendroidal set).
		\item[(e)] The pushout-product property holds: if $U\rightarrowtail X$ and $V\rightarrowtail Y$ are cofibrations the so is $U\otimes Y\cup X\otimes V \rightarrowtail X\otimes Y$; and in addition, the latter is a trivial cofibration if at least one of $U\rightarrowtail X$ or $V\rightarrowtail Y$ is.
	\end{enumerate}
\end{theorem}

\begin{proof}
	The additional properties (a)-(d) are immediate from the way the model structure was established, again exactly as in \cite{HHM, moerdijklectures}. Part (e) now follows form these properties together with Proposition \ref{prop7-3}. Indeed, to prove that $U\otimes Y\cup X\otimes V \rightarrowtail X\otimes Y$ is a trivial cofibration if, say, $U\rightarrowtail X$ is, we use the general fact that a map is a trivial cofibration iff it has the RLP with respect to fibration \emph{between fibrant objects}. Moreover, using an induction on $V\rightarrowtail Y$, it suffices to consider the case where $V\to Y$ is a closed boundary inclusion, of the form $\partial_{cl} T \rightarrowtail T$ for a closed tree $T$. It then suffices to show that if $E\to B$ is a fibration between fibrant objects, so is 
	\[
	\HHom(T, E)\to \HHom(\partial_{cl} T, E) \times_{\HHom(\partial_{cl} T, B)} \HHom(T, B) 
	\]
	(where $\HHom$ is the internal $\Hom$ adjoint to $\otimes$). Using Part (d) and the $\Hom$-tensor adjunction again, it then suffices to show that 
	\[
	U\otimes T\cup X\otimes \partial_{cl} T \rightarrowtail X\otimes T
	\]
	is a trivial cofibration in the two cases where (i) $U\rightarrowtail X$ is a very inner horn inclusion of the form $\Lambda_{cl}^{d} S \rightarrowtail S$, say; or (ii) $U\rightarrowtail X$ is of the form $\overline{\eta} \to cl_!(J)$. Case (i) follows form Proposition \ref{prop7-3}. Case (ii) follows since $\overline{\eta}\otimes T \cup cl_!(J) \otimes \partial_{cl} T \rightarrowtail cl_!(J) \otimes T$ has been shown to be a trivial cofibration in Section \ref{sect8}.
\end{proof}

\begin{remark} \label{remark9-3}
	One important advantage of the category of \emph{closed} dendroidal sets over that of all dendroidal sets is that for all practical purposes, the former behaves like a \emph{monoidal model category}, the only defect being that the tensor product is only associative up to a trivial cofibration. More precisely, if $X$, $Y$ and $Z$ are normal objects, then the canonical map (cf. \cite{HHM}, Section 6.3) 
	\[
	(X\otimes Y)\otimes Z \to X\otimes Y \otimes Z
	\]
	is a trivial cofibration. Indeed, by a standard induction on cofibrant objects, one can reduce the problem to the case where $X$, $Y$ and $Z$ are representable. In other words, we claim that for closed trees $R$, $S$ and $T$, the map $(R\otimes S)\otimes T \rightarrowtail R\otimes S \otimes T$ is a trivial cofibration. (Recall from \cite{HHM} that $R\otimes S\otimes T$ is the union of all threefold shuffles of $R$, $S$ and $T$, while $(R\otimes S)\otimes T$ is the union over the smaller set of shuffles given by shuffling $T$ with a single shuffle of $R\otimes S$ at the time; cf also the example below.) Write $R'\rightarrowtail R$ for the closed spine of $R$ and similarly for $S$ and $T$. Then there is a diagram 
	\[
	\xymatrix{
	(R'\otimes S') \otimes T' \ar@{>->}[r] \ar@{>->}[d] & R'\otimes S'\otimes T' \ar@{>->}[d] \\
	(R\otimes S) \otimes T \ar@{>->}[r] & R\otimes S\otimes T
 	} \]	
 where the vertical maps are via morphisms, by Proposition \ref{prop7-3} and Theorem \ref{theorem7-6}. Thus, to prove the claim, it suffices to prove that the upper map is a trivial cofibration. Using the fact that all objects involved in this map are unions of closed corollas, this now reduces the problem to showing that for three closed corollas $\overline{C}_k$, $\overline{C}_l$ and $\overline{C}_m$, the map 
 \[
 	(\overline{C}_k\otimes \overline{C}_l)\otimes \overline{C}_m \rightarrowtail\overline{C}_k \otimes \overline{C}_l \otimes \overline{C}_m
 \]
 is a trivial cofibration. In fact, it is a via morphism. Let us prove this for $k=l=1$ and $m=2$. It will be clear that the argument is exactly the same (but notationally more involved) for general $k$, $l$ and $m$. So, write the three corollas as 
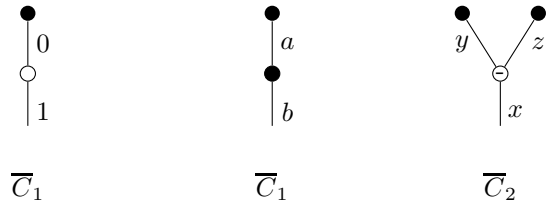
\[
\begin{tikzpicture} 
[level distance=8mm, 
every node/.style={fill, circle, minimum size=.2cm, inner sep=0pt}, 
level 1/.style={sibling distance=10mm}, 
level 2/.style={sibling distance=10mm}, 
level 3/.style={sibling distance=10mm}]

\node(anchorR1)[style={color=white}] {} [grow'=up] 
child {node(vertexR1)[draw, fill=none] {} 
	child{node {}}
};

\node[style={color=white}, right=3cm of anchorR1] {} [grow'=up] 
child {node(vertexR2)[draw] {} 
	child{node {}}
};

\node[style={color=white}, right=6cm of anchorR1] {} [grow'=up] 
child {node(vertexR3)[draw, fill=none]{-}  
	child{node {}}
	child{node {}}
};

\tikzstyle{every node}=[]

\node at ($(vertexR1) + (.2cm,.4cm)$) {$0$};
\node at ($(vertexR1) + (.2cm,-.5cm)$) {$1$};
\node at ($(vertexR1) + (0,-1.5cm)$) {$\overline{C}_1$};

\node at ($(vertexR2) + (.2cm,.4cm)$) {$a$};
\node at ($(vertexR2) + (.2cm,-.5cm)$) {$b$};
\node at ($(vertexR2) + (0,-1.5cm)$) {$\overline{C}_1$};

\node at ($(vertexR3) + (-.5cm,.4cm)$) {$y$};
\node at ($(vertexR3) + (.5cm,.4cm)$) {$z$};
\node at ($(vertexR3) + (.2cm,-.5cm)$) {$x$};
\node at ($(vertexR3) + (0,-1.5cm)$) {$\overline{C}_2$};

\end{tikzpicture} 
\]
Then $\overline{C}_1\otimes \overline{C}_1$ is a union of two shuffles 
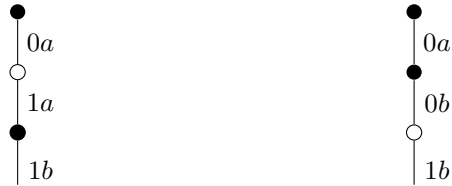
\[
\begin{tikzpicture} 
[level distance=8mm, 
every node/.style={fill, circle, minimum size=.2cm, inner sep=0pt}, 
level 1/.style={sibling distance=10mm}, 
level 2/.style={sibling distance=10mm}, 
level 3/.style={sibling distance=10mm}]

\node(anchorR1)[style={color=white}] {} [grow'=up] 
child {node(vertexR1)[draw] {} 
child{node[draw, fill=none] {}
	child{node {}}
}
};

\node[style={color=white}, right=5cm of anchorR1] {} [grow'=up] 
child {node(vertexR2)[draw, fill=none] {} 
child{node {}
	child{node {}}
}
};
	
\tikzstyle{every node}=[]

\node at ($(vertexR1) + (.3cm,1.2cm)$) {$0a$};
\node at ($(vertexR1) + (.3cm,.4cm)$) {$1a$};
\node at ($(vertexR1) + (.3cm,-.5cm)$) {$1b$};

\node at ($(vertexR2) + (.3cm,1.2cm)$) {$0a$};
\node at ($(vertexR2) + (.3cm,.4cm)$) {$0b$};
\node at ($(vertexR2) + (.3cm,-.5cm)$) {$1b$};

\end{tikzpicture} 
\]
 So $(\overline{C}_1\otimes\overline{C}_1 )\otimes \overline{C}_2$ is the union of the following two: 
 
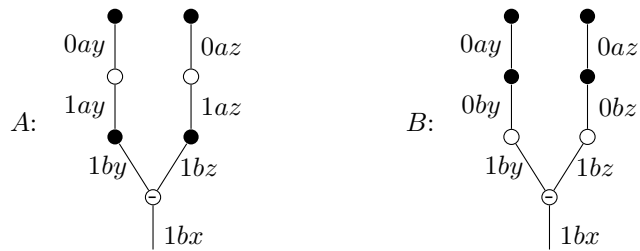
\[
 \begin{tikzpicture} 
 [level distance=8mm, 
 every node/.style={fill, circle, minimum size=.2cm, inner sep=0pt}, 
 level 1/.style={sibling distance=10mm}, 
 level 2/.style={sibling distance=10mm}, 
 level 3/.style={sibling distance=10mm}]
 
 \node(anchorR1)[style={color=white}] {} [grow'=up] 
 child {node(vertexR1)[draw, fill=none] {-} 
 	child{node {}
 		child{node[draw, fill=none] {}
 				child{node {}}}}
 child{node {}
 	child{node[draw, fill=none] {}
 		child{node {}}}}	
 };
 
 \node[style={color=white}, right=5cm of anchorR1] {} [grow'=up] 
 child {node(vertexR2)[draw, fill=none] {-}
	child{node[draw, fill=none] {}
		child{node {}
			child{node {}}}}
	child{node[draw, fill=none] {}
		child{node {}
			child{node {}}}}	
};
 
 \tikzstyle{every node}=[]
 
 \node at ($(vertexR1) + (-.9cm,2cm)$) {$0ay$};
 \node at ($(vertexR1) + (.9cm,2cm)$) {$0az$};
 \node at ($(vertexR1) + (-.9cm,1.2cm)$) {$1ay$};
 \node at ($(vertexR1) + (.9cm,1.2cm)$) {$1az$};
 \node at ($(vertexR1) + (-.6cm,.4cm)$) {$1by$}; 
 \node at ($(vertexR1) + (.6cm,.4cm)$) {$1bz$};
 \node at ($(vertexR1) + (.4cm,-.5cm)$) {$1bx$};
 \node at ($(vertexR1) + (-1.7cm,1cm)$) {$A$:};
 
\node at ($(vertexR2) + (-.9cm,2cm)$) {$0ay$};
\node at ($(vertexR2) + (.9cm,2cm)$) {$0az$};
\node at ($(vertexR2) + (-.9cm,1.2cm)$) {$0by$};
\node at ($(vertexR2) + (.9cm,1.2cm)$) {$0bz$};
\node at ($(vertexR2) + (-.6cm,.4cm)$) {$1by$}; 
\node at ($(vertexR2) + (.6cm,.4cm)$) {$1bz$};
\node at ($(vertexR2) + (.4cm,-.5cm)$) {$1bx$};
\node at ($(vertexR2) + (-1.7cm,1cm)$) {$B$:};
 
 \end{tikzpicture} 
 \] 
 together with the shuffles obtained by percolating through the binary vertex: 
 
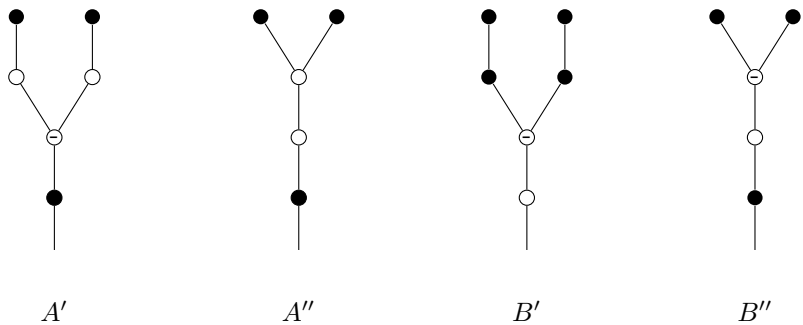
\[
 \begin{tikzpicture} 
 [level distance=8mm, 
 every node/.style={fill, circle, minimum size=.2cm, inner sep=0pt}, 
 level 1/.style={sibling distance=10mm}, 
 level 2/.style={sibling distance=10mm}, 
 level 3/.style={sibling distance=10mm}]
 
 \node(anchorR1)[style={color=white}] {} [grow'=up] 
 child {node(vertexR1)[draw] {} 
 	child{node[draw, fill=none] {-}
 	child{node[draw, fill=none] {}
 		child{node {}}}
 	child{node[draw, fill=none] {}
 		child{node {}}}
 	}
 };
 
 \node[style={color=white}, right=3cm of anchorR1] {} [grow'=up] 
 child {node(vertexR2)[draw] {} 
 child{node[draw, fill=none] {}
 	child{node[draw, fill=none] {}
 		child{node {}}
 		child{node {}}
 	}}
 };
 
 \node[style={color=white}, right=6cm of anchorR1] {} [grow'=up] 
 child {node(vertexR3)[draw, fill=none] {} 
 	child{node[draw, fill=none] {-}
 		child{node {}
 			child{node {}}}
 		child{node {}
 			child{node {}}} 		
 	}
 };
 
 \node[style={color=white}, right=9cm of anchorR1] {} [grow'=up] 
 child {node(vertexR4) {} 
 	child{node[draw, fill=none] {}
 		child{node[draw, fill=none] {-}
 			child{node {}}
 			child{node {}} 
 	}}
 };
 
 \tikzstyle{every node}=[]
 
 \node at ($(vertexR1) + (0,-1.5cm)$) {$A'$};
 
 \node at ($(vertexR2) + (0,-1.5cm)$) {$A''$};
 
 \node at ($(vertexR3) + (0,-1.5cm)$) {$B'$};
 
 \node at ($(vertexR4) + (0,-1.5cm)$) {$B''$};
 
 \end{tikzpicture} 
 \]
 In $\overline{C}_1 \otimes \overline{C}_1 \otimes \overline{C}_2$, there are two more shuffles which mix the two parts in $\overline{C}_1\otimes \overline{C}_2$: 
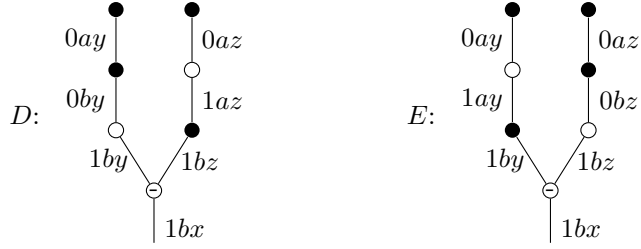
\[
\begin{tikzpicture} 
[level distance=8mm, 
every node/.style={fill, circle, minimum size=.2cm, inner sep=0pt}, 
level 1/.style={sibling distance=10mm}, 
level 2/.style={sibling distance=10mm}, 
level 3/.style={sibling distance=10mm}]

\node(anchorR1)[style={color=white}] {} [grow'=up] 
child {node(vertexR1)[draw, fill=none] {-} 
	child{node[draw, fill=none] {}
		child{node {}
			child{node {}}}}
	child{node {}
		child{node[draw, fill=none] {}
			child{node {}}}}	
};

\node[style={color=white}, right=5cm of anchorR1] {} [grow'=up] 
child {node(vertexR2)[draw, fill=none] {-}
	child{node {}
		child{node[draw, fill=none] {}
			child{node {}}}}
	child{node[draw, fill=none] {}
		child{node {}
			child{node {}}}}	
};

\tikzstyle{every node}=[]

\node at ($(vertexR1) + (-.9cm,2cm)$) {$0ay$};
\node at ($(vertexR1) + (.9cm,2cm)$) {$0az$};
\node at ($(vertexR1) + (-.9cm,1.2cm)$) {$0by$};
\node at ($(vertexR1) + (.9cm,1.2cm)$) {$1az$};
\node at ($(vertexR1) + (-.6cm,.4cm)$) {$1by$}; 
\node at ($(vertexR1) + (.6cm,.4cm)$) {$1bz$};
\node at ($(vertexR1) + (.4cm,-.5cm)$) {$1bx$};
\node at ($(vertexR1) + (-1.7cm,1cm)$) {$D$:};

\node at ($(vertexR2) + (-.9cm,2cm)$) {$0ay$};
\node at ($(vertexR2) + (.9cm,2cm)$) {$0az$};
\node at ($(vertexR2) + (-.9cm,1.2cm)$) {$1ay$};
\node at ($(vertexR2) + (.9cm,1.2cm)$) {$0bz$};
\node at ($(vertexR2) + (-.6cm,.4cm)$) {$1by$}; 
\node at ($(vertexR2) + (.6cm,.4cm)$) {$1bz$};
\node at ($(vertexR2) + (.4cm,-.5cm)$) {$1bx$};
\node at ($(vertexR2) + (-1.7cm,1cm)$) {$E$:};

\end{tikzpicture} 
\] 
So we need to see how to adjoin these. More precisely, \[\overline{C}_1 \otimes \overline{C}_1 \otimes \overline{C}_2=((\overline{C}_1 \otimes \overline{C}_1) \otimes \overline{C}_2)\cup D \cup E.\] For the shuffle $D$, the upper very inner faces (contracting $0by$, respectively $1az$) are contained in $A\cap B$ hence in $(\overline{C}_1 \otimes \overline{C}_1) \otimes \overline{C}_2$. But neither the very inner face contracting $1by$ nor the one contracting $1bz$ (nor their intersection) is. And similarly, neither $\partial_{0ay}$ nor $\partial_{0az}$ nor their intersection is. By Lemma \ref{lem7-2} we see that \[(\overline{C}_1 \otimes \overline{C}_1 )\otimes \overline{C}_2\cap D \rightarrowtail D\] is a via morphism, and hence so is its pushout $(\overline{C}_1 \otimes \overline{C}_1 )\otimes \overline{C}_2\to ((\overline{C}_1 \otimes \overline{C}_1 )\otimes \overline{C}_2)\cup D$.
In the same way, one shows that $(((\overline{C}_1 \otimes \overline{C}_1 )\otimes \overline{C}_2)\cup D) \cap E \rightarrowtail E$ is a via morphism, and hence so is its pushout
$((\overline{C}_1 \otimes \overline{C}_1 )\otimes \overline{C}_2)\cup D \to \overline{C}_1 \otimes \overline{C}_1 \otimes \overline{C}_2$.
\end{remark}

\begin{corollary} \label{cor9-4}
	The tensor product on $\cdSets$ induces a symmetric monoidal structure on the homotopy category associated to the unital operadic model structure of Theorem \ref{teo9-1}. 
\end{corollary}

This corollary is one of the main differences between this model structure on $\cdSets$ and the operadic one on $\dSets$. 

\section{Some Quillen functors} \label{sect10}

Recall that we have introduced the following functors between small categories: 
\begin{equation*}
\xymatrix{
	\Omega_o \ar[r]^o \ar[rd]_h & \Omega \ar[d]^{cl} & \\ 
	& \Omega_{cl} \ar[r]^u & \Omega
} \qquad h=cl\circ o, \quad cl \dashv u
\end{equation*}

These induce restriction functors between presheaf categories 
\begin{equation*}
\xymatrix{
	\odSets  & \dSets\ar[l]_{o^*}   & \\ 
	& \cdSets  \ar[lu]^{h^*} \ar[u]^{cl^*} & \dSets \ar[l]_{u^*}
} 
\end{equation*}
which each have a left adjoint $(-)_!$ and a right adjoint $(-)_*$. Moreover, these categories carry model structures: the operadic one on $\dSets$ which induces a model structure on its slice $\odSets=\dSets/O$, and the unital operadic one on $\cdSets$ discussed in the previous section. In this section, we will investigate to what extent these different adjoint pairs form Quillen pairs. For completeness, we begin by stating a trivial observation.

\begin{proposition}
	The adjoint pair $\xymatrix{
		o_! : \odSets \ar@<.5ex>[r] & \dSets : o^* \ar@<.5ex>[l]  } $ is a Quillen pair.
\end{proposition}

\begin{proof}
	This is simply a special case of the well-known general property of an adjunction of the form $\xymatrix{\mathcal{E}/X \ar@<.5ex>[r] & \mathcal{E} \ar@<.5ex>[l] }$ associated to an arbitrary object $X$ in a model category $\mathcal{E}$. 
\end{proof}

\begin{remark} \label{rem10-2}
	The functor $o^*$ has a right adjoint, and preserves normal monomorphisms, i.e. cofibrations. Nonetheless, it is not a left Quillen functor. For example, for the tree $T$ pictured below, $o_!o^*$ maps the inner horn $\Lambda^b(T)\rightarrowtail T$ to the embedding of $T^\circ$ into $A$, as in 
	
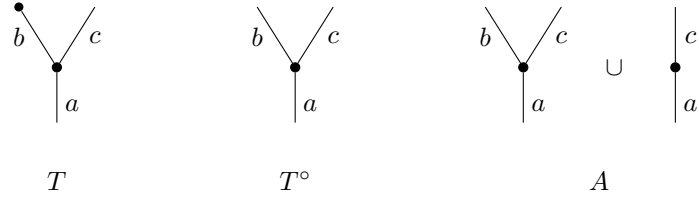
\[
	\begin{tikzpicture} 
	[level distance=8mm, 
	every node/.style={fill, circle, minimum size=.12cm, inner sep=0pt}, 
	level 1/.style={sibling distance=10mm}, 
	level 2/.style={sibling distance=10mm}, 
	level 3/.style={sibling distance=10mm}]
	
	\node(anchorR1)[style={color=white}] {} [grow'=up] 
	child {node(vertexR1)[draw] {} 
		child{node {}}
		child
	};
	
	\node[style={color=white}, right=3cm of anchorR1] {} [grow'=up] 
	child {node(vertexR2)[draw] {} 
		child
		child
	};
	
	\node[style={color=white}, right=6cm of anchorR1] {} [grow'=up] 
	child {node(vertexR3)[draw] {} 
		child
		child
	};
	
	\node[style={color=white}, right=8cm of anchorR1] {} [grow'=up] 
	child {node(vertexR4)[draw] {} 
		child
	};
	
	\tikzstyle{every node}=[]
	
	\node at ($(vertexR1) + (-.5cm,.4cm)$) {$b$};
	\node at ($(vertexR1) + (.5cm,.4cm)$) {$c$};
	\node at ($(vertexR1) + (.2cm,-.5cm)$) {$a$};
	\node at ($(vertexR1) + (0,-1.5cm)$) {$T$};
	
	\node at ($(vertexR2) + (-.5cm,.4cm)$) {$b$};
	\node at ($(vertexR2) + (.5cm,.4cm)$) {$c$};
	\node at ($(vertexR2) + (.2cm,-.5cm)$) {$a$};
	\node at ($(vertexR2) + (0,-1.5cm)$) {$T^\circ$};
	
	\node at ($(vertexR3) + (-.5cm,.4cm)$) {$b$};
	\node at ($(vertexR3) + (.5cm,.4cm)$) {$c$};		
	\node at ($(vertexR3) + (.2cm,-.5cm)$) {$a$};
	\node at ($(vertexR3) + (1,-1.5cm)$) {$A$};
	\node at ($(vertexR3) + (1.2,0cm)$) {$\cup$};
	
	\node at ($(vertexR4) + (.2cm,.4cm)$) {$c$};
	\node at ($(vertexR4) + (.2cm,-.5cm)$) {$a$};

	\end{tikzpicture} 
	\]
	where $\cup$ denotes the pushout over $a$ and $c$. This embedding $T^\circ \rightarrowtail A$ cannot be a trivial cofibration, because that would make $T^\circ$ a retract of $A$, which is clearly impossible. 
\end{remark}

\begin{remark} \label{rem10-3}
	The functor $cl_!$ does not preserve (normal) monomorphisms as we have seen, so cannot be a left Quillen functor. 
\end{remark}

Next we turn to the adjoint functors induced by $u\colon \Omega_{cl} \to \Omega$.

\begin{proposition} \label{prop10-4}
	The adjoint pair $\xymatrix{
		u_! : \cdSets \ar@<.5ex>[r] & \dSets : u^* \ar@<.5ex>[l]  } $ is a Quillen pair.
\end{proposition}

\begin{proof}
	If $T$ is a closed tree, its closed boundary $\partial_{cl} T\to T$ consists of the union of inner faces of $T$, together with the root face in case the root vertex of $T$ is unary. Since the intersection of any two such faces is again a closed tree, it follows that $\partial_{cl}(T) \to T$ is a monomorphism in $\dSets$, necessarily normal since $T$ is normal. In other words, $u_!(\partial_{cl} T)\to u_!(T)$ is a normal monomorphism. Since these closed boundary inclusions generate the cofibrations in $\cdSets$, this shows that $u_!$ preserves cofibrations. 
	
	To verify that $u_!$ preserves \emph{trivial} cofibrations as well, it suffices to check that $u^*$ preserves fibrations between fibrant objects. Thus, by Theorem \ref{the9-2d} (d), it suffices to check that $u_!$ sends the morphisms of the following two kinds to weak equivalences: the inclusion $\overline{\eta}=cl_!\{0\} \rightarrowtail cl_!(J)$ of one of its endpoints into the ``closed interval'', and all inclusions of very inner horns $\Lambda^e T \to T$ into closed trees $T$. 
	
	For the first type of inclusion, observe that $\{0\}\to J$ is a trivial cofibration in $\dSets$, and hence so is its pushout 
	\[
	\overline { \{ 0 \} }=u_! cl_!\{ 0 \} \rightarrowtail J \cup u_! cl_!\{ 0 \} = J \cup \overline{ \{ 0 \} }.
	\]
	We claim that $J\cup \overline{\{0\}}\rightarrowtail u_! cl_!(J)=: \overline{J}$ is inner anodyne. Observe that the non-degenerate dendrices of $\overline{J}$ are all of the form $\overline{L}_n:=u_! cl_! [n]$: 
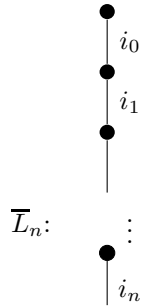
	\[
	\begin{tikzpicture} 
	[level distance=8mm, 
	every node/.style={fill, circle, minimum size=.2cm, inner sep=0pt}, 
	level 1/.style={sibling distance=10mm}, 
	level 2/.style={sibling distance=10mm}, 
	level 3/.style={sibling distance=10mm}]
	
	\node(anchorR1)[style={color=white}] {} [grow'=up] 
	child {node(vertexR1)[draw] {} 
		child{edge from parent [draw=none] 
			child{node {}
			child{node {}
			child{node {}}}}
		}
	};
	
	\tikzstyle{every node}=[]
	
	\node at ($(vertexR1) + (.3cm,2.8cm)$) {$i_0$};
	\node at ($(vertexR1) + (.3cm,2cm)$) {$i_1$};
	\node at ($(vertexR1) + (.3cm,.4cm)$) {$\vdots$};
	\node at ($(vertexR1) + (.3cm,-.5cm)$) {$i_n$};
	\node at ($(vertexR1) + (-1cm,.4cm)$) {$\overline{L}_n$:};
	\end{tikzpicture} 
	\]
	where $i_0, \ldots, i_n$ is an alternating sequence of zeroes and ones, and hence all these are faces of ones of the form where $i_0=0$ and we use $\overline{L}_n$ only for the one with $i_0$ from now on. We can adjoin these dendrices to $J\cup\overline{ \{ 0 \} }=J\cup \overline{L}_0$ by induction. For $n=0$, the dendrex is already there. For $n=1$, the dendrex $\overline{L}_1$ has its top face in $J$ and its root face is $\overline{L}_0=\overline{ \{ 0 \} }$, so it misses only its inner face. Thus $(J\cup \overline{L}_0)\cap \overline{L}_1 = \Lambda^0 \overline{L}_1 \rightarrowtail \overline{L}_1$ is inner anodyne, and hence so is $J\cup \overline{L}_0 \to J\cup \overline{L}_0 \cup \overline{L}_1 =J \cup \overline{L}_1$. Similarly, $(J\cup \overline{L}_{n-1})\cap \overline{L}_n = \Lambda^0 \overline{L}_n \rightarrowtail \overline{L}_n$ only misses its top inner face. Indeed, the top outer face of $\overline{L}_n$ belongs to $J$ and the other inner faces are faces of $\overline{L}_{n-1}$. Thus $J\cup \overline{L}_{n-1} \to J\cup \overline{L}_n$ is again inner anodyne. Taking the colimit, we conclude that $J\cup \overline{L}_0 \rightarrowtail \bigcup_n \overline{L}_n = \overline{J}$ is inner anodyne. 
	
	Next, let us turn to the inclusions of the form $u_!(\Lambda_{cl}^e T) \rightarrowtail u_!(T)$ for a closed tree $T\in \Omega_{cl}$ and a very inner edge $e$ in $T$. Notice that $u_! T$ is simply $T$ as an object of $\Omega$, or as a representable object in $\dSets$, while 
	\[
	u_!(\Lambda_{cl}^e T) \rightarrowtail \Lambda^e T \rightarrowtail T.
	\]
	The first inclusion is not an isomorphism, because the non-closed faces chopping off stumps from $T$ belong to $\Lambda^e T$ but not to $u_!(\Lambda_{cl}^e T)$, the latter being the union of \emph{closed} trees. The fact that $u_!(\Lambda_{cl}^e T)\rightarrowtail T$ is a trivial cofibration is thus a special case of the following slightly more general lemma. 
\end{proof}

\begin{lemma} \label{lemma10-5}
	Let $S$ be an arbitrary tree (not necessarily closed) and let $e$ be an inner edge in $S$. Let $A\subseteq \Lambda^e S$ the union of some faces of $S$, namely the root face (if it exists), all the inner faces except the one contracting $e$, and all the top faces \emph{except} a number of top faces chopping off stumps. Suppose that none of these stumps is immediately above $e$. Then $A\rightarrowtail S$ is inner anodyne. 
\end{lemma}

\begin{proof}
	List the stumps $v$ for which $\partial_v S$ is missing from $A$ as $v_1$, \ldots, $v_k$. Then $A\cup \partial_{v_1} S\cup \ldots \partial_{v_k}S \rightarrowtail S$ is the inner horn $\Lambda^e S \rightarrowtail S$, so it suffices to prove that each $A\cup \partial_{v_1} S\cup \ldots \partial_{v_i}S \rightarrowtail A\cup \partial_{v_1} S\cup \ldots \partial_{v_{i+1}}S $ is inner anodyne. Now $(A\cup \partial_{v_1} S\cup \ldots \cup \partial_{v_i}S)\cap \partial_{v_{i+1}} S \rightarrowtail \partial_{v_{i+1}} S$ is an inclusion of the form stated in the lemma, but missing only $k-(i+1)$ top faces chopping of stumps. So, proceeding by induction on $k$, it suffices to prove the case $k=1$. Write $v=v_1$ for the one stump for which $\partial_v S$ is missing from $A$. Thus $A\cup \partial_v S =\Lambda^e S$, and it suffices to show that $A\cap \partial_v S=\Lambda^e(\partial_v S)$. We show both inclusions in turn.
	
	($\supseteq$) The faces of $\partial_v S$ are faces of $S$, with the possible exception of the top face $\partial_w \partial_v S$ for the vertex $w$ immediately below $v$. This top face of $\partial_v S$ exists if $w$ is a top vertex of $\partial_v S$. But then $\partial_w \partial_v S$ is a face of $\partial_b S\subseteq A$, where $b\neq e$ is the edge between $v$ and $w$. This proves that $\Lambda^e \partial_v(S) \subseteq A\cap \partial_v S$.
	
	($\subseteq$) The object $A$ is the union of three or four types of faces of $S$: 
	\begin{enumerate}
		\item[(i)] inner faces $\partial_a S$ where $a\neq e$ and $a$ is not attached to $v$, 
		\item[(ii)] top faces $\partial_w S$ for $w$ a top vertex of $S$, $w\neq v$, 
		\item[(iii)] possibly the root face $\partial_{\text{root}} S$, 
		\item[(iv)] the face $\partial_b S$ where $b$ is the edge below $v$. 
	\end{enumerate}
	We will show that for each of these faces, the intersection with $\partial_v S$ is contained in $\Lambda^e \partial_v S$. 
	\begin{enumerate}
		\item[(i)] For such a face $\partial_a S$, the intersection $\partial_a S \cap \partial_v S$ is an inner face of $\partial_v S$, hence contained in $\Lambda^e \partial_v S$.
		\item [(ii)] Similarly, as $w\neq v$, for such a face $\partial_w S$ the intersection $\partial_w S \cap \partial_v S$ is a top face of $\partial_v S$, hence contained in $\Lambda^e \partial_v S$. 
		\item[(iii)] If the root face $\partial_{root} S$ exists, then $\partial_{root} S \cap \partial_v S = \partial_{root} \partial_v S$ because $S$ contains $e$ hence is not a corolla, and this intersection is contained in $\Lambda^e \partial_v S$. 
		\item[(iv)] The only non-trivial case is that of $\partial_b S\cap \partial_v S$ for the edge $b$ immediately below $v$. Write $w$ for the vertex below $b$, and $a_1, \ldots, a_n$ ($n\geq 0$) for the input edges of $w$ other than $b$. Write $T_{a_i}$ for the subtree of $S$ with root $a_i$, and $T_r$ for the subtree below $w$: 
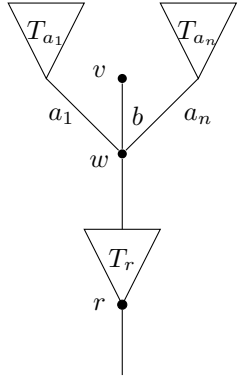
			\[
		\begin{tikzpicture} 
		[level distance=10mm, 
		every node/.style={fill, circle, minimum size=.12cm, inner sep=0pt}, 
		level 1/.style={sibling distance=10mm}, 
		level 2/.style={sibling distance=10mm}, 
		level 3/.style={sibling distance=10mm}]
		
		\node(anchorR1)[style={color=white}] {} [grow'=up] 
		child {node(vertexR1)[draw] {} 
			child{ edge from parent [draw=none]
			child{node {}
				child
				child{node {}}
				child
			}}
		};
		
		\draw (-1.5,5) -- (-1,4) -- (-.5,5) -- cycle;
		\draw (-.5,2) -- (0,1) -- (.5,2) -- cycle;
		\draw (.5,5) -- (1,4) -- (1.5,5) -- cycle;
		
		\tikzstyle{every node}=[]
		
		\node at ($(vertexR1) + (0cm,.6cm)$) {$T_r$};
		\node at ($(vertexR1) + (-.3cm,1.9cm)$) {$w$};
		\node at ($(vertexR1) + (-.8cm,2.5cm)$) {$a_1$};
		\node at ($(vertexR1) + (.2cm,2.5cm)$) {$b$};
		\node at ($(vertexR1) + (1cm,2.5cm)$) {$a_n$};
		\node at ($(vertexR1) + (-.3cm,3.1cm)$) {$v$};	
		\node at ($(vertexR1) + (1cm,3.6cm)$) {$T_{a_n}$};
		\node at ($(vertexR1) + (-1cm,3.6cm)$) {$T_{a_1}$};
		\node at ($(vertexR1) + (-.3cm,0cm)$) {$r$};	
		
		\end{tikzpicture} 
		\] 
		Then $\partial_b S\cap \partial_v S=T_{a_1}\cup \ldots \cup T_{a_n}\cup T_r$, and we show each of these parts is contained in $\Lambda^e\partial_v S$. Notice that $T_r \subseteq \partial_{a_i} \partial_v S$ which belongs to $\Lambda^e\partial_v S$ as soon as $a_i$ is an inner edge of $\partial_v S$. Similarly, $T_{a_i}\subseteq \partial_c \partial_v S \subseteq \Lambda^e \partial_v S$ if we can find an inner edge $c$ in $\partial_v S$ which does not belong to $T_{a_i}$, or if $T_{a_i}$ is contained in the root face of $S$. This proves that $\partial_b S \cap \partial_v S \subseteq \Lambda^e \partial_v S$, except possibly in the case where each $a_i$ is a leaf of $S$, or $n=0$. But in that case $\partial_w \partial_v S = \partial_{wv} \partial_b S \subseteq \Lambda^e S$ (where $wv$ is the new vertex resulting from contracting $b$) is a face of $\partial_v S$ containing $T_r$, while each $a_i$ is contained in $\partial_c \partial_v S$ for the edge $c$ immediately below $w$ (which must be inner as $S$ contains the edge $e$ not immediately below $v$). 
		\end{enumerate}
This proves that $\partial_b S \cap \partial_v S \subseteq \Lambda^e \partial_v S$, and completes the proof of the inclusion $\subseteq$, and hence that of the lemma.  	
\end{proof}

Finally, we consider one of the adjoint pairs induced by the composition $h=cl \circ o\colon \Omega_o \to \Omega_{cl}$.

\begin{theorem} \label{teo10-6}
	The adjoint pair $\xymatrix{
		h_! : \cdSets \ar@<.5ex>[r] & \odSets : h^* \ar@<.5ex>[l]  } $ is a Quillen pair. 
\end{theorem}

\begin{proof}
	The functor $h^*$ is the composition $h^*=o^* cl^* = o^* u_!$, so it clearly preserves normal monomorphisms as both $o^*$ and $u_!$ do (cf. Remark \ref{rem10-2} and Proposition \ref{prop10-4} above). It thus suffices to show that $h^*$ sends the two types of ``quasi-generating'' (in the sense of detecting the fibrations between fibrant objects) trivial cofibrations to weak equivalences. For the trivial cofibration $\overline{ \{ 0 \} }\to \overline{J}$ featuring in the previous proof also, this is clear because $h^*$ maps it back to $\{ 0 \} \to J$. 
	
	Let us consider a very inner horn inclusion $\Lambda^e T\rightarrowtail T$ for a closed tree $T$. If $S$ is an arbitrary closed tree, let us write $S^\circ \subseteq S$ for the open subtree obtained by chopping off all the stumps of $S$. For an upwards closed set $A$ of edges in $T$ (i.e., $f\geq g \in A \Rightarrow f\in A$), write $T[A]$ for the closed tree obtained by contracting all the edges in $A$, and $T[A]^\circ \subseteq T[A] \subseteq T$ for the resulting open face of $T$. Here is an example: 
		
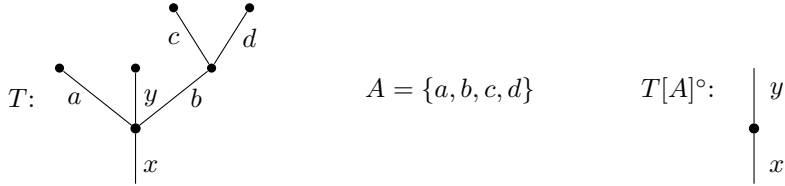
\[
	\begin{tikzpicture} 
	[level distance=8mm, 
	every node/.style={fill, circle, minimum size=.12cm, inner sep=0pt}, 
	level 1/.style={sibling distance=10mm}, 
	level 2/.style={sibling distance=10mm}, 
	level 3/.style={sibling distance=10mm}]
	
	\node(anchorR1)[style={color=white}] {} [grow'=up] 
	child {node(vertexR1)[draw] {} 
		child{node {}}
		child{node {}}
		child{node {}
			child{node {}}
			child{node {}}}
	};
	
	\node[style={color=white}, right=8cm of anchorR1] {} [grow'=up] 
	child {node(vertexR2)[draw] {} 
		child
		};

	\tikzstyle{every node}=[]
	
	\node at ($(vertexR1) + (-1.5cm,.4cm)$) {$T$:};
	\node at ($(vertexR1) + (.5cm,1.2cm)$) {$c$};
	\node at ($(vertexR1) + (1.5cm,1.2cm)$) {$d$};
	\node at ($(vertexR1) + (-.8cm,.4cm)$) {$a$};
	\node at ($(vertexR1) + (.8cm,.4cm)$) {$b$};
	\node at ($(vertexR1) + (.2cm,.4cm)$) {$y$};
	\node at ($(vertexR1) + (.2cm,-.5cm)$) {$x$};

	\node at ($(vertexR2) + (.3cm,.5cm)$) {$y$};
	\node at ($(vertexR2) + (.3cm,-.5cm)$) {$x$};
	\node at ($(vertexR2) + (-4cm,.5cm)$) {$A=\{a,b,c,d\}$};
	\node at ($(vertexR2) + (-1cm,.5cm)$) {$T[A]^\circ$:};

	\end{tikzpicture} 
	\] 
	Then $h^*(T)= \bigcup_A T[A]^\circ$, where the union is as subobject of $T$, and ranges over all such upwards closed subsets $A$. Now all these trees $T[A]^\circ$ are contained in $h^*(\Lambda^e[T])$, except for the case $A=\emptyset$ when $T[\emptyset]^\circ=T^\circ$. (This uses that $e$ is very inner so $\{e\}$ is not itself upwards closed.) But $T^\circ\cap h^*(\Lambda^e[T])=\Lambda^e[T^\circ]$. So $T^\circ\cap h^*(\Lambda^e[T])\rightarrowtail T^\circ$ is inner anodyne, and hence so is its pushout $h^*(\Lambda^e[T]) \to T^\circ\cup h^*(\Lambda^e[T]) = h^*(T)$. This completes the proof of the theorem. 
\end{proof}

\section{The Boardman-Vogt resolution} \label{sect11}

The Boardman-Vogt resolution for simplicial operads associated to each such operad $P$ is a map $W(P)\to P$ which forms a cofibrant resolution in case $P$ is $\Sigma$-free (cf. \cite{bergermoerdijkBVresolution}). In particular, it associates to each tree $T$ (viewed as operad) a simplicial operad $W(T)$. This operad $W(T)$ is simple to describe explicitly: the colours of $W(T)$ are the edges of $T$. For edges $e_1, \ldots, e_n$ and $e$ in $T$, the space $W(T)(e_1, \ldots, e_n; e)$ is empty, unless there is a (unique) subtree of $T$ with leaves $e_1, \ldots, e_n$ and root $e$, in which case 
\[
W(T)(e_1, \ldots, e_n; e) =\prod_{d\in I(e_1, \ldots, e_n; e)} \Delta[1], 
\]
where $I(e_1, \ldots, e_n; e)$ is the set of \emph{inner} edges in this subtree. Composition in the operad $W(T)$ is defined by assigning ``length $1$'' to the newly arising inner edges. This construction defines a functor $W\colon \Omega\to \Oper$ into the category of simplicial operads, which induces a pair of adjoint functors
\[
\xymatrix{
	w_! : \dSets \ar@<.5ex>[r] & \Oper : w^* \ar@<.5ex>[l]  } 
\]
by Kan extension. The following theorem is one of the main results of \cite{cisinskimoerdijk1, cisinskimoerdijk2, cisinskimoerdijk3}:

\begin{theorem}
	The pair $\xymatrix{
		w_! : \dSets \ar@<.5ex>[r] & \Oper : w^* \ar@<.5ex>[l]  } $ is a Quillen equivalence for the operadic model structure on the category $\dSets$ of dendroidal sets and the projective model structure on the category $\Oper$ of simplicial operads. 
\end{theorem}

Since $W$ maps open trees to open operads, $w_!$ maps open dendroidal sets to open operads, and the theorem has the following immediate corollary. 

\begin{corollary}
	This pair restricts to a Quillen equivalence 
	\[
	\xymatrix{ {\accentset{\circ}{w_!}} : \odSets \ar@<.5ex>[r] & \oOper :  {\accentset{\circ}{w}^*}  \ar@<.5ex>[l]  } 
	\]
	between open dendroidal sets and open operads.
\end{corollary}

For closed operads, the situation is slightly different, as the following example shows. 

\begin{example}
	Let $P$ be a closed operad, with set of colours $C$, all of whose spaces of operations are contractible Kan complexes. Then the map $P\to 1$ into the terminal object is a trivial fibration. Consider the closed tree $T$, 
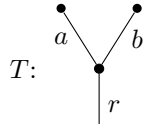
	\[
	\begin{tikzpicture} 
	[level distance=8mm, 
	every node/.style={fill, circle, minimum size=.12cm, inner sep=0pt}, 
	level 1/.style={sibling distance=10mm}, 
	level 2/.style={sibling distance=10mm}, 
	level 3/.style={sibling distance=10mm}]
	
	\node(anchorR1)[style={color=white}] {} [grow'=up] 
	child {node(vertexR1)[draw] {} 
		child{node {}}
		child{node {}}
	};
	
	\tikzstyle{every node}=[]
	
	\node at ($(vertexR1) + (-.5cm,.4cm)$) {$a$};
	\node at ($(vertexR1) + (.5cm,.4cm)$) {$b$};
	\node at ($(vertexR1) + (.2cm,-.5cm)$) {$r$};
	\node at ($(vertexR1) + (-1cm,0cm)$) {$T$:};
	\end{tikzpicture} 
	\]
	and its closed boundary $\partial_{cl}(T)$, and a commutative square of the form 
	\[
	\xymatrix{
	& w_!(\partial_{cl}(T)) \ar[r] \ar[d] & P \ar[d] \\  T \ar[r]^u  & W(T) \ar@{.>}[ru] \ar[r] & 1,  }
	\]
	together with the map of operads $T\to W(T)$ which assigns length $1$ to any edge. Write $a$, $b$ and $r$ also for the colours of $P$ which are images of the corresponding edges of $T$ under $\partial_{cl}(T) \to P$. Then a diagonal lift as dotted in the diagram would imply that given unary operation $\alpha\in P(a; r)$ and $\beta\in P(b; r)$, there exists a binary operation $\gamma\in P(a,b;r)$ with $\gamma\circ_a c_a = \beta$ and $\gamma \circ_b c_b=\alpha$, where $c_a$ and $c_b$ are the unique constants of colours $a$ and $b$, respectively. For such a lift to exist, it is necessary that the map 
	\[
	P(a,b; r)\to P(a; r) \times P(b; r)
	\]
	is a trivial fibration. In other words, $P$ needs to be Reedy fibrant rather than just projectively fibrant. 
	\end{example}

	Thus, for the Boardman-Vogt resolution for closed operads, we shall have to work with the Reedy model structure. On the other hand, for closed trees we can get by with a somewhat smaller resolution, putting lengths only on some of the \emph{inner} edges. This gives rise to a functor
	\[
	\overline{W} \colon \Omega_{cl}\to \cOper
	\]
	which assigns to each closed tree $T$ a closed simplicial operad $\overline{W}(T)$. Its colours are again the edges of $T$. And for a sequence of edges $e_1, \ldots, e_n, e$ of $T$ , the space of operations $\overline{W}(T)(e_1, \ldots, e_n;e)$ is again empty unless these edges span a subtree of $T$ with leaves $e_1, \ldots, e_n$ and root $e$. In that case
	\begin{equation} \label{eq11-1}
	\overline{W}(T)(e_1, \ldots, e_n; e)=\prod_{d\in D(e_1, \ldots e_n; e)} \Delta[1]
	\end{equation}
	where $d$ ranges over the set $D(e_1, \ldots, e_n;e)$ of those inner edges of the tree \[T(e_1, \ldots, e_n; e)\] which have at least one of the leaves $e_1, \ldots, e_n$ above them. (We informally refer to the product over these edges $d$ as lengths assigned to these edges, as in \cite{bergermoerdijkBVresolution}.)
	
	For example, for the closed tree $T$ pictured on the left 
	
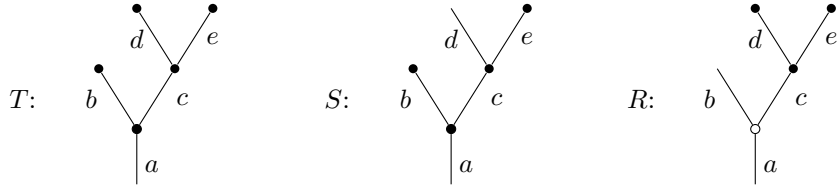
\[
	\begin{tikzpicture} 
	[level distance=8mm, 
	every node/.style={fill, circle, minimum size=.12cm, inner sep=0pt}, 
	level 1/.style={sibling distance=10mm}, 
	level 2/.style={sibling distance=10mm}, 
	level 3/.style={sibling distance=10mm}]
	
	\node(anchorR1)[style={color=white}] {} [grow'=up] 
	child {node(vertexR1)[draw] {} 
		child{node {}}
		child{node {}
			child{node {}}
			child{node {}}}
	};
	
	\node[style={color=white}, right=4cm of anchorR1] {} [grow'=up] 
	child {node(vertexR2)[draw] {} 
			child{node {}}
			child{node {}
				child
				child{node {}}}
	};
	
	\node[style={color=white}, right=8cm of anchorR1] {} [grow'=up] 
	child {node(vertexR3)[draw, fill=none] {} 
			child
	child{node {}
		child{node {}}
		child{node {}}}
	};
	
	\tikzstyle{every node}=[]
	
	\node at ($(vertexR1) + (-1.5cm,.4cm)$) {$T$:};
	\node at ($(vertexR1) + (0cm,1.2cm)$) {$d$};
	\node at ($(vertexR1) + (1cm,1.2cm)$) {$e$};
	\node at ($(vertexR1) + (-.6cm,.4cm)$) {$b$};
	\node at ($(vertexR1) + (.6cm,.4cm)$) {$c$};
	\node at ($(vertexR1) + (.2cm,-.5cm)$) {$a$};

	\node at ($(vertexR2) + (-1.5cm,.4cm)$) {$S$:};
	\node at ($(vertexR2) + (0cm,1.2cm)$) {$d$};
	\node at ($(vertexR2) + (1cm,1.2cm)$) {$e$};
	\node at ($(vertexR2) + (-.6cm,.4cm)$) {$b$};
	\node at ($(vertexR2) + (.6cm,.4cm)$) {$c$};
	\node at ($(vertexR2) + (.2cm,-.5cm)$) {$a$};

	\node at ($(vertexR3) + (-1.5cm,.4cm)$) {$R$:};
	\node at ($(vertexR3) + (0cm,1.2cm)$) {$d$};
	\node at ($(vertexR3) + (1cm,1.2cm)$) {$e$};
	\node at ($(vertexR3) + (-.6cm,.4cm)$) {$b$};
	\node at ($(vertexR3) + (.6cm,.4cm)$) {$c$};
	\node at ($(vertexR3) + (.2cm,-.5cm)$) {$a$};
	\end{tikzpicture} 
	\] 
	the space $\overline{W}(T)(d;a)$ corresponding to the subtree $S=T(d;a)$ is a copy of $\Delta[1]$, while the space $\overline{W}(T)(b; a)$ corresponding to the subtree $R$ is just a point. 
	
	Composition in this operad $\overline{W}(T)$ is again defined by grafting subtrees, assigning length $1$ to newly arising inner edges which require a length, and erasing lengths on edges that no longer have leaves above them in the grafted tree. 
	
	The construction of the operad $\overline{W}(T)$ is obviously functorial in $T$, and induces adjoint functors
	\[
	\xymatrix{
		\overline{w}_! : \cdSets \ar@<.5ex>[r] & \cOper : \overline{w}^* \ar@<.5ex>[l]  } 
	\]
	
	\begin{theorem} \label{teo11-4}
		This adjoint pair is a Quillen pair for the unital operadic model structure on the category $\cdSets$ of closed dendroidal sets and the Reedy model structure on the category $\cOper$ of closed (or unital) simplicial operads. 
	\end{theorem}

	\begin{proof}
		Let $T$ be a closed tree, and let $\partial_{cl}(T)$ be its closed boundary, as before. To show that the inclusion $\overline{w}_!(\partial_{cl} T) \rightarrowtail \overline{w}_! (T)$ is a Reedy cofibration, consider a lifting problem of the form 
		\[
		\xymatrix{
			\overline{w}_!(\partial_{cl} T) \ar[r]^\psi \ar[d] & Q \ar[d]^f \\   W(T) \ar@{.>}[ru]^\chi \ar[r]^\rho & P  }
		\]
		where $f$ is a trivial Reedy fibration between unital operads $P$ and $Q$. To define a diagonal lift $\chi$, consider a non-empty space of operations $\overline{W}(T)(e_1, \ldots, e_n;e)$ for edges $e_1, \ldots, e_n$ and $e$ in $T$. Now first notice that unless $e_1, \ldots, e_n$ enumerate all the maximal edges in $T$ (the ones immediately below the stumps) and $e$ is the root edge of $T$, the tree $T(e_1, \ldots, e_n; e)$ is a subtree of a closed face of $T$, so the space of operations $\overline{W}(T)(e_1, \ldots, e_n;e)$ is already contained in $\overline{w}_!(\partial_{cl} T)$. So assume $e_1, \ldots, e_n$ are all the maximal edges and $e$ is the root; in other words, $T(e_1, \ldots, e_n; e)=T^\circ$ in the notation used earlier. The map $\chi\colon \overline{W}(T)(e_1, \ldots, e_n; e) \to Q$ we are looking for is already prescribed by $\psi$ on the subspace of the product $\prod_d \Delta[1]$ for which one of the coordinates is $1$, because then the operation is a composition of operations occurring in $\overline{w}_!(\partial_{cl} T)$. It is also prescribed by $\psi$ if one of the coordinates, say the one for the edge $d$, is zero, because then it is an operation in the image of $\overline{w}_!(\partial_{d} T) \to \overline{w}_!(\partial_{cl} T) \to \overline{w}_!(T)$. Finally, the value of $\chi$ is provided by $\psi$ if we compose with one of the (unique) constants of colour $e_i$, because then the operation lies in $\overline{W}(\partial_{e_i} T)(e_1, \ldots, \hat{e}_i, \ldots, e_n; e)$. Thus, the required map $\chi$ is a solution to a lifting problem of the form 
		\[\xymatrix{
		\partial(\Delta[1]^D)  \ar[r] \ar@{>->}[d] & Q(\psi e_1, \ldots, \psi e_n; \psi e) \ar[d]\\ 
		\Delta[1]^D \ar[r] & P(\varphi e_1, \ldots, \varphi e_n; \varphi e)	 \times_{P^-(\varphi e_1, \ldots, \varphi e_n; \varphi e)} Q^-(\psi e_1, \ldots, \psi e_n; \psi e)
		}\]
		where $D=D(e_1, \ldots, e_n; e)$ as in (\ref{eq11-1}), and $P^-$ and $Q^-$ are as defined in Section \ref{sect3}. Such a solution indeed exists since $Q\to P$ is assumed to be a trivial Reedy fibration. 
		
		Next, we prove that $\overline{w}_!$ also preserves trivial cofibrations. For this, it suffices to show that it sends the two kinds of trivial cofibrations which detect fibrations between fibrant objects (called ``quasi-generating'' before) to trivial Reedy fibrations. 
		
		Consider first the inclusion of a very inner horn $\Lambda_{cl}^a(T)\rightarrowtail T$. Here the argument for the existence of a lift in a diagram of the form 
		\[
		\xymatrix{
			\overline{w}_!(\Lambda_{cl}^a(T)) \ar[r]^{\quad \psi} \ar[d] & Q \ar[d] \\   \overline{w}_!(T) \ar@{.>}[ru] \ar[r]^\rho & P  }
		\]
		where $Q\to P$ is a Reedy fibration is entirely similar to the previous argument, and reduces to the problem of finding a lift in  
		\[\xymatrix{
			\Lambda^a(\Delta[1]^D)  \ar[r] \ar@{>->}[d] & Q(\psi e_1, \ldots, \psi e_n; \psi e) \ar[d]\\ 
			\Delta[1]^D \ar[r] & P(\varphi e_1, \ldots, \varphi e_n; \varphi e)	 \times_{P^-(\varphi e_1, \ldots, \varphi e_n; \varphi e)} Q^-(\psi e_1, \ldots, \psi e_n; \psi e)
		}\]
		where $\Lambda^a(\Delta[1]^D)\subseteq \partial(\Delta[1]^D)$ is that part of the boundary consisting of those coordinates which either have length $1$ at $a$, or length $0$ or $1$ at some $d\neq a$. Such a lift exists because $\Lambda^a(\Delta[1]^D) \rightarrowtail \Delta[1]^D$ is a trivial cofibration of simplicial sets. 
		
		Next, and finally, consider an inclusion $\overline{ \{ 0 \} } \rightarrowtail \overline{J}$. We claim that $\overline{w}_! \overline{ \{ 0 \} } \to \overline{w}_! (\overline{J})$ is a trivial Reedy cofibration. Indeed, it is a Reedy cofibration since we already know that $\overline{w}_!$ preserves cofibrations, so it suffices to prove it is a weak equivalence. In other words, we need to prove that each of the spaces of operations 
		\[
		\overline{w}_!(\overline{J})(i,j), \quad i,j=0 \text{ or } 1,
		\]		
		is weakly contractible. For example, let us consider the case $i=0=j$. (The other cases are similar.) Then $\overline{w}_!(\overline{J})(0,0)$ is a colimit of cubes coming from lengths on edges in the tree with $2n+1$ edges: 
		\begin{equation} \label{eq11-2}
		\begin{tikzpicture} 
		[level distance=8mm, 
		every node/.style={fill, circle, minimum size=.2cm, inner sep=0pt}, 
		level 1/.style={sibling distance=10mm}, 
		level 2/.style={sibling distance=10mm}, 
		level 3/.style={sibling distance=10mm}]
		
		\node(anchorR1)[style={color=white}] {} [grow'=up] 
		child {node(vertexR1)[draw] {} 
			child{
				child{ edge from parent [draw=none]
					child{node {}
						child{node {}
							child{node {}
								child{node {}}}}}}
			}
		};
		
		\tikzstyle{every node}=[]
		
		\node at ($(vertexR1) + (-.3cm,4.4cm)$) {$0$};
		\node at ($(vertexR1) + (-.3cm,3.6cm)$) {$1$};
		\node at ($(vertexR1) + (.6cm,3.6cm)$) {$t_{1}$};
		\node at ($(vertexR1) + (-.3cm,2.8cm)$) {$0$};
		\node at ($(vertexR1) + (.6cm,2.8cm)$) {$t_{2}$};
		\node at ($(vertexR1) + (-.3cm,2cm)$) {$1$};
		\node at ($(vertexR1) + (.3cm,1.2cm)$) {$\vdots$};
		\node at ($(vertexR1) + (-.3cm,.4cm)$) {$1$};
		\node at ($(vertexR1) + (.6cm,.4cm)$) {$t_{2n-1}$};
		\node at ($(vertexR1) + (-.3cm,-.5cm)$) {$0$};
		\end{tikzpicture} 
		\end{equation}
		Such a tree has $2n-1$ very inner edges, so it is a cube $\Delta[1]^{2n-1}$. Thus 
		\[
		\overline{w}_!(\overline{J})(0,0) =\varinjlim A^{(n)}
		\]
		is a colimit of a sequence 
		\[
		A^{(1)} \to A^{(2)}\to A^{(3)} \to \ldots 
		\]
		where $A^{(n+1)}=A^{(n)} \cup \Delta[1]^{2n-1}$ and $A^{(1)}=pt$. These cubes are glued together as follows. Write $Z_n\subseteq \Delta[1]^{2n-1}$ for the simplicial subset of those coordinates $(t_1, \ldots, t_{2n-1})$ where at least one of the $t_i=0$. Define
		\[
		\alpha_n\colon Z_n \to A^{(n)}
		\]
		by mapping to a lower dimensional cube as follows: 
		\begin{itemize}
			\item[-] if $t_1=0$, (contract the upper two edges in (\ref{eq11-2}) and erase $t_2$, i.e.) map $(t_1, \ldots, t_{2n-1})$ to $(t_3, \ldots, t_{2n-1})$.
			\item[-] if $t_i=0$, $i>0$, (contract the ($i-1$)th and $i$th edge, and) map $(t_1, \ldots, t_{2n-1})$ to $(t_1, \ldots, t_{i-1}\vee t_{i+1}, \ldots, t_{2n-1})$.
			\item if $t_{2n-1}=0$, map $(t_1, \ldots, t_{2n-1})$ to $(t_1, \ldots, t_{2n-3})$.
		\end{itemize}
	Then $A^{(n+1)}$ is defined as the pushout 
	\[\xymatrix{
		Z_n  \ar[r]^{\alpha_n} \ar@{>->}[d] & A^{(n)} \ar@{>->}[d]\\ 
		\Delta[1]^{2n-1} \ar[r] & A^{(n+1)},
	}\]
	in particular $A^{(n)}\to A^{(n+1)}$ is a weak equivalence. 
	
	This shows that $\overline{w}_!(\overline{J})$ has a weakly contractible space of operations $0\to 0$. The case of the other spaces of operations being similar, this proves that $\overline{w}_! \overline{ \{ 0 \} } \to \overline{w}_! (\overline{J})$ is a weak equivalence. 
	
	This completes the proof of the theorem. 
	\end{proof}

\section{Open and closed dendroidal spaces} \label{sect12}

Let us write 
\[
\dSpaces = \sSets^{\Omega^{op}}
\]
for the category of dendroidal spaces. It carries an evident simplicial structure which we denote by $\underline{\times}$. Thus, for a dendroidal space $X$ and a simplicial set $M$, 
\[
(X\underline{\times}M) (T) = X(T)\times M, 
\]
by definition. This category of dendroidal spaces carries two equivalent model structures, the projective and the Reedy one. These model structures and the left Quillen equivalence between these are denoted 
\[
(\dSpaces)_P \xrightarrow{\sim} (\dSpaces)_R.
\]
The second one can be localized by the Reedy cofibration 
\begin{equation} \label{eq12-1}
\Lambda^e(T) \underline{\times} \Delta[n] \cup T\underline{\times} \partial\Delta[n] \rightarrowtail T\underline{\times} \Delta[n]
\end{equation}
for each inner edge $e$ in a tree $T$ and each $n\geqslant 0$, as well as 
\begin{equation} \label{eq12-2}
\{0\} \underline{\times} \Delta[n] \cup J \underline{\times} \partial\Delta[n] \rightarrowtail J\underline{\times} \Delta[n]
\end{equation}
for each $n\geqslant 0$. This results in a ``complete Segal'' model structure $(\dSpaces)_{RSC}$, whose fibrant objects are the Reedy fibrant objects satisfying the \emph{Segal condition} (meaning that they have the RLP with respect to the maps of type (\ref{eq12-1})) and the \emph{completeness} condition (RLP with respect to (\ref{eq12-2})). A similar localization of the projective model structure by (cofibrant replacements of) the maps in (\ref{eq12-1}) and (\ref{eq12-2}) results in a model category $(\dSpaces) _{PSC}$ and a left Quillen equivalence 
\[
(\dSpaces)_{PSC} \xrightarrow{\sim} (\dSpaces)_{RSC}.
\]
The adjoint functors $\xymatrix{d_! : \dSets \ar@<.5ex>[r] & \dSpaces : d^* \ar@<.5ex>[l]}$ sending a dendroidal set to the corresponding discrete dendroidal space and its right adjoint sending a dendroidal space to its dendroidal set of vertices, define a Quillen equivalence 
\begin{equation} \label{eq12-3}
\xymatrix{ d_! : \dSets \ar@<.5ex>[r] & (\dSpaces)_{RSC} : d^*, \ar@<.5ex>[l]  } 
\end{equation}
see \cite{cisinskimoerdijk2} for details. 

By slicing over the dendroidal set $O$ (or over the dendroidal space $d_!(O)$, respectively) we obtain similar model structures and Quillen equivalences for the category $\odSpaces=\sSets^{\Omega_o^{op}}$ of \emph{open dendroidal spaces},
\begin{equation} \label{eq12-4}
\xymatrix{  (\odSpaces)_{PSC} \ar@<.5ex>[r] & (\odSpaces)_{RSC}  \ar@<.5ex>[l]  \ar@<.5ex>[r] & \odSets \ar@<.5ex>[l]} 
\end{equation}
(left Quillen functors on top).

Exactly the same constructions and arguments apply to the category 
\[
\cdSpaces = \sSets^{\Omega_{cl}^{op}}
\]
of \emph{closed dendroidal spaces}, and the localizations by the closed analogues of (\ref{eq12-1}) and (\ref{eq12-2}); namely, maps of type (\ref{eq12-1}) for every very inner edge $e$ in a closed tree $T$, and for (\ref{eq12-2}) the closure $\overline{ \{ 0 \} } \rightarrowtail \overline{J}$ (i.e. $cl_!(\{ 0 \} ) \rightarrowtail cl_!(J)$) instead of $\{0 \} \rightarrowtail J$ (cf. Theorem \ref{teo9-2} (d)). We again refer to the localizations as the projective and Reedy complete Segal model structures, now on the category of closed dendroidal spaces. Analogous to (\ref{eq12-4}), there are Quillen equivalences
\begin{equation} \label{eq12-5}
\xymatrix{  (\cdSpaces)_{PSC} \ar@<.5ex>[r] & (\cdSpaces)_{RSC}  \ar@<.5ex>[l]  \ar@<.5ex>[r] & \cdSets, \ar@<.5ex>[l]} 
\end{equation}
where the latter is equipped with the unital operadic model structure of Theorem~\ref{teo9-1}.

The arguments for Proposition \ref{prop10-4} of Section \ref{sect10} show that there are Quillen pairs induced by the inclusion $u\colon \Omega_{cl} \hookrightarrow \Omega$ as in the upper two rows of the following commutative diagram (the lower row is the Quillen pair of Proposition \ref{prop10-4} and the columns are (\ref{eq12-4}) and (\ref{eq12-5})): 
\begin{equation} \label{diagram11-1}
\xymatrix{ u_! : (\cdSpaces)_{PSC} \ar@<.5ex>[r]  \ar@<.5ex>[d]  & (\odSpaces)_{PSC} : u^* \ar@<.5ex>[l]   \ar@<.5ex>[d] \\
u_! : (\cdSpaces)_{RSC} \ar@<.5ex>[r]  \ar@<.5ex>[u]^\sim  \ar@<.5ex>[d]   & (\odSpaces)_{RSC} : u^* \ar@<.5ex>[l]  \ar@<.5ex>[d]  \ar@<.5ex>[u]^\sim  \\
u_! : \cdSets \ar@<.5ex>[r]  \ar@<.5ex>[u]^\sim  & \odSets : u^* \ar@<.5ex>[l]  \ar@<.5ex>[u]^\sim 
} 
\end{equation}

The functor $cl_!$ does not preserve normal monomorphisms (cf. Remark \ref{rem10-3}), so does not induce a Quillen pair $\xymatrix{\dSpaces_R \ar@<.5ex>[r] & \cdSpaces_R  \ar@<.5ex>[l]}$, but it does so for the projective structure, 
\[
\xymatrix{cl_! : \dSpaces_P \ar@<.5ex>[r] & \cdSpaces_P : cl^*. \ar@<.5ex>[l]}
\]

\begin{lemma} \label{lemma12-1}
	This Quillen pair restricts to a Quillen pair 
	\[
	\xymatrix{cl_! : \dSpaces_{PSC} \ar@<.5ex>[r] & \cdSpaces_{PSC} : cl^*. \ar@<.5ex>[l]}
	\]
\end{lemma}

\begin{proof}
	We need to show that the localizing maps defining the passage from \linebreak $(\dSpaces)_P$~to $(\dSpaces)_{PSC}$ are sent to weak equivalences in $(\cdSpaces)_{PSC}$ by the functor $cl_!$. For the Segal condition, we observe that instead of localizing by inner horns, we might equally well localize by ``spines'' or ``Segal cores'', cf. \cite{cisinskimoerdijk2}. A cofibrant resolution of the Segal core $Sc(T)$ of a tree $T$ in the projective model structure $(\dSpaces)_P$ is its ``\v Cech nerve''
	\begin{equation} \label{eq12-6}
	\xymatrix{ Sc(T) & \ar@{->>}[l] \coprod_{v_0} C_{v_0}  & \coprod_{v_0, v_1}  C_{v_0} \cap C_{v_1} \ar@<.5ex>[l]\ar@<-.5ex>[l]  & \ar@<1ex>[l]\ar[l]\ar@<-1ex>[l] }
	\end{equation}
	where the $v_i$ range over the vertices of $T$ and $C_{v_i}\subseteq T$ denotes the corresponding corolla. The intersections $C_{v_0}\cap \ldots \cap C_{v_n}$ occurring here are either empty or copies of the unit tree $\eta$ (or corollas in case $v_0, \ldots, v_n$ are all the same). These types of intersections are preserved by the closure functor $cl_!$. So the image of this cofibrant resolution (\ref{eq12-6}) under the functor $cl_!$ is precisely the \v Cech type resolution of the closed spine $\cSpine(\overline{T}) \rightarrowtail \overline{T}$, which is a localizing map for $\cdSpaces_{PSC}$, cf. Theorem \ref{theorem7-6}. A similar argument applies to the localizing map $\{0\}\rightarrowtail J$. Indeed, $J=\bigcup L_n$ and $\overline{J} = \bigcup \overline{L}_n$, as in the proof of Proposition \ref{prop10-4}. The intersections $L_{n_1} \cap \ldots \cap L_{n_p}$ of this cover of $J$ are again of the form $L_m$ for a smaller $m$, and these intersections are preserved by the closure operator $cl_!$. This proves the lemma.
\end{proof}

Next, let us turn to the composite functor $h=cl \circ o\colon \Omega_o \to \Omega_{cl}$. First of all, the arguments of Theorem  \ref{teo10-6} show the following. 

\begin{lemma}
	The functor $h\colon \Omega_o \to \Omega_{cl}$ induces a Quillen pair rendering the following diagram commutative:
	\[ 
	\xymatrix{ 
		h^* : (\cdSpaces)_{RSC} \ar@<0.5ex>[r]   \ar@<0.5ex>[d]   & (\odSpaces)_{RSC} : h_* \ar@<0.5ex>[l]  \ar@<0.5ex>[d]    \\
		h^± : \cdSets \ar@<0.5ex>[r]  \ar@<0.5ex>[u]^\sim  & \odSets : h_* \ar@<0.5ex>[l]  \ar@<0.5ex>[u]^\sim 
	} \]
(left Quillen functors on the left and on top).
\end{lemma}

The functor $h^*$ is of course not left Quillen for the projective model structure. However, the functor $h_! \colon (\odSpaces)_P\to (\cdSpaces)_P$ is, and has the following property. 

\begin{lemma}
	The Quillen pair $\xymatrix{h_! : (\odSpaces)_P \ar@<.5ex>[r] & (\cdSpaces)_P : h^* \ar@<.5ex>[l]}$ restricts to a Quillen pair \[
	\xymatrix{h_! : (\odSpaces)_{PSC} \ar@<.5ex>[r] & (\cdSpaces)_{PSC} : h^* \ar@<.5ex>[l]}
	\]
\end{lemma}

\begin{proof}
	This follows immediately from the definition of $h$ as the composition $h=cl \circ o$, together with Lemma \ref{lemma12-1}.
\end{proof}

Let us observe some consequences of these lemmas. 

\begin{proposition}
	The functor $Lh^* \colon Ho(\cdSets) \to Ho(\odSets)$ induced by the Quillen pair of Theorem \ref{teo10-6} has both a left and a right adjoint.
\end{proposition}

\begin{proof}
	This follows from the commutativity of the square 
	\[
	\xymatrix{
		\cdSets \ar[r]^{d_! \quad } \ar[d]^{h^*} & (\cdSpaces)_{RSC} \ar[d]^{h^*}	&  (\cdSpaces)_{PSC}\ar[d]^{h^*} \ar[l]_{id}\\
		\odSets \ar[r]^{d_! \quad} & (\odSpaces)_{RSC} & (\odSpaces)_{PSC} \ar[l]_{id}
	} \]
	in which the horizontal functors are left Quillen equivalences. The functor $h^*$ in the middle is left Quillen, while the same functor on the right is right Quillen. These two functors act in the same way on objects which are both Reedy cofibrant and Reedy fibrant (and hence also projectively fibrant). 
\end{proof}

\begin{proposition} \label{prop12-5}
	The functor $Lh^* \colon Ho(\cdSets) \to Ho(\odSets)$ detects isomorphisms.
\end{proposition}

\begin{proof}
	First, observe that the right Quillen functor \[
	h^*\colon (\cdSpaces)_{PSC}\to (\odSpaces)_{PSC}
	\]
	evidently detects weak equivalences between fibrant objects. Indeed, for such an object $X$, the value $h^*(X)$ of the functor is defined on an open tree $S$ by 
	\[
	h^*(X)_S = X_{\overline{S}},
	\]
	so $h^*$ already detects weak equivalences as a functor $(\cdSpaces)_P\to (\odSpaces)_P$. But the weak equivalences between fibrant objects in the localized model structure remain the same, showing that $h^*$ also detects weak equivalences between fibrant objects as a functor $(\cdSpaces)_{PSC}\to (\odSpaces)_{PSC}$. 
	
	Next, note that the statement in the proposition is equivalent to the assertion that 
	\[
	Lh^*\colon Ho(\cdSpaces_{RSC}) \to Ho(\odSpaces_{RSC})
	\]
	detects weak equivalences. But on fibrant and cofibrant objects in $(\cdSpaces)_{RSC}$, $Lh^*$ is represented by the same functor $h^*$ as \[Rh^* \colon Ho(\cdSpaces_{PSC}) \to Ho(\odSpaces_{PSC})\] is. This functor detects isomorphisms, as observed at the start of the proof. 
\end{proof}

\section{Closed dendroidal sets are equivalent to unital operads} \label{sect13}

The goal of this section is to prove the following ``rectification'' theorem. 

\begin{theorem} \label{teo13-1}
	The Quillen pair of Theorem \ref{teo11-4} \[
	\xymatrix{\overline{w}_! : \cdSets \ar@<.5ex>[r] & \cOper : \overline{w}^* \ar@<.5ex>[l]}
	\]
	is a Quillen equivalence.
\end{theorem}

Recall from Section \ref{sect11} that the model structures involved are the unital operadic model structure on the category $\cdSets$ of closed dendroidal sets and the Reedy model structure on the category $\cOper$ of closed or unital simplicial operads. 

One of the main results of \cite{cisinskimoerdijk3} is a similar Quillen equivalence 
\begin{equation} \label{eq13-1}
	\xymatrix{ w_! : \dSets \ar@<.5ex>[r] & \Oper : w^* \ar@<.5ex>[l]}
\end{equation}
between the operadic model structure on all dendroidal sets and the transferred or projective model structure on simplicial operads. By slicing over suitable objects (the dendroidal set $O$ and the operad for which $O$ is the value under $w^*$, i.e. the non-unital commutative operad $Comm^+$), this Quillen equivalence induces a similar Quillen equivalence which we state explicitly as a lemma for easy reference. 

\begin{lemma} \label{lemma13-2}
	The Quillen equivalence (\ref{eq13-1}) restricts to a Quillen equivalence 
	\[
		\xymatrix{{\accentset{\circ}{w}_!} : \odSets \ar@<.5ex>[r] & \oOper : {\accentset{\circ}{w}^*} \ar@<.5ex>[l]}
	\]
	between open dendroidal sets and open simplicial operads. 
\end{lemma}

\begin{proof}
	As said, this follows immediately from the Quillen equivalence (\ref{eq13-1}) established in \cite{cisinskimoerdijk3}.
\end{proof}

Now consider the following diagram of model categories and Quillen adjunctions: 
\begin{equation} \label{eq13-2}
\xymatrix{ 
	 \cdSets \ar@<0.5ex>[r]^{\overline{w}_!}   \ar@<0.5ex>[d]^{h^*}   & \cOper \ar@<0.5ex>[l]^{\overline{w}^*}  \ar@<0.5ex>[d]^{g^*}    \\
	 \odSets \ar@<0.5ex>[r]^{{\accentset{\circ}{w}_!}}  \ar@<0.5ex>[u]^{h_*} & \oOper. \ar@<0.5ex>[l]^{{\accentset{\circ}{w}^*}}  \ar@<0.5ex>[u]^{g_*} 
} \end{equation}

\begin{lemma} \label{lemma13-3}
	The diagram commutes up to isomorphism. 
\end{lemma}

\begin{remark}
	For what follows it is actually enough to know that there is a natural weak equivalence ${\accentset{\circ}{w}_!} h^* X \xrightarrow{\sim} g^*\overline{w}_! X$ for each cofibrant object $X$. Since the category $\cdSets$ is left proper, a standard induction over cofibrant objects reduces the problem again to representables. And since the functors involved are left Quillen, the trivial cofibrations $\cSpine(T) \rightarrowtail T$ of Theorem \ref{theorem7-6} then reduce the problem to showing that the map $\theta_{\overline{C_n}} \colon {{\accentset{\circ}{w}_!}} h^* \overline{C_n} \to g^*\overline{w}_!\overline{C_n}$ is a weak equivalence for each closed corolla $\overline{C_n}$, which is obvious. 
\end{remark}

\begin{proof}[Proof (of Lemma \ref{lemma13-3})]
	We will exhibit a natural isomorphism $\theta_X \colon {{\accentset{\circ}{w}_!}} h^* X \to g^*\overline{w}_!X$ for each closed dendroidal set $X$. Since the left adjoint functors involved preserve colimits, it suffices to define such an isomorphism for representable $X$, i.e. 
	\[
	\theta_T \colon {{\accentset{\circ}{w}_!}} h^* T \to g^*\overline{w}_!T
	\]
	for each closed tree $T$. Recall that 
	\[
	h^*(T) = \bigcup_A T[A]^\circ
	\]
	where $A$ ranges over upwards closed subforests of the tree $T$, and $T[A]$ is the tree obtained by contracting the edges in $A$ while $(-)^\circ$ is the operation of chopping off the stumps from a closed tree. For each such $A$, there is an evident map of operads 
	\[
	W(T[A]^\circ) \to \overline{W}(T)
	\]
	and together these are easily seen to induce the desired isomorphism. Indeed, suppose $e_1, \ldots, e_n$ and $e$ are edges in $T$ spanning a subtree $T(e_1, \ldots, e_n; e)$ with leaves $e_1,\ldots, e_n$ and root $e$. Then operations in $\overline{W}(T)(e_1, \ldots, e_n; e)$ can be represented as operations in $W(T[A]^\circ)$ for a maximal and canonical such $A$, viz. the forest consisting of all the edges in $T$ which are not in the tree $T(e_1, \ldots, e_n; e)$, nor on the path from $e$ down to the root of $T$. 
\end{proof}

The next lemma states another commutation property of the diagram (\ref{eq13-2}). 

\begin{lemma}[``Projection formula'']
	There is a natural isomorphism of derived functors \[
	Lh^*R\overline{w}^* \simeq R{\accentset{\circ}{w}^*} Lg^*. 
	\]
\end{lemma}

\begin{proof}
	Consider first the left Quillen equivalence $d_! \colon \dSets \xrightarrow{\sim} \dSpaces$ of Section \ref{sect12} and its variants for open and closed dendroidal sets and spaces. Recall from \cite{cisinskimoerdijk1} that for a fibrant dendroidal set $X$, a fibrant replacement of $d_!(X)$ in $(\dSpaces)_{RSC}$ (a ``completion'') is defined as the dendroidal space $d_!(X)^\wedge$ whose value at a tree $T$ is the simplicial set (Kan complex)
	\[
	\underline{Map}(T, X)
	\]
	where ``Map'' refers to the model structure on dendroidal sets. (An example of a cofibrant cosimplicial resolution of $T$ by which to compute this ``Map'' is $n\mapsto T\otimes |\Delta[n]|_J$, where $|-|_J$ refers to geometric realization relative to the interval $J$.) Similar descriptions of completion apply to $\cdSets$ and $\odSets$. 
	
	Now, to prove the lemma, consider a fibrant and $\Sigma$-free unital operad $P$ and its open part $g^*P$. Then a complete Segal model $(d_!{\accentset{\circ}{w}^*}g^*(P))^\wedge$ for the open dendroidal set ${\accentset{\circ}{w}^*}g^*(P)$ can be described for each open tree $S$ by
	\begin{align*}
	(d_!{\accentset{\circ}{w}^*}g^*(P))^\wedge(S) & = Map_{\odSets} (S, {\accentset{\circ}{w}^*}g^*(P)) \\
	& = Map(g_!{\accentset{\circ}{w}_!}(S), P),
	\end{align*}
	the latter ``Map'' referring to the model structure on closed operads. 
	
	On the other hand, a complete Segal model $(d_!w^*(P))^\wedge$ for the closed dendroidal set $\overline{w}^*(P)$ is described for each closed tree $T$ by 
	\begin{align*}
	(d_!w^*(P))^\wedge  (T) &  = Map_{\cdSets} (T, \overline{w}^* P) \\
	& = Map(\overline{W}(T), P)
	\end{align*}
	the latter ``Map'' referring to closed operads again. This object is normal if $P$ is $\Sigma$-free, so a model for $Lh^*R\overline{w}^*(P)$ in $(\odSpaces)_{RSC}$ is defined by the functor 
	\[
	S\mapsto Map(\overline{W}(\overline{S}), P).
	\]
	(This model is again ``complete Segal'', but for the \emph{projective} model structure on $\odSpaces$, cf. Section \ref{sect12}.)
	
	The proof of the lemma now simply consists of the observation that for each open tree $S$ there is a natural weak equivalence \[ g_! W(S) \to \overline{W}(\overline{S}).\]
	Indeed, an operation in $g_!W(S)(e_1, \ldots, e_n; e)$ is given by operation in \[W(S)(e_1, \ldots, e_n, f_1, \ldots, f_m; e)\] composed with the substitution of constants for $f_1, \ldots, f_m$. This is represented by an operation in $\overline{W}(\overline{S})$ for the tree $\overline{S}(e_1, \ldots, e_n; e)$ given by $S(e_1, \ldots, e_n, f_1, \ldots, f_m; e)$ with the closed trees $\overline{S}(f_1), \ldots, \overline{S}(f_n)$ with roots $f_1, \ldots, f_k$, respectively, grafted on top of it. This defines a map $g_!W(S) \to \overline{W}(\overline{S})$. The fact that this map is a weak equivalence is obvious form the commutativity of 
	\[
	\xymatrix{
	g_! W(S) \ar[r] \ar[d]^\sim & \overline{W}(\overline{S}) \ar[d]^\sim \\ 
	g_!(S) \ar@{=}[r] & \overline{S}}
	\]
	where in the lower row, $S$ and $\overline{S}$ are viewed as (discrete) operads. 
\end{proof}

The theorem stated at the beginning of the section now follows easily. 

\begin{proof}[Proof of Theorem.] First of all, let us observe that $\overline{w}^*$ detects weak equivalences between fibrant objects. Indeed, this is clear from the projection formula of the last lemma, together with the fact that both $g^*$ and ${\accentset{\circ}{w}^*}$ detect weak equivalences, cf. Proposition \ref{prop4-1} and Lemma \ref{lemma13-2}.
	
Consider then the derived units and counits 
\[\eta\colon X \to R\overline{w}^* L\overline{w}_! X \quad \text{and} \quad \varepsilon \colon L\overline{w}_! R\overline{w}^*P \to P\] 
for a closed dendroidal set $X$ and a closed operad $P$. By the triangular identities and the fact that $R\overline{w}^*$ detects isomorphisms as just observed, it follows that if the derived unit is an isomorphism then so is the derived counit. Thus, it suffices to prove for each fibrant and cofibrant closed dendroidal set $X$ that the unit $\eta\colon X \to R\overline{w}^* (\overline{w}_! (X))$ is a weak equivalence. In fact, by Proposition \ref{prop12-5} it suffices to prove that 
\[
Lh^*(\eta) \colon Lh^*(X)  \to Lh^* R\overline{w}^* (\overline{w}_! X)
\]
is a weak equivalence. Using the projection formula again, this map can be identified with a map 
\[
Lh^*(X)  \to R{\accentset{\circ}{w}^*} Lg^* (\overline{w}_! X)= R{\accentset{\circ}{w}^*} ({\accentset{\circ}{w}_!}h^* X)
\]
(the last identity by Lemma \ref{lemma13-3}). Running through the definitions, one identifies this map with the derived unit at $h^*(X)$ of the adjunction between ${\accentset{\circ}{w}_!}$ and ${\accentset{\circ}{w}^*}$.
The unit is a weak equivalence by Lemma \ref{lemma13-2}. This completes the proof of the theorem. 
\end{proof}

\section{Weakly unital operads} \label{sect14}

In this final section we will briefly discuss the property of an operad having a contractible space of constants for each colour, and the corresponding property of dendroidal spaces. We begin with the latter. 

Consider again the projective complete Segal model structure on the category of (all) dendroidal spaces, and the Quillen pair
\begin{equation} \label{eq14-1}
\xymatrix{cl_! : (\dSpaces)_{PSC} \ar@<.5ex>[r] & (\cdSpaces)_{PSC} : cl^* \ar@<.5ex>[l]}
\end{equation}
of Lemma \ref{lemma12-1} above. Write $(\dSpaces)_{PSCU}$ for the left Bousfield localization by the map $\eta\to \overline{\eta}$. Thus, a fibrant object in $(\dSpaces)_{PSC}$ is local (i.e., fibrant in $(\dSpaces)_{PSCU}$) precisely when $X(\overline{\eta}) \to X(\eta)$ is a (weak) homotopy equivalence between Kan complexes. This localization can be seen as a push forward of a similar localization of the operadic model structure on $\dSets$ which we denote by $\dSets_U$. Moreover, the closure functor $cl_!$ maps $\eta \to \overline{\eta}$ to an isomorphism, so clearly the Quillen pair (\ref{eq14-1}) factors through $(\dSpaces)_{PSCU}$. All put together, we obtain a diagram of left Quillen functors 
\[
\xymatrix{
	\dSets \ar[r]^{d_! \quad }_{\sim \quad} \ar[d] & (\dSpaces)_{RSC} \ar[d]	&  (\dSpaces)_{PSC}\ar[d] \ar[l]_{\sim} \ar[r]^{cl_!} & (\cdSpaces)_{PSC}\\
	\dSets_U \ar[r]^{d_! \quad}_{\sim \quad}  & (\dSpaces)_{RSCU} & (\dSpaces)_{PSCU} \ar[l]_{\sim} \ar@{.>}[ru] &
} \]
in which all four horizontal functors in the two squares are Quillen equivalent and the vertical maps are localizations. 

\begin{proposition}
	The Quillen pair (\ref{eq14-1}) induces a Quillen equivalence 
	\[
	\xymatrix{cl_! : (\dSpaces)_{PSCU} \ar@<.5ex>[r] & (\cdSpaces)_{PSC} : cl^* \ar@<.5ex>[l]}
	\]
\end{proposition}

\begin{proof}
	First of all, as $cl \colon \Omega \to \Omega_{cl}$ is surjective on objects, the functor \[cl^* \colon (\cdSpaces)_P \to (\dSpaces)_P\] clearly detects weak equivalences between arbitrary objects. It follows that the right Quillen functor $cl^*\colon (\cdSpaces)_{PSC} \to (\dSpaces)_{PSCU}$ detects weak equivalences between fibrant objects. 
	
	It thus suffices to prove that the derived unit $X\to Rcl^* Lcl_!(X)$ is a weak equivalence for each (cofibrant) object $X$ in $(\dSpaces)_{PSCU}$. But $cl^*=u_!$ is also left Quillen (cf. the diagram (\ref{diagram11-1})), hence $cl^*$ preserves weak equivalences between cofibrant objects. So it suffices to prove that the non-derived unit $X\to cl^*cl_!(X)=u_!cl_!(X)$ is a weak equivalence in $(\dSpaces)_{PSCU}$ for each cofibrant object $X$. 
	
	Since the functors $u_!$ and $cl_!$ involved preserve colimits and the model categories are left proper, we can use induction on cofibrant objects, and reduce the problem to the case where $X$ is representable. In other words, we have to show that for each tree $S$ the unit $S\to cl^* cl_! (S)$ is a weak equivalence in $(\dSpaces)_{PSCU}$. But this unit is the map $S\to \overline{S}$, and the following lemma completes the proof. 
\end{proof}

\begin{lemma}
	For any tree $S$, the inclusion $S\to \overline{S}$ is a trivial cofibration in $\dSets_U$ (and hence a weak equivalence in $\cdSpaces_{PSCU}$ and $\cdSpaces_{RSCU}$). 
\end{lemma}

\begin{proof}
	Let $\lambda S$ be the set of leaves of $S$, and consider the pushout
	\[
	\xymatrix{ \coprod_{\ell \in \lambda S} \eta \ar[r] \ar@{ >->}[d] & S \ar@{ >->}[d] \\
	\coprod_{\ell \in \lambda S} \overline{\eta} \ar[r]  & S'. }
	\]
	Then $S\to S'$ is a pushout of a trivial cofibration in $\dSets_U$, hence itself a trivial cofibration. Moreover, the map $S' \to \overline{S}$ is a composition of grafting morphisms, grafting a copy of $\overline{\eta}$ onto each leaf of $S$. So $S'\to \overline{S}$ is inner anodyne (cf. \cite{moerdijkweiss}, \cite{HMbook}).
\end{proof}

\begin{corollary} \label{cor14-3}
	There is a zigzag of Quillen equivalences 
	\[
	\dSets_U \simeq \cdSets,
	\]
	between the localization by $\eta\to \overline{\eta}$ of the operadic model structure on $\dSets$ and the unital operadic model structure on $\cdSets$. 
\end{corollary}

To conclude this paper, we briefly consider the effect of the localization $\dSets_U$ on the equivalence (\ref{eq13-1}) (from \cite{cisinskimoerdijk3}) between dendroidal sets and simplicial operads. To this end, we introduce the following terminology. 

\begin{definition}
	A simplicial operad $P$ is called \emph{weakly unital} (or \emph{weakly closed}) if for each colour $c$ in $P$, the space $P(-; c)$ of constants of colour $c$ is weakly contractible. We write
	\[Ho(\Oper_{WU}) \subseteq Ho(\Oper)\]	
	for the full subcategory spanned by these weakly unital operads. 
\end{definition}

\begin{proposition}
	The Quillen equivalence $\xymatrix{ w_! : \dSets \ar@<.5ex>[r] & \Oper : w^* \ar@<.5ex>[l]}$ of \cite{cisinskimoerdijk3} induces an equivalence of categories 
	\[
	Ho(\dSets_U) \simeq Ho(\Oper_{WU}).
	\]
\end{proposition}
	
\begin{proof}
	Recall from \cite{cisinskimoerdijk3} that $w_!\colon \dSets \to \Oper$ was proved to be a Quillen equivalence through a homotopy commutative diagram of left Quillen equivalences 
	\begin{equation} \label{eq14-2}
	\xymatrix{ (\mathbf{Segal preoperads})_{tame} \ar[rr]^{\overline{\tau}_d} \ar[d]^{id_!} & & \Oper \\ 
	(\mathbf{Segal preoperads})_{Reedy} \ar[rd]^{\gamma^*} & & \dSets \ar[u]^{w_!} \ar[ld]^{d_!} \\ 
	& (\dSpaces)_{RSC}. &}
	\end{equation}
	Here Segal preoperads are dendroidal spaces $X$ with a discrete space $X(\eta)$ of objects, and $\gamma^*$ is simply the inclusion. The functor $\overline{\tau}_d$ in the diagram is the left adjoint of the levelwise nerve functor, denoted $\overline{N}_d \colon \Oper \to (\mathbf{Segalpreoperads})$. A \emph{Segal operad} is such a Segal preoperad satisfying a Segal condition. The fibrant objects in the two model structures on Segal preoperads are all Segal operads. A key property of Segal operads is that a map is a weak equivalence iff it is fullly faithful and essentially surjective on colours. Now consider a fibrant simplicial operad $P$. Then $P$ is weakly unital iff $\overline{N}_d(P)(-; c)$ is (weakly) contractible for each colour $c$, i.e. iff $\overline{N}_d(P)(\overline{\eta})\to \overline{N}_d(P)(\eta)$ is a trivial fibration. Of course $\overline{N}_d(P)(\eta)$ is simply the discrete space $C$ of colours of $P$. Let $\hat{P} \twoheadrightarrow \overline{N}_d (P)$ be a tamely cofibrant resolution. This is a weak equivalence between Segal operads, so $\hat{P}(\overline{\eta}) \to C$ is again a trivial fibration because weak equivalences between Segal operads are fully faithful, as said. Then $Lid_! \overline{N}_d (P) = \hat{P}$ is a Reedy cofibrant Segal operad with the same property, and hence $L\gamma^* Lid_!(\overline{N}_d(P))=\hat{P}$ as a dendroidal space still has this property. This means that $\hat{P}$ as a dendroidal space is local with respect to the $RSCU$-localization of dendroidal spaces. This sequence of implications can obviously be reversed, showing that $P$ is weakly unital iff $L\gamma^* Lid_! R\overline{N}_d(P)$ is local in this sense. Combining this with the left Quillen equivalence $cl_!$ in the commutative diagram (\ref{eq14-2}) identifies the image of $Ho(\dSets_U) \subseteq Ho(\dSets)$ under $Ho(\dSets) \xrightarrow{Lw_!} Ho(\Oper)$ with $Ho(\Oper_{WU})$ proving the result. 
\end{proof}	

\begin{corollary}
	There is an equivalence $Ho(\Oper_{WU})\simeq Ho(\cOper)$ between the homotopy categories of closed and of weakly unital operads. 
\end{corollary}
	
	Chasing through the functors involved, one easily checks that this equivalence is induced by the inclusion $\cOper \to \Oper$.

\bibliographystyle{plain} 
\bibliography{biblio}

\end{document}